\documentclass[10pt,a4paper]{amsart}
\usepackage[margin=1in]{geometry}
\usepackage{amsmath}
\usepackage{amsfonts}
\usepackage{amssymb}
\usepackage{graphicx}
\usepackage{amsthm}
\usepackage{enumitem}
\usepackage{amscd}
\usepackage{mathrsfs}
\usepackage{mathtools}
\usepackage{tikz-cd}
\usepackage{eucal}
\usepackage[colorlinks = true, linkcolor = blue, bookmarks=true]{hyperref}

\usepackage[
backend=biber,
style=alphabetic,
sorting=nyt,
maxbibnames=99,
abbreviate=false, 
isbn=false,
]{biblatex}
\author{Zhiyuan Chen}
\address{Department of Mathematics, Princeton University, Princeton, NJ, 08544-1000, USA}
\email{zc5426@princeton.edu}

\title{Stable degeneration of families of klt singularities with constant local volume}

\theoremstyle{plain}
\newtheorem{thm}{Theorem}[section]
\newtheorem*{thm*}{Theorem}

\newtheorem{lem}[thm]{Lemma}

\newtheorem{cor}[thm]{Corollary}

\theoremstyle{definition}
\newtheorem{defn}[thm]{Definition}

\newtheorem{exmp}[thm]{Example}
\newtheorem{prg}[thm]{}
\newtheorem{rem}[thm]{Remark}

\let\originalleft\left
\let\originalright\right
\renewcommand{\left}{\mathopen{}\mathclose\bgroup\originalleft}
\renewcommand{\right}{\aftergroup\egroup\originalright}

\renewcommand{\subset}{\subseteq}

\begin{document}

\begin{abstract}
    For a klt singularity, C. Xu and Z. Zhuang \cite{XZ_stable_deg} proved the associated graded algebra of a minimizing valuation of the normalized volume function is finitely generated, finishing the proof of the stable degeneration conjecture proposed by C. Li and C. Xu. We prove a family version of the stable degeneration: for a locally stable family of klt singularities with constant local volume, the ideal sequences of the minimizing valuations for the normalized volume function form families of ideals with flat cosupport, which induce a degeneration to a locally stable family of K-semistable log Fano cone singularities. In the proof, we give a method to construct families of Koll\'ar models, which are a crucial tool introduced by Xu--Zhuang to prove finite generation for valuations of higher rational rank. 
\end{abstract}

\maketitle

\tableofcontents

\section{Introduction}

Along with recent developments of the algebraic theory of K-stability of Fano varieties, an analogous K-stability theory for klt singularities---the local counterparts of Fano varieties that are fundamental in the MMP (minimal model program)---is studied, through the normalized volume function introduced in \cite{Li_vol}. For a klt singularity $x \in (X, \Delta)$, the normalized volume of a real valuation $v \in \mathrm{Val}_{X,x}$ centered at $x$ is 
\begin{equation*}
    \widehat{\mathrm{vol}}_{X, \Delta}(v) \coloneqq A_{X,\Delta}(v)^{\dim X} \cdot \mathrm{vol}_X(v), 
\end{equation*}
where $A_{X,\Delta}(v)$ denotes the log discrepancy, and $\mathrm{vol}_X(v)$ denotes the volume introduced in \cite{ELS}. The infimum $\widehat{\mathrm{vol}}(x; X, \Delta) \coloneqq \inf_v \widehat{\mathrm{vol}}_{X, \Delta}(v)$ is called the local volume of the klt singularity $x \in (X, \Delta)$. 

The local K-stability theory is particularly interested in the minimizer of $\widehat{\mathrm{vol}}$, that is, a real valuation $v^{\mathrm{m}} \in \mathrm{Val}_{X,x}$ such that $\widehat{\mathrm{vol}}(x; X, \Delta) = \widehat{\mathrm{vol}}_{X, \Delta}(v^{\mathrm{m}})$, as highlighted in the \emph{Stable Degeneration Conjecture} proposed in \cite[Conj.\ 7.1]{Li_vol} and \cite[Conj.\ 1.2]{LX_stability}: There exists a minimizer $v^{\mathrm{m}}$ of $\widehat{\mathrm{vol}}_{X, \Delta}$, which is unique up to scaling and is quasi-monomial; the associated graded algebra $R_0 = \mathrm{gr}_{v^{\mathrm{m}}}(\mathscr{O}_{X,x})$ is finitely generated, and induces a klt degeneration $(X_0 = \mathop{\mathrm{Spec}}(R_0), \Delta_0)$, which is a K-semistable log Fano cone singularity under the natural torus action, and is the unique such K-semistable degeneration. 

The Stable Degeneration Conjecture has been proved is a series of works: The existence, uniqueness, and quasi-monomial property of a minimizer are proved in \cite{Blu_exist}, \cite{XZ_unique}, \cite{Xu_qm}; the K-semistability of the degeneration is proved in \cite{LX_stability} assuming finite generation; and finally, the finite generation is proved in \cite{XZ_stable_deg}. See \cite{LLX_guide} and \cite{Zhuang_survey} for surveys and more comprehensive references. 

The minimizing valuation $v^{\mathrm{m}}$, or the K-semistable degeneration, is associated with a klt singularity canonically. Then we ask whether the K-semistable degenerations for a family of klt singularities form a family. We consider locally stable families in the sense of \cite{Kol_fam}. As the local volume function $\widehat{\mathrm{vol}}$ is constructible and lower semi-continuous by \cite{Xu_qm} and \cite{BL_lsc}, we consider families of klt singularities with constant local volume. 

Instead of the minimizing valuation $v^{\mathrm{m}} \in \mathrm{Val}_{X,x}$ itself, we consider the ideal sequence $\mathfrak{a}_{\bullet} = \{\mathfrak{a}_{\lambda}\}_{\lambda \in \mathbb{R}_{\geq 0}}$ associated with it, where $\mathfrak{a}_{\lambda} \coloneqq \{ f : v^{\mathrm{m}}(f) \geq \lambda \} \subset \mathscr{O}_{X,x}$. Note that the ideal sequence is used to define the volume and the associated graded algebra. We will prove the following: 

\begin{thm}[Theorem \ref{K-semistable_degen_over_semi-normal}] \label{main_thm_1}
    Let $S$ be a semi-normal scheme essentially of finite type over a field of characteristic zero. Let $\pi \colon (X, \Delta) \to S$ be a locally stable family of klt pairs, with a section $x \colon S \to X$ of $\pi$, such that
    \begin{equation*}
        s \mapsto \widehat{\mathrm{vol}}(x_{s}; X_{s}, \Delta_{s})
    \end{equation*}
    is a locally constant function on $S$, where $(X_s, \Delta_s)$ is the fiber over a point $s \in S$, and $x_s = x(s) \in X_s$.\footnote{We will say $\pi \colon (X, \Delta) \to S$ with $x \in X(S)$, meaning an $S$-point of the $S$-scheme $X$, is a locally stable family of klt singularities. See Section \ref{sec: preliminaries} for our notations.} Suppose $v^{\mathrm{m}}_s \in \mathrm{Val}_{X_s, x_s}$ is a minimizer of the normalized volume function for $x_s \in (X_s, \Delta_s)$, scaled such that $A_{X_s, \Delta_s}(v^{\mathrm{m}}_s) = 1$ for all $s \in S$. Then there is an ideal sequence $\mathfrak{a}_{\bullet} \subset \mathscr{O}_X$ cosupported at $x(S) \subset X$ such that the following hold: 
    \begin{enumerate}[label=\emph{(\arabic*)}, nosep]
        \item $\mathfrak{a}_{\lambda}/\mathfrak{a}_{>\lambda}$ is flat over $S$ for all $\lambda \geq 0$. 
        \item For every $s \in S$, $\mathfrak{a}_{s, \bullet} \coloneqq \{ \mathfrak{a}_{\lambda} \mathscr{O}_{X_s} \}_{\lambda}$ is the ideal sequence on $X_s$ associated with $v_s^{\mathrm{m}}$. 
        \item Let $X_0 = \mathop{\mathrm{Spec}_S}(\mathrm{gr}_{\mathfrak{a}}(\mathscr{O}_X))$, where 
        \begin{equation*}
            \mathrm{gr}_{\mathfrak{a}}(\mathscr{O}_X) \coloneqq \bigoplus_{\lambda \geq 0} \mathfrak{a}_{\lambda}/\mathfrak{a}_{>\lambda}, 
        \end{equation*}
        then the canonical morphism $\pi_0 \colon X_0 \to S$ is flat, of finite type, with normal and geometrically integral fibers. The grading induces an action of a torus $\mathbb{T} \simeq (\mathbb{G}_{\mathrm{m},S})^r$ on $X_0$, whose fixed locus is the image of a section $x_0 \colon S \to X$. 
        \item There exists a $\mathbb{T}$-invariant effective $\mathbb{Q}$-divisor $\Delta_0$ on $X_0$ such that $\pi_0 \colon (X_0, \Delta_0) \to S$ is a locally stable family of K-semistable log Fano cone singularities, and $(X_{0,s}, \Delta_{0,s})$ is the degeneration of $x_s \in (X_s, \Delta_s)$ induced by $v_s^{\mathrm{m}}$ for every $s \in S$. 
    \end{enumerate}
\end{thm}

Our proof uses the theory of \emph{Koll\'ar models} developed in \cite{XZ_stable_deg}, which is a higher rank analogue of Koll\'ar components introduced in \cite{Xu_KP}. In \cite{XZ_stable_deg}, it is shown that the minimizing valuation $v^{\mathrm{m}}$ occurs on a Koll\'ar model, and every such quasi-monomial valuation has a finitely generated associated graded algebra which induces a klt degeneration. We will generalize these results to families of klt singularities. The key is to construct a family of Koll\'ar models that accommodates the minimizing valuation on each fiber, with the same coordinates when the dual complexes of different fibers are identified: 

\begin{thm}[Theorem \ref{construct_Kol_model_of_minimizer_over_DVR}] \label{main_thm_2}
    Let $S = \mathop{\mathrm{Spec}}(A)$, where $A$ is a DVR essentially of finite type over a field of characteristic zero, with the generic point $\eta \in S$ and the closed point $s \in S$. Let $\pi \colon (X, \Delta) \to S$ with $x \in X(S)$ be a locally stable family of klt singularities such that
    \begin{equation*}
        \widehat{\mathrm{vol}}(x_{\eta}; X_{\eta}, \Delta_{\eta}) = \widehat{\mathrm{vol}}(x_{s}; X_{s}, \Delta_{s}). 
    \end{equation*}
    Suppose $v^{\mathrm{m}}_{\eta} \in \mathrm{Val}_{X_{\eta}, x_{\eta}}$ and $v^{\mathrm{m}}_s \in \mathrm{Val}_{X_s, x_s}$ are minimizers of the normalized volume for $x_{\eta} \in (X_{\eta}, \Delta_{\eta})$ and $x_s \in (X_s, \Delta_s)$, respectively, scaled such that $A_{X_{\eta}, \Delta_{\eta}}(v^{\mathrm{m}}_{\eta}) = A_{X_s, \Delta_s}(v^{\mathrm{m}}_s)$. Then there exists a locally stable family of Koll\'ar models $f \colon (Y, E) \to (X, \Delta)$ at $x$ over $S$ such that 
    \begin{equation*}
        v^{\mathrm{m}}_{\eta} \in \mathrm{QM}(Y_{\eta}, E_{\eta}) \quad \text{and} \quad v^{\mathrm{m}}_{s} \in \mathrm{QM}(Y_{s}, E_{s}), 
    \end{equation*}
    and they are identified under the canonical isomorphism $\mathrm{QM}(Y_{\eta}, E_{\eta}) \simeq \mathrm{QM}(Y_{s}, E_{s})$ in \ref{quasi-monomial_deg_val}. 
\end{thm}

Theorem \ref{main_thm_1} implies a representability for semi-normal schemes: 

\begin{thm}[Corollary \ref{representability}]
    Let $S$ be a reduced scheme that is essentially of finite type over a field of characteristic zero, and $\pi \colon (X, \Delta) \to S$ with $x \in X(S)$ be a locally stable family of klt singularities. Then there exists a locally closed stratification $\bigsqcup_{i} S_{i} \to S$ satisfying the following condition: If $g \colon T \to S$ is a morphism of schemes where $T$ is semi-normal, then the base change $\pi_T \colon (X_T, \Delta_T) \to T$ with $x_T \in X_T(T)$ admits a degeneration to a locally stable family of K-semistable log Fano cone singularities over $T$ if and only if $g$ factors through some  $S^{\mathit{sn}}_i \to S$, where $S^{\mathit{sn}}_i$ is the semi-normalization of $S_i$. 
\end{thm}

The local volume is a measurement of the singularity (see, for example, \cite{LX_cubic}). We can compare the condition of constant local volume with the classical theory of equisingular plane curves: 

\begin{exmp}
    Let $\Bbbk$ be an algebraically closed field of characteristic $0$. For a plane curve $C \subset \mathbb{A}^2_{\Bbbk}$ that is (singular and) unibranch at the origin $0 \in \mathbb{A}^2$, one can associate a tuple $(a, b, \ldots)$ of positive integers, called the \emph{Puiseux characteristic}. Two such curves are \emph{equisingular} (at $0$) if and only if they have the same Puiseux characteristics (see \cite{Zar_curve}, \cite{Zar_equisingular_I}, \cite{Zar_equisingular_III}). 

    For $0 \leq \lambda < \frac{1}{a} + \frac{1}{b}$, the pair $(\mathbb{A}^2, \lambda C)$ is klt at $0 \in \mathbb{A}^2$. Moreover, in Lemma \ref{local_vol_of_unibranch_curve} we show that
    \begin{equation*}
        \widehat{\mathrm{vol}}(0; \mathbb{A}^2, \lambda C) = \left\{ \begin{array}{ll}
            ab\left(\frac{1}{a} + \frac{1}{b} - \lambda\right)^2 & \text{if}\ \frac{1}{a} - \frac{1}{b} \leq \lambda < \frac{1}{a} + \frac{1}{b},  \\
            4(1-\lambda a) & \text{if}\ 0 \leq \lambda < \frac{1}{a} - \frac{1}{b}. 
        \end{array} \right. 
    \end{equation*}
    In particular, the local volume of $0 \in (\mathbb{A}^2, \lambda C)$ only depends on $\lambda$ and the first two terms of the Puiseux characteristic of $C \subset \mathbb{A}^2$ at $0$. Therefore, Theorem \ref{main_thm_1} applies to families of equisingular unibranch plane curves; see Corollary \ref{family_of_equisingular_curve_branch}. 
\end{exmp}

\subsection*{Outline of the proof} We sketch the proof of Theorem \ref{main_thm_2}. 

The base $S$ is the spectrum of a DVR, so we can degenerate the ideal sequence $\mathfrak{a}_{\bullet}(v^{\mathrm{m}}_{\eta})$ on the generic fiber $X_{\eta}$ to an ideal sequence $I_{\bullet}$ on $X$ with flat cosupport over $S$. Then $I_{\bullet}$ has constant multiplicity over $S$; by the lower semi-continuity of the lc thresholds, the closed fiber $I_{s,\bullet}$ computes $\widehat{\mathrm{vol}}(x_s; X_{s}, \Delta_{s})$. By a stronger uniqueness of minimizer implicitly in \cite{BLQ_filt}, we conclude that $\mathfrak{a}_{s,\bullet}$ is the saturation of $I_{s,\bullet}$; see Lemma \ref{minimizing_ideal_seq}. 

Note that $v_{\eta}^{\mathrm{m}}$ computes the lc threshold of $I_{\eta, \bullet} = \mathfrak{a}_{\bullet}(v^{\mathrm{m}}_{\eta})$ on $(X_{\eta}, \Delta_{\eta})$. Then we can construct a model $f \colon (Y, E) \to (X, \Delta)$ relatively of Fano type over $X$ such that $v_{\eta}^{\mathrm{m}} \in \mathrm{QM}(Y_{\eta}, E_{\eta})$, by a similar argument that has been used in \cite[Thm.\ 4.2]{LX_stability}, \cite[Thm.\ 4.48]{Xu_K-Stability_Book}, and \cite[Lem.\ 3.2]{XZ_stable_deg}; see Lemma \ref{specialize_val_computes_constant_lct}. Since $I_{\bullet}$ has constant lc threshold, $\pi \circ f \colon (Y, f^{-1}_{*}\Delta + E) \to S$ is a locally stable family, so we can find a real valuation $v_s \in \mathrm{Val}_{X_s, x_s}$ that is monomial on $(Y_s, E_s)$, corresponding to $v_{\eta}^{\mathrm{m}}$ as in \ref{quasi-monomial_deg_val}. Since $v_s$ computes the lc thresholds of $I_{s, \bullet}$ on $(X_s, \Delta_s)$, so we can conclude $v_s = v^{\mathrm{m}}_s$ by the uniqueness. 

Now the closed fiber $(Y_s, E_s)$ is slc (semi-log canonical) by the construction, whereas we would like it to be qdlt so that we can get a family of Koll\'ar models. The key is that $v^{\mathrm{m}}_{\eta}$ and $v^{\mathrm{m}}_{s}$ are \emph{special valuations}, namely, monomial lc places of special $\mathbb{Q}$-complements, so they compute the lc threshold of $I_{\bullet}$ after any small perturbation of $\Delta$ to $\Delta + \epsilon D$; see Lemma \ref{special_val_equiv_conditions}. A heuristic is that such perturbations exclude any lc center on $(Y_s, E_s)$ that does not contain the center of $v^{\mathrm{m}}_s$. This method was used in the proof of \cite[Prop.\ 4.6]{XZ_stable_deg}. In our situation, however, the model $(Y, E)$ constructed above may change when we perturb the boundary. 

Instead, note that $v^{\mathrm{m}}_s$ can be viewed as the limit of the weighted blow-up valuations of $v^{\mathrm{m}}_{\eta}$ and the discrete valuation $v_0$ induced by $X_s$, when the weight of $v_0$ goes to $+\infty$. Then the perturbation allows us to find a special $\mathbb{Q}$-complement for which all these weighted blow-ups are monomial lc places. Hence we get a model of qdlt Fano type model accommodating them by \cite[Thm.\ 3.14]{XZ_stable_deg}, whose anti-canonical model is a locally stable family of Koll\'ar models as desired. 

\subsection*{Acknowledgment} I would like to thank my advisor Chenyang Xu for suggesting me this topic, and for constant support and guidance on this research. I would also like to thank Junyao Peng and Lu Qi for inspiring discussions. 

\section{Preliminaries} \label{sec: preliminaries}

\subsection{Birational geometry and singularities}

All schemes that we consider will be quasi-compact, separated and essentially of finite type over a field of characteristic zero. 

A \emph{pair} $(X, \Delta)$ consists of a reduced, equidimensional scheme $X$ with a Weil $\mathbb{Q}$-divisor $\Delta$. We further assume that $X$ satisfies Serre's condition $(\mathrm{S}_2)$, satisfies $(\mathrm{G}_1)$ (namely, Gorenstein in codimension $1$), and is regular at the generic point of every component $\Delta$; see \cite[Def.\ 1.5]{Kollar_singularity}. Let $K_X$ denote a canonical divisor on $X$. The pair $(X, \Delta)$ is said to be \emph{local} (resp., \emph{semi-local}) if $X$ is the spectrum of a local (resp., semi-local) ring. 

A \emph{singularity} $x \in (X, \Delta)$ consists of a pair $(X, \Delta)$ that is essentially of finite type over a field $k$ of characteristic $0$, and a $k$-point $x \in X(k)$.\footnote{It is possible to consider arbitrary closed points, and on more general excellent schemes. But many theorems for the normalized volume are only proved in this case in the literature, sometimes even assuming $k$ is algebraically closed. Also, requiring $x \in X(k)$ is compatible with our definition that a family of singularities is given by a section.} Since we are interested in the properties near the point $x$ for a singularity $x \in (X, \Delta)$, we always assume $X$ is affine, and sometimes localize at $x$ if necessary. 

We follow the standard terminology about singularities of pairs in \cite{Kollar_singularity} when $\Delta \geq 0$, but add the prefix ``sub-'' when $\Delta$ is not necessarily effective, e.g., \emph{sub-lc}, \emph{sub-klt}. We also need the notion of \emph{slc} pairs (short for \emph{semi-log canonical}, see \cite[\S 5]{Kollar_singularity}), which is a non-normal variant of lc pairs. 

\subsubsection{Qdlt pairs} We recall the concept of \emph{qdlt} pairs introduced in \cite{dFKX}. 

\begin{defn}[cf.\ {\cite[Def.\ 35]{dFKX}}, {\cite[Def.\ 2.1, 2.2]{XZ_stable_deg}}] \label{def_qdlt}
    A pair $(X, \Delta)$ is said to be \emph{simple-toroidal} if $\Delta$ is reduced and every $x \in X$ has an open neighborhood $U \subset X$ such that $(U, \Delta|_U) \simeq (V, D)/G$, where $(V, D)$ is an snc pair, and $G$ is a finite abelian group acting on $V$ that preserves each component of $D$ and acts freely on $V \smallsetminus D$. 
    
    A pair $(X, \Delta)$ is said to be \emph{qdlt} if it is lc, and $(\mathop{\mathrm{Spec}} \mathscr{O}_{X,z}, \Delta|_{\mathop{\mathrm{Spec}} \mathscr{O}_{X,z}})$ is simple-toroidal for every lc center $z \in (X, \Delta)$. 
\end{defn}

\begin{lem}[{\cite[Prop.\ 34]{dFKX}}] \label{simple_toroidal}
    Let $(X, \Delta)$ be a local slc pair of dimension $r$. Suppose $D_1, \ldots, D_r$ are reduced $\mathbb{Q}$-Cartier divisors on $X$ such that $D_1 + \cdots + D_r \leq \Delta$. Then $(X, \Delta)$ is simple-toroidal, and $\Delta = D_1 + \cdots + D_r$, where each $D_i$ is normal and irreducible. 
\end{lem}

\subsubsection{Models} 

\begin{defn}
    Let $(X, \Delta)$ be a pair. A \emph{birational model} (or simply, a \emph{model}) 
    \begin{equation*}
        f \colon (Y, E) \to (X, \Delta)
    \end{equation*}
    for $(X, \Delta)$ consists of a pair $(Y, E)$ where $Y$ is a normal scheme and $E$ is a reduced Weil divisor on $Y$, and a projective birational morphism $f \colon Y \to X$. It is called a model at a point $x \in X$ if all components of $E$ are centered at $x$. 

    A model $f \colon (Y, E) \to (X, \Delta)$ is said to be \emph{snc} (resp., \emph{simple-toroidal}) if $\mathrm{Ex}(f)$ has pure codimension one in $Y$, and $(Y, \mathrm{Supp}(f^{-1}_{*}\Delta + \mathrm{Ex}(f) + E))$ is snc (resp., simple-toroidal). It is called a \emph{log resolution}\footnote{This seems to be a non-standard use of terminology as we allow log resolutions to be only toroidal, rather than snc. However, the theory of resolution of singularities works as usual. For example, we can take a log resolution that is a local isomorphism over the simple-toroidal locus (see \cite{ATW_log_resolution}). Anyway, toroidal pairs are logarithmic regular in the sense of logarithmic geometry (see \cite[\S2.3]{AT_toroidal}). } if moreover $E$ contains all components of $f^{-1}_{*}\Delta$ and $\mathrm{Ex}(f)$.
\end{defn}

\begin{lem}[cf. {\cite[Cor.\ 1.4.3]{BCHM}}] \label{extract_div}
    Let $(X, \Delta)$ be a klt pair, and $\mathfrak{a} \subset \mathscr{O}_X$ be a coherent ideal with $\mathrm{lct}(X, \Delta; \mathfrak{a}^c) \geq 1$ for some $c > 0$. Suppose $E_1, \ldots, E_r$ are exceptional prime divisors over $X$ such that 
    \begin{equation*}
        A_{X, \Delta + \mathfrak{a}^c}(E_i) \coloneqq A_{X, \Delta}(E_i) - c \mathop{\mathrm{ord}_{E_i}}(\mathfrak{a}) < 1
    \end{equation*}
    for all $i = 1, \ldots, r$. Then there exists a model $f \colon (Y, E) \to (X, \Delta)$ where $E = \sum_{i=1}^{r} E_i$ is the sum of all $f$-exceptional divisors. Moreover, we may assume $Y$ is $\mathbb{Q}$-factorial. 
\end{lem}

\begin{proof}
    Let $g \colon (W, E) \to (X, \Delta)$ be an snc model, where $E = \sum_{i=1}^{r} E_i$ is the sum of the given divisors, such that $g^{-1}\mathfrak{a} \cdot \mathscr{O}_W = \mathscr{O}_W(-A)$ for an effective divisor $A$ and $g^{-1}\Delta + \mathrm{Ex}(g) + A$ has snc support. The conditions $A_{X, \Delta + \mathfrak{a}^c}(E_i) < 1$ still hold if we replace $c$ with $c - \epsilon$ for some $\epsilon \ll 1$. Thus we may assume that $\mathrm{lct}(X, \Delta; \mathfrak{a}^c) > 1$. Let $F = \sum_{j=1}^s F_j$ be the sum of all $g$-exceptional divisors other than $E_1, \ldots, E_r$, and choose $0 < \delta \ll 1$ such that 
    \begin{equation*}
        A_{X, \Delta + \mathfrak{a}^c}(F_j) = A_{X, \Delta}(F_j) - c \mathop{\mathrm{ord}_{F_j}}(\mathfrak{a}) > \delta
    \end{equation*}
    for all $j = 1, \ldots, s$. Let $F' = \sum_{j=1}^{s} (A_{X, \Delta + \mathfrak{a}^c}(F_j) - \delta) F_j$. Suppose we can run a relative $F'$-MMP over $X$ with scaling of a $g$-ample divisor $H$ on $W$, which terminates with a relative minimal model $f \colon Y \to X$ such that $Y$ is $\mathbb{Q}$-factorial. Then the MMP $\phi \colon W \dashrightarrow Y$ contracts all the $F_j$ and none of the $E_i$, hence $f \colon Y \to X$ is the desired model. It remains to show the such an MMP exists and terminates. 
    
    First assume that $X$ is affine. Write $K_W + \Delta_W = g^{*}(K_X + \Delta)$, then $\Delta_W + cA$ has snc support, and $(W, \Delta_W + cA)$ is sub-klt. Let $A' = g^{-1}_{*}g_{*}A$ be the part of $A$ that is not $g$-exceptional, and 
    \begin{equation*}
        \Delta_W' = \Delta_W + cA + F' = g^{-1}_{*}\Delta + cA' + \sum_{i=1}^{r} (1 - A_{X, \Delta + \mathfrak{a}^c}(E_i)) E_i + (1 - \delta) F. 
    \end{equation*}
    Then $(W, \Delta_W')$ is klt. Note that $\mathscr{O}_W(-A)$ is relatively globally generated over $X$, hence globally generated since $X$ is affine. By the Bertini theorem, we can choose a general effective $\mathbb{Q}$-divisor $G \in |{-A}|_{\mathbb{Q}}$ such that $(W, \Delta_W' + cG)$ is also klt. Then we can run the relative $(K_W + \Delta_W' + cG)$-MMP over $X$ with scaling of $H$, which terminates with a $\mathbb{Q}$-factorial relative minimal model by \cite{BCHM}. Since 
    \begin{equation*}
        K_W + \Delta_W' + cG = K_W + \Delta_W + c(A + G) + F' \sim_{X, \mathbb{R}} F', 
    \end{equation*}
    the $(K_W + \Delta_W' + cG)$-MMP is the same as the relative $F'$-MMP with scaling of $H$. 

    In general, the MMP with scaling descends from a Zariski open covering by \cite[Thm.\ 2.5]{VP}. 
\end{proof}

\subsubsection{Locally stable families} We use the notion of locally stable families of pairs following \cite{Kol_fam}. 

Let $S$ be a reduced scheme. A (well-defined) family of pairs $\pi \colon (X, \Delta) \to S$ consists of a morphism of schemes $\pi \colon X \to S$ that is flat and of finite type, with reduced, equidimensional, $(\mathrm{S}_2)$ fibers, and a Weil $\mathbb{Q}$-divisor $\Delta$ on $X$ that has well-defined pullbacks (see \cite[Thm-Def.\ 4.3]{Kol_fam}). If $S$ is normal and $\pi$ has normal fibers of pure dimension $n$, then it suffices to assume that $\mathrm{Supp}(\Delta) \to S$ has pure relative dimension $n-1$. In general, we always assume components of $\Delta$ are \emph{relative Mumford divisors} on $X$ (see \cite[Def.\ 4.68]{Kol_fam}).  

A family of pairs $\pi \colon (X, \Delta) \to S$ is said to be \emph{locally stable} if the fibers $(X_s, \Delta_s)$ are slc for all $s \in S$, and $K_{X/S} + \Delta$ is $\mathbb{Q}$-Cartier (see \cite[Def-Thm.\ 4.7]{Kol_fam}). If moreover $(X_s, \Delta_s)$ are klt (resp., qdlt, lc) for all $s \in S$, we say $\pi \colon (X, \Delta) \to S$ is a locally stable family of klt (resp., qdlt, lc) pairs. 

If $S$ is normal, then a locally stable family of klt pairs $\pi \colon (X, \Delta) \to S$ is also called a \emph{$\mathbb{Q}$-Gorenstein family} of klt pairs, since the essential condition is that $K_{X/S} + \Delta$ being $\mathbb{Q}$-Carter. 

\begin{defn} \label{notation_family_of_sing}
    Let $S$ be a reduced scheme. A \emph{family of singularities} over $S$ consists of a (well-defined) family of pairs $\pi \colon (X, \Delta) \to S$ where $\pi$ is affine, together with a section $x \colon S \to X$ of $\pi$; we denote the section by $x \in X(S)$, meaning an $S$-point of the $S$-scheme $X$. Note that $x \colon S \to X$ is a closed immersion. In the fiber $(X_s, \Delta_s)$ at a point $s \in S$, the section $x$ gives a $\kappa(s)$-point $x_s \in X_s(\kappa(s))$, where $\kappa(s)$ is the residue field of $s \in S$. 

    A \emph{locally stable family of klt singularities} over $S$ consists of a family of singularities $\pi \colon (X, \Delta) \to S$ with $x \in X(S)$, such that $\pi \colon (X, \Delta) \to S$ is a locally stable family of klt pairs. 
\end{defn}

\begin{defn}
    Let $S$ be a reduced scheme, $\pi \colon (X, \Delta) \to S$ be a family of pairs. A \emph{family of models} 
    \begin{equation*}
        f \colon (Y, E) \to (X, \Delta)
    \end{equation*}
    consists of a projective birational morphism $f \colon Y \to X$ such that $\pi \circ f$ is flat, and a relative Mumford divisor $E$ on $Y/S$ (see \cite[Def.\ 4.68]{Kol_fam}), such that every fiber $f_s \colon (Y_s, E_s) \to (X_s, \Delta_s)$ is a model. 

    Let $x \in X(S)$ be a section. Then $f \colon (Y, E) \to (X, \Delta)$ is called a \emph{family of models at $x$} if moreover every fiber $f_s \colon (Y_s, E_s) \to (X_s, \Delta_s)$ is a model at $x_s$, that is, all components of $E_s$ are centered at $x_s$. 
\end{defn}

\subsection{Valuations}

\subsubsection{Real valuations} All valuations that we consider will be real valuations. 

Let $X$ be an integral scheme with function field $K(X)$. A \emph{real valuation} $v$ on $K(X)$ gives a valuation ring $\mathcal{O}_v$ of $K(X)$ of rank $1$ with an order-preserving embedding of the valuation group $\Gamma_{v} \hookrightarrow \mathbb{R}$, which recovers the valuation $v \colon K(X)^{\times} \to \mathbb{R}$. By convention, we set $v(0) = +\infty$. 

The \emph{rational rank} of $v \in \mathrm{Val}_X$ is $\mathop{\mathrm{rat.rank}}(v) \coloneqq \dim_{\mathbb{Q}}(\Gamma_v \otimes_{\mathbb{Z}} \mathbb{Q})$. 

The set of all real valuations on $K$ centered on $X$ is denoted by $\mathrm{Val}_{X}$. For $x \in X$, let 
\begin{equation*}
    \mathrm{Val}_{X,x} \coloneqq \{ v \in \mathrm{Val}_X : \mathrm{center}_X(v) = x \}. 
\end{equation*}
More generally, if $X = \bigcup_{i} X_i$ is a reduced scheme, where $X_i$ are the irreducible components of $X$, then we write $\mathrm{Val}_X \coloneqq \bigsqcup_{i} \mathrm{Val}_{X_i}$. 

Suppose $f \colon Y \to X$ is a dominant morphism, where $Y$ is an integral scheme, then we have an induced map $\mathrm{Val}_Y \to \mathrm{Val}_X$ by restricting a valuation on $K(Y)$ to the subfield $K(X)$ via $f$. If $f$ is birational and proper, then this is a bijection, and we usually identify $\mathrm{Val}_Y = \mathrm{Val}_X$. 

A real valuation $v \in \mathrm{Val}_X$ can be evaluated at functions, coherent ideals, and Cartier divisors on $X$, as well as their $\mathbb{R}$-linear combinations (e.g., $\mathbb{R}$-Cartier $\mathbb{R}$-divisors). 

\subsubsection{Quasi-monomial valuations}

Let $(Y, E = \sum_i E_i)$ be an snc pair. Suppose $y \in Y$ is a generic point of a stratum $E_1 \cap \cdots \cap E_r \subset Y$, and $t_1, \ldots, t_r \in \mathscr{O}_{Y,y}$ is a regular system of parameters such that $E_i$ is defined by $t_i$ near $y$ for all $i = 1, \ldots, r$. Then there is an isomorphism $\hat{\mathscr{O}}_{Y,y} \simeq \kappa(y)[\![t_1, \ldots, t_r]\!]$. 

For $\alpha = (\alpha_1, \ldots, \alpha_r) \in \mathbb{R}_{\geq 0}^{r}$, there exists a unique real valuation $v_{\alpha} \in \mathrm{Val}_Y$ satisfying the following: if $f \in \mathscr{O}_{Y,y}$ is written as $f = \sum_{m \in \mathbb{N}^r} c_{m}t_1^{m_1} \cdots t_r^{m_r}$ in $\hat{\mathscr{O}}_{Y,y} \simeq \kappa(y)[\![t_1, \ldots, t_r]\!]$, then 
\begin{equation*}
    v_{\alpha}(f) = \inf \{ \langle \alpha, m \rangle : c_{m} \neq 0 \} \in \mathbb{R}_{\geq 0} \cup \{+\infty\}, \quad \text{where}\ \langle \alpha, m \rangle = \sum_{i=1}^{r} \alpha_im_i; 
\end{equation*}
see \cite[Prop.\ 3.1]{JM}. We call $v_{\alpha}$ a \emph{monomial valuation} at $y \in (Y, E)$, and denote by $\mathrm{QM}_y(Y, E)$ the set of all monomial valuations at $y$. Finally, we define $\mathrm{QM}(Y, E) \coloneqq \bigcup_y \mathrm{QM}_y(Y, E)$, where $y$ ranges over all generic points of all strata of $(Y, E)$. 

Following \cite{XZ_stable_deg}, we consider a generalization of monomial valuations to \emph{simple-toroidal} pairs. 

\begin{defn}
    Let $(Y, E = \sum_{i} E_i)$ be a simple-toroidal pair. Suppose $y \in Y$ is a generic point of a stratum $\bigcap_{i=1}^{r} E_i \subset Y$. Then locally at $y$ there is a finite abelian cover 
    \begin{equation*}
        \mu \colon (Y', E') \to (Y, E)
    \end{equation*}
    which is \'etale away from $E$, such that $(Y', E' = \sum_i E'_i)$ is snc, and $\mu^{*}E_i = m_iE_i'$ for some $m_i > 0$. Let $y' \in \mu^{-1}(y)$, then $y'$ is a generic point of $\bigcap_{i=1}^{r} E_i' \subset Y'$. For $\alpha = (\alpha_1, \ldots, \alpha_r) \in \mathbb{R}_{\geq 0}^r$, define $v_{\alpha} \in \mathrm{Val}_Y$ to be the restriction of $v'_{\alpha'} \in \mathrm{QM}_{y'}(Y', E')$ corresponding to $\alpha' = (\alpha_1/m_1, \ldots, \alpha_r/m_r)$. Thus 
    \begin{equation*}
        v_{\alpha}(E_i) = v'_{\alpha'}(\mu^{*}E_i) = m_iv'_{\alpha'}(E_i') = \alpha_i. 
    \end{equation*}
    Note that $v_{\alpha}$ does not depend on the choice of a finite abelian cover. 
    
    We call $v_{\alpha}$ a \emph{monomial valuation} at $y \in (Y, E)$, and denote by $\mathrm{QM}_y(Y, E) \subset \mathrm{Val}_Y$ the set of all such $v_{\alpha}$. The \emph{relative interior} $\mathrm{QM}_y^{\circ}(Y, E) \subset \mathrm{QM}_y(Y, E)$ consists of $v_{\alpha}$ for which $\alpha \in \mathbb{R}_{>0}^r$. By definition there is a canonical bijection
    \begin{equation*}
        \mathrm{QM}_y(Y, E) \simeq \mathbb{R}_{\geq 0}^{r}, \quad v \mapsto (v(E_1), \ldots, v(E_r)), 
    \end{equation*}
    which is a homeomorphism when $\mathrm{Val}_Y$ is equipped with the weakest topology such that the evaluation map $\mathrm{Val}_Y \to \mathbb{R}$, $v \mapsto v(f)$ is continuous for all $f \in K(Y)^{\times}$ (see \cite[Lem.\ 4.5]{JM}). 
    
    More generally, let $(Y, E)$ be a pair where $E = \sum_{i} E_i$ is reduced, such that $(\mathop{\mathrm{Spec}} \mathscr{O}_{Y,y}, E|_{\mathrm{Spec} \mathscr{O}_{Y,y}})$ is simple-toroidal for all generic points $y$ of all strata $\bigcap_{i \in I} E_i \subset Y$. Then we define 
    \begin{equation*}
        \mathrm{QM}(Y, E) \coloneqq \bigcup\nolimits_{y} \mathrm{QM}_{y}(Y, E) = \bigsqcup\nolimits_{y} \mathrm{QM}^{\circ}_y(Y, E). 
    \end{equation*}
\end{defn}

\begin{rem}
    Let $(Y, E)$ be a pair where $E$ is a reduced Weil divisor. Then $(Y, E)$ is simple-toroidal if and only if $Y \smallsetminus E \hookrightarrow Y$ is a toroidal embedding without self-intersection such that the rational convex polyhedral cones $\sigma^Z \subset N^Z_{\mathbb{R}}$ associated to all strata $Z \subset Y$ are simplicial (see \cite[II.\S1]{KKMSD_toroidal} for the relevant notions). 

    In fact, if $(Y, E)$ is Zariski locally, at a point $y \in E$, isomorphic to the quotient $(V, D)/G$ where $(V, D)$ is snc and $G$ is a finite abelian group, then the action of $G$ on $(V, D)$ is formally isomorphic to an action of $G'$ on $\mathbb{A}^n$ where $G \simeq G' \subset (\mathbb{G}_{\mathrm{m}})^n$, and $D$ corresponds to a union of coordinate hyperplanes $D' \subset \mathbb{A}^n$ (see \cite[{}3.17]{Kollar_singularity}). Then $(\mathbb{A}^n, D')/G'$ is an affine toric variety associated to a simplicial cone, and $(Y, E)$ is formally isomorphic to $(\mathbb{A}^n, D')/G'$ at $y \in E$. That is, $Y \smallsetminus E \hookrightarrow Y$ is a toroidal embedding, and the associated cones are simplicial; it has no self-intersections since every stratum is normal. 
    
    Conversely, assume $Y \smallsetminus E \hookrightarrow Y$ is such a toroidal embedding, then every component of $E$ is $\mathbb{Q}$-Cartier (see \cite[II.\S1, Lem.\ 1]{KKMSD_toroidal}). To see $(Y, E)$ is simple-toroidal in the sense of Definition \ref{def_qdlt}, we may take a composition of cyclic covers, and assume every component of $E$ is Cartier. In this case, the cone $\sigma^Z \subset N^Z_{\mathbb{R}}$ is spanned by a basis of the lattice $N^Z$, hence $(Y, E)$ is snc. 

    Let $Z \subset Y$ be a stratum, with the generic point $z \in Z$. Recall that $\sigma^Z \subset N^Z_{\mathbb{R}}$ is the dual cone of the cone of effective Cartier divisors supported on $E$ and passing through $Z$ (see \cite[II.\S1, Def.\ 3]{KKMSD_toroidal}). Then there is a canonical isomorphism 
    \begin{equation*}
        \mathrm{QM}_z(Y, E) \simeq \sigma^Z
    \end{equation*}
    and $\sigma^Z \cap N^Z$ corresponds to the lattice generated by valuations given by prime divisors over $Y$. Hence we can equip $\mathrm{QM}(Y, E)$ with an integral structure (see \cite[71]{KKMSD_toroidal}). 

    By the results in \cite[II.\S2]{KKMSD_toroidal}, if $f \colon (Y', E') \to (Y, E)$ is a toroidal resolution, then 
    \begin{equation*}
        \mathrm{QM}(Y', E') = \mathrm{QM}(Y, E). 
    \end{equation*}
    Conversely, if $\sigma \subset \mathrm{QM}_y(Y, E)$ is a rational simplicial cone, then there exists a toroidal proper birational morphism $f \colon (Y', E') \to (Y, E)$ with $(Y', E')$ simple-toroidal and $\mathrm{QM}_{y'}(Y', E') = \sigma$ for some $y' \in f^{-1}(y)$. Moreover, we can choose $(Y', E')$ such that every $f$-exceptional divisor corresponds to an edge of $\sigma$, and $\mathrm{Ex}(f)$ supports an $f$-ample divisor, by \cite[Lem.\ 3.16]{XZ_stable_deg}. 
\end{rem} 

\begin{lem} \label{degenerate_monomial_val}
    Let $(Y, E + F)$ be a local simple-toroidal pair, where $E = \sum_{i=1}^{r} E_i$ and $F = \sum_{j=1}^{s} F_j$ have no common components. Let $Y_0 = \bigcap_{j=1}^{s} F_j$, and $E_0 = E|_{Y_0}$. Then $(Y_0, E_0)$ is simple-toroidal, and there is a commutative diagram 
    \begin{equation*}
        \begin{tikzcd}
            \mathrm{QM}(Y, E) \ar[r, "i", hook] \ar[d, "\simeq"] & \mathrm{QM}(Y, E + F) \ar[r, "p"] \ar[d, "\simeq"] & \mathrm{QM}(Y_0, E_0) \ar[d, "\simeq"] \\ 
            \mathbb{R}_{\geq 0}^{r} \ar[r, "{\alpha \mapsto (\alpha, 0)}"] & \mathbb{R}_{\geq 0}^{r+s} \ar[r, "{(\alpha, \beta) \mapsto \alpha}"] & \mathbb{R}_{\geq 0}^{r}
        \end{tikzcd}
    \end{equation*}
    so that the composition in the first row is a canonical isomorphism 
    \begin{equation*}
        p \circ i \colon \mathrm{QM}(Y, E) \simeq \mathrm{QM}(Y_0, E_0). 
    \end{equation*}
    Moreover, if $v_{\alpha} \in \mathrm{QM}(Y, E)$ maps to $\bar{v}_{\alpha} \in \mathrm{QM}(Y_0, E_0)$, and $\beta = (\beta_j) \in \mathbb{R}_{\geq 0}^s$, then 
    \begin{equation*}
        v_{\alpha}(f) \leq v_{\alpha,\beta}(f) \leq \sup \{ v_{\alpha, \beta}(f) : \beta \in \mathbb{R}_{\geq 0}^s \} = \lim_{\beta \to \infty} v_{\alpha, \beta}(f) = \bar{v}_{\alpha}(\bar{f}), 
    \end{equation*}
    for every $f \in \mathscr{O}_Y$, where $\bar{f} \in \mathscr{O}_{Y_0}$ is the image of $f$, and $\beta \to \infty$ means each $\beta_j \to \infty$. 
\end{lem}

\begin{proof}
    By localization and passing to a finite abelian cover, we may assume that $(Y, E + F)$ is snc, and $\bigcap_{i=1}^{r} E_i \cap \bigcap_{j=1}^{s} F_j = \{y\}$, where $y \in Y$ is the closed point. Write $R = \mathscr{O}_{Y,y}$, and take a regular system of parameters $t_1, \ldots, t_r, x_1, \ldots, x_s \in R$ such that $E_i = \mathrm{div}(t_i)$ and $F_j = \mathrm{div}(x_j)$. Then 
    \begin{equation*}
        \mathscr{O}_{Y_0, y} \simeq R_0 \coloneqq R/(x_1, \ldots, x_s)
    \end{equation*}
    is a regular local ring, and $\bar{t}_1, \ldots, \bar{t}_r \in R_0$ is a regular system of parameters defining $E_0$. Write 
    \begin{equation*}
        f = \sum_{m \in \mathbb{N}^r} \sum_{n \in \mathbb{N}^s} c_{m,n}t_1^{m_1} \cdots t_r^{m_r}x_1^{n_1} \cdots x_s^{n_s} \in \hat{R} \simeq \kappa(y)[\![t_1, \ldots, t_r, x_1, \ldots, x_s]\!], 
    \end{equation*}
    so that 
    \begin{equation*}
        \bar{f} = \sum_{m \in \mathbb{N}^r} c_{m,0}\bar{t}_1^{m_1} \cdots \bar{t}_r^{m_r} \in \hat{R}_0 \simeq \kappa(y)[\![\bar{t}_1, \ldots, \bar{t}_r]\!]. 
    \end{equation*}
    Hence 
    \begin{equation*}
        \begin{aligned}
            v_{\alpha}(f) = \inf \{ \langle \alpha, m \rangle : c_{m,n} \neq 0 \} &\leq \inf \{ \langle \alpha, m \rangle + \langle \beta, n \rangle : c_{m,n} \neq 0 \} \\
            &\leq \sup_{\beta \geq 0} \inf \{ \langle \alpha, m \rangle + \langle \beta, n \rangle : c_{m,n} \neq 0 \} \\
            &= \lim_{\beta \to \infty} \inf \{ \langle \alpha, m \rangle + \langle \beta, n \rangle : c_{m,n} \neq 0 \} \\
            &= \inf \{ \langle \alpha, m \rangle : c_{m,0} \neq 0 \} = \bar{v}_{\alpha}(\bar{f})
        \end{aligned}
    \end{equation*}
    since $\langle \alpha, m \rangle + \langle \beta, n \rangle = \sum_{i=1}^{r} \alpha_i m_i + \sum_{j=1}^{s} \beta_j n_j \geq \min_j \beta_j$ when $n \neq 0$. 
\end{proof}

\begin{defn}
    Let $X$ be an integral scheme. A real valuation $v \in \mathrm{Val}_X$ is said to be \emph{quasi-monomial} if $v \in \mathrm{QM}(Y, E)$ for some snc pair $(Y, E)$ with a birational morphism $f \colon Y \to X$ of finite type. The set of quasi-monomial valuations on $X$ is denoted by $\mathrm{Val}_X^{\mathrm{qm}} \subset \mathrm{Val}_X$. It is equivalent to require $(Y, E)$ to be simple-toroidal. 

    Note that divisorial valuations are exactly quasi-monomial valuations of rational rank $1$. The set of divisorial valuations on $X$ is denoted by $\mathrm{Val}_X^{\mathrm{div}} \subset \mathrm{Val}_X^{\mathrm{qm}}$. 
\end{defn}

\subsubsection{Log discrepancies}

Let $(X, \Delta)$ be a pair such that $K_X + \Delta$ is $\mathbb{Q}$-Cartier. Suppose $f \colon Y \to X$ is a birational morphism of finite type, and $F$ is a prime divisor on $Y$. Then the log discrepancy 
\begin{equation*}
    A_{X,\Delta}(F) \coloneqq 1 + \mathrm{ord}_F(K_Y - f^{*}(K_X + \Delta))
\end{equation*}
only depends on the valuation $\mathrm{ord}_F \in \mathrm{Val}_X$. The log discrepancy extends to a function 
\begin{equation*}
    A_{X,\Delta} \colon \mathrm{Val}_X \to \mathbb{R} \cup \{+ \infty\}
\end{equation*}
such that 
\begin{equation*}
    A_{X, \Delta}(v) \coloneqq \sup_{(Y, E)} \sum\nolimits_i v(E_i) \cdot A_{X, \Delta}(E_i)
\end{equation*}
where the supremum ranges over all log resolutions $f \colon (Y, E = \sum_{i} E_i) \to (X, \Delta)$; see \cite[\S 5]{JM}. 

\begin{lem}
    Let $(X, \Delta)$ be a pair such that $K_X + \Delta$ is $\mathbb{Q}$-Cartier. Suppose 
    \begin{equation*}
        f \colon (Y, E = E_1 + \cdots + E_r) \to (X, \Delta)
    \end{equation*}
    is a log resolution. Then 
    \begin{equation*}
        A_{X, \Delta}(v) \geq \sum_{i = 1}^{r} v(E_i) \cdot A_{X, \Delta}(E_i)
    \end{equation*}
    for all $v \in \mathrm{Val}_X$, and the equality holds if and only if $v \in \mathrm{QM}(Y, E)$. 
\end{lem}

\begin{proof}
    Note that the value $\sum_{i = 1}^{r} v(E_i) \cdot A_{X, \Delta}(E_i)$ does not change if we replace $(Y, E)$ with a toroidal resolution, and the set $\mathrm{QM}(Y, E)$ also remains the same. Hence we may assume that $f \colon (Y, E) \to (X, \Delta)$ is an snc model, so the inequality follows by definition. 

    If the equality holds, then it is proved that $v \in \mathrm{QM}(Y, E)$ in \cite[Cor.\ 5.4]{JM} when $X$ is regular and $\Delta = 0$. In general, we can reduce it to the regular case as follows. Let $\Delta_Y$ be the crepant pullback of $\Delta$, so that 
    \begin{equation*}
        \Delta_Y = \sum_{i=1}^{r} (1 - A_{X, \Delta}(E_i)) \cdot E_i, 
    \end{equation*}
    and 
    \begin{equation*}
        A_{X, \Delta}(v) = A_{Y, \Delta_Y}(v) = A_{Y}(v) - \sum_{i=1}^{r} v(E_i) \cdot (1 - A_{X, \Delta}(E_i)). 
    \end{equation*}
    If $A_{X, \Delta}(v) = \sum_{i = 1}^{r} v(E_i) \cdot A_{X, \Delta}(E_i)$, then $A_Y(v) = \sum_{i=1}^{r} v(E_i)$, hence $v \in \mathrm{QM}(Y, E)$. 
\end{proof}

\begin{defn}
    Let $(X, \Delta)$ be an lc pair. A real valuation $v \in \mathrm{Val}_X$ is called an \emph{lc place} for $(X, \Delta)$ if $A_{X, \Delta}(v) = 0$. Let $\mathrm{LCP}(X, \Delta) \subset \mathrm{Val}_X$ denote the set of all lc places for $(X, \Delta)$. A point $z \in X$ is called an \emph{lc center} of $(X, \Delta)$ if $z = \mathrm{center}_X(v)$ for some $v \in \mathrm{LCP}(X, \Delta)$. Note that $\mathrm{LCP}(X, \Delta) \subset \mathrm{Val}_{X}^{\mathrm{qm}}$. 
\end{defn}

\subsection{Local volumes}

\subsubsection{Ideal sequences} We consider ideal sequences indexed by $\mathbb{R}_{\geq 0}$. They are called \emph{graded systems of ideals} by \cite{ELS}, and \emph{filtrations} by \cite{BLQ_filt}. 

\begin{defn}
    Let $X$ be a scheme. An \emph{ideal sequence} $\mathfrak{a}_{\bullet} = \{\mathfrak{a}_{\lambda}\}_{\lambda \geq 0}$ on $X$ consists of coherent ideals $\mathfrak{a}_{\lambda} \subset \mathscr{O}_X$ for $\lambda \in \mathbb{R}_{\geq 0}$, satisfying the following conditions: 
    \begin{enumerate}[label=(\arabic*), nosep]
        \item $\mathfrak{a}_{\lambda} \subset \mathfrak{a}_{\mu}$ if $\lambda \geq \mu$; 
        \item $\mathfrak{a}_{\lambda} \mathfrak{a}_{\mu} \subset \mathfrak{a}_{\lambda + \mu}$ for all $\lambda, \mu \geq 0$. 
        \item $\mathfrak{a}_0 = \mathscr{O}_X$, and $\mathfrak{a}_{\lambda} = \bigcap_{\lambda' < \lambda} \mathfrak{a}_{\lambda'}$ for all $\lambda > 0$. 
    \end{enumerate}
    Let $x \in X$ be a point, with the prime ideal $\mathfrak{p}_x \subset \mathscr{O}_X$. We say an ideal sequence $\mathfrak{a}_{\bullet}$ is centered at $x$ if $\mathfrak{a}_{\lambda}$ is $\mathfrak{p}_x$-primary for all $\lambda > 0$. In this case, $\mathfrak{a}_{\bullet}$ is uniquely determined by the corresponding ideal sequence $\mathfrak{a}_{\bullet}\mathscr{O}_{X,x} = \{\mathfrak{a}_{\lambda}\mathscr{O}_{X,x}\}_{\lambda}$ of $\mathscr{O}_{X,x}$; hence we may assume $x$ is a closed point by localization if necessary. 
    
    Suppose $X$ in integral and $v \in \mathrm{Val}_X$. The ideal sequence $\mathfrak{a}_{\bullet}(v)$ associated with $v$ is defined by 
    \begin{equation*}
        \mathfrak{a}_{\lambda}(v) \coloneqq \{ f \in \mathscr{O}_X : v(f) \geq \lambda \}. 
    \end{equation*}
    If $\mathrm{center}_X(v) = x$, then $\mathfrak{a}_{\bullet}(v)$ is an ideal sequence centered at $x$. 
    
    A real valuation $v$ on $X$ can be evaluated at an ideal sequence $\mathfrak{a}_{\bullet}$ by 
    \begin{equation*}
        v(\mathfrak{a}_{\bullet}) = \inf_{\lambda > 0} \frac{v(\mathfrak{a}_{\lambda})}{\lambda} = \lim_{\lambda \to \infty} \frac{v(\mathfrak{a}_{\lambda})}{\lambda} \in \mathbb{R}_{\geq 0} \cup \{+\infty\}. 
    \end{equation*}
    Note that $v(\mathfrak{a}_{\bullet}(v)) = 1$. 
\end{defn}

\begin{defn}
    Let $X$ be an integral scheme, $x \in X$ be a closed point, and $\mathfrak{a}_{\bullet} = \{\mathfrak{a}_{\lambda}\}_{\lambda \geq 0}$ be an ideal sequence centered at $x$. Let $n = \dim_x(X)$. The \emph{multiplicity} of $\mathfrak{a}_{\bullet}$ is 
    \begin{equation*}
        \mathrm{e}(\mathfrak{a}_{\bullet}) = \mathrm{e}_{X}(\mathfrak{a}_{\bullet}) \coloneqq \limsup_{\lambda \to \infty} \frac{\mathrm{length}_{\mathscr{O}_X}(\mathscr{O}_X/\mathfrak{a}_{\lambda})}{\lambda^n/n!}. 
    \end{equation*}
    The limit superior is actually a limit by \cite[Thm.\ 1.1]{Cut_limit} under our assumptions for schemes. 

    Suppose $v \in \mathrm{Val}_{X,x}$. The \emph{volume} of $v$ is 
    \begin{equation*}
        \mathrm{vol}(v) = \mathrm{vol}_{X}(v) \coloneqq \mathrm{e}_{X}(\mathfrak{a}_{\bullet}(v)), 
    \end{equation*}
    where $\mathfrak{a}_{\bullet}(v) \subset \mathscr{O}_X$ is the ideal sequence associated with $v$. The set of valuations with positive volume is denoted by $\mathrm{Val}_{X,x}^{+}$. Note that $\mathrm{Val}_{X,x}^{\mathrm{qm}} \subset \mathrm{Val}_{X,x}^{+}$ by Izumi's Theorem; see \cite[Prop.\ 2.7]{BLQ_filt}. 
\end{defn}

\begin{defn}[cf.\ {\cite[Def.\ 3.1]{BLQ_filt}}]
    Let $X$ be an integral scheme, $x \in X$ be a closed point, and $\mathfrak{a}_{\bullet}$ be an ideal sequence centered at $x$. The \emph{saturation} $\tilde{\mathfrak{a}}_{\bullet}$ of $\mathfrak{a}_{\bullet}$ is the ideal sequence defined by 
    \begin{equation*}
        \tilde{\mathfrak{a}}_{\lambda} = \left\{ f \in \mathscr{O}_X : v(f) \geq \lambda \cdot v(\mathfrak{a}_{\bullet}) \ \text{for all}\ v \in \mathrm{Val}_{X,x}^{\mathrm{div}} \right\}. 
    \end{equation*}
    It has following basic properties: 
    \begin{enumerate}[label=(\arabic*), nosep]
        \item \cite[Lem.\ 3.8]{BLQ_filt} $\mathfrak{a}_{\lambda} \subset \tilde{\mathfrak{a}}_{\lambda}$ for all $\lambda \geq 0$; 
        \item \cite[Cor.\ 3.16]{BLQ_filt} $\mathrm{e}_X(\mathfrak{a}_{\bullet}) = \mathrm{e}_X(\tilde{\mathfrak{a}}_{\bullet})$; 
        \item \cite[Prop.\ 3.12]{BLQ_filt} $v(\mathfrak{a}_{\bullet}) = v(\tilde{\mathfrak{a}}_{\bullet})$ for all $v \in \mathrm{Val}_{X,x}^{+}$. 
    \end{enumerate}
\end{defn}

\begin{lem} \label{lct_of_saturation}
    Let $x \in (X, \Delta)$ be a klt singularity, and $\mathfrak{a}_{\bullet}$ be an ideal sequence centered at $x$ with 
    \begin{equation*}
        \mathrm{lct}(X, \Delta; \mathfrak{a}_{\bullet}) < +\infty. 
    \end{equation*}
    Suppose $\tilde{\mathfrak{a}}_{\bullet}$ is the saturation of $\mathfrak{a}_{\bullet}$. Then 
    \begin{equation*}
        \mathrm{lct}(X, \Delta; \tilde{\mathfrak{a}}_{\bullet}) = \mathrm{lct}(X, \Delta; \mathfrak{a}_{\bullet}), 
    \end{equation*}
    and a valuation $v \in \mathrm{Val}_{X,x}$ computes the lc threshold of $\tilde{\mathfrak{a}}_{\bullet}$ if and only if it computes that of $\mathfrak{a}_{\bullet}$. 
\end{lem}

\begin{proof}
    We have 
    \begin{equation*}
        \mathrm{lct}(X, \Delta; \tilde{\mathfrak{a}}_{\bullet}) = \inf_{w \in \mathrm{Val}_{X,x}^{\mathrm{div}}} \frac{A_{X,\Delta}(w)}{w(\tilde{\mathfrak{a}}_{\bullet})} = \inf_{w \in \mathrm{Val}_{X,x}^{\mathrm{div}}} \frac{A_{X,\Delta}(w)}{w(\mathfrak{a}_{\bullet})} = \mathrm{lct}(X, \Delta; \mathfrak{a}_{\bullet})
    \end{equation*}
    since $w(\tilde{\mathfrak{a}}_{\bullet}) = w(\mathfrak{a}_{\bullet})$ for every $w \in \mathrm{Val}_{X,x}^{\mathrm{div}}$. For a general real valuation $v \in \mathrm{Val}_{X,x}$, we have 
    \begin{equation*}
        \frac{A_{X,\Delta}(v)}{v(\tilde{\mathfrak{a}}_{\bullet})} \geq \frac{A_{X,\Delta}(v)}{v(\mathfrak{a}_{\bullet})}
    \end{equation*}
    since $\mathfrak{a}_{\bullet} \subset \tilde{\mathfrak{a}}_{\bullet}$. Thus, if $v$ computes the lc threshold of $\tilde{\mathfrak{a}}_{\bullet}$, then it also computes that of $\mathfrak{a}_{\bullet}$. Conversely, assume $v$ computes the lc threshold of $\mathfrak{a}_{\bullet}$. Then $v$ also computes the lc threshold of the ideal sequence $\mathfrak{a}_{\bullet}(v)$ associated with itself by \cite[Thm.\ 7.8]{JM}. Hence 
    \begin{equation*}
        \inf_{w \in \mathrm{Val}_{X,x}^{\mathrm{div}}} \frac{A_{X,\Delta}(w)}{w(\mathfrak{a}_{\bullet}(v))} = \mathrm{lct}(X, \Delta; \mathfrak{a}_{\bullet}(v)) = \frac{A_{X,\Delta}(v)}{v(\mathfrak{a}_{\bullet}(v))} < +\infty. 
    \end{equation*}
    Then there exists a divisorial valuation $w \in \mathrm{Val}_{X,x}^{\mathrm{div}}$ such that $w(\mathfrak{a}_{\bullet}(v)) > 0$. By rescaling, we may assume that $w(\mathfrak{a}_{\bullet}(v)) = 1$, so that $w(f) \geq v(f)$ for all $f \in \mathscr{O}_{X,x}$. Thus 
    \begin{equation*}
        \mathrm{vol}_X(v) \geq \mathrm{vol}_X(w) > 0. 
    \end{equation*}
    By \cite[Prop.\ 3.12]{BLQ_filt}, $v(\tilde{\mathfrak{a}}_{\bullet}) = v(\mathfrak{a}_{\bullet})$. Hence $v$ also computes the lc threshold of $\tilde{\mathfrak{a}}_{\bullet}$. 
\end{proof}

\subsubsection{Normalized volumes and the minimizers}

\begin{defn}[cf.\ {\cite[\S1]{Li_vol}}]
    Let $x \in (X, \Delta)$ be a klt singularity, and $v \in \mathrm{Val}_{X,x}$ be a real valuation. The normalized volume of $v$ is 
    \begin{equation*}
        \widehat{\mathrm{vol}}(v) = \widehat{\mathrm{vol}}_{X, \Delta}(v) \coloneqq A_{X, \Delta}(v)^n \cdot \mathrm{vol}_X(v) \in \mathbb{R}_{\geq 0} \cup \{+\infty\}
    \end{equation*}
    where $n = \dim_x(X)$, and by convention $+\infty \cdot 0 = +\infty$. 

    The \emph{local volume} of $x \in (X, \Delta)$ is 
    \begin{equation*}
        \widehat{\mathrm{vol}}(x; X, \Delta) \coloneqq \inf \left\{ \widehat{\mathrm{vol}}_{X, \Delta}(v) : v \in \mathrm{Val}_{X,x} \right\}. 
    \end{equation*}
    The local volume is positive by \cite[Thm.\ 1.2]{Li_vol}. A real valuation $v \in \mathrm{Val}_{X,x}$ is called a \emph{minimizer} for the normalized volume if $\widehat{\mathrm{vol}}(v) = \widehat{\mathrm{vol}}(x; X, \Delta)$. 
\end{defn}

Let $x \in (X, \Delta)$ be a klt singularity. There exists a minimizer $v \in \mathrm{Val}_{X,x}$ of the normalized volume by \cite[Main Thm.]{Blu_exist} and \cite[Rmk.\ 3.8]{Xu_qm}, which is unique up to scaling by \cite[Thm.\ 1.1]{XZ_unique} and \cite[Cor.\ 1.3]{BLQ_filt}, and is quasi-monomial by \cite[Thm.\ 1.2]{Xu_qm}. 

\begin{lem} \label{minimizing_ideal_seq}
    Let $x \in (X, \Delta)$ be a klt singularity, and $v^{\mathrm{m}} \in \mathrm{Val}_{X,x}$ be a minimizer of the normalized volume. Let $n = \dim_x(X)$. Then the following hold: 
    \begin{enumerate}[label=\emph{(\arabic*)}, nosep]
        \item $v^{\mathrm{m}}$ computes the lc threshold of $\mathfrak{a}_{\bullet}(v^{\mathrm{m}})$, and
        \begin{equation*}
            \widehat{\mathrm{vol}}(x; X, \Delta) = \mathrm{lct}(X, \Delta; \mathfrak{a}_{\bullet}(v^{\mathrm{m}}))^n \cdot \mathrm{e}_X(\mathfrak{a}_{\bullet}(v^{\mathrm{m}})) = \inf \left\{ \mathrm{lct}(X, \Delta; \mathfrak{a}_{\bullet})^n \cdot \mathrm{e}_X(\mathfrak{a}_{\bullet}) \right\}
        \end{equation*}
        where $\mathfrak{a}_{\bullet}$ ranges over all ideal sequences centered at $x$; 
        \item if $v \in \mathrm{Val}_{X,x}$ computes the lc threshold of $\mathfrak{a}_{\bullet}(v^{\mathrm{m}})$, then $v = c \cdot v^{\mathrm{m}}$ for some $c \in \mathbb{R}_{>0}$; 
        \item if $\mathfrak{a}_{\bullet} = \{\mathfrak{a}_{\lambda}\}_{\lambda}$ is an ideal sequence centered at $x$ such that $\widehat{\mathrm{vol}}(x; X, \Delta) = \mathrm{lct}(X, \Delta; \mathfrak{a}_{\bullet})^n \cdot \mathrm{e}_X(\mathfrak{a}_{\bullet})$, then $\tilde{\mathfrak{a}}_{\lambda} = \mathfrak{a}_{c\lambda}(v^{\mathrm{m}})$ for some $c \in \mathbb{R}_{>0}$, for all $\lambda \geq 0$. 
    \end{enumerate}
\end{lem}

\begin{proof}
    (1). By \cite[Thm.\ 27]{Liu_vol_KE}, we have 
    \begin{equation*}
        \widehat{\mathrm{vol}}(x; X, \Delta) = \inf \left\{ \mathrm{lct}(X, \Delta; \mathfrak{a}_{\bullet})^n \cdot \mathrm{e}_X(\mathfrak{a}_{\bullet}) \right\}. 
    \end{equation*}
    Since $v^{\mathrm{m}}(\mathfrak{a}_{\bullet}(v^{\mathrm{m}})) = 1$, and $\mathrm{e}_X(\mathfrak{a}_{\bullet}(v^{\mathrm{m}})) = \mathrm{vol}_X(v^{\mathrm{m}})$ by definition, and 
    \begin{equation*}
        \mathrm{lct}(X, \Delta; \mathfrak{a}_{\bullet}(v^{\mathrm{m}})) \leq \frac{A_{X,\Delta}(v^{\mathrm{m}})}{v^{\mathrm{m}}(\mathfrak{a}_{\bullet}(v^{\mathrm{m}}))} = A_{X,\Delta}(v^{\mathrm{m}}), 
    \end{equation*}
    we have 
    \begin{equation*}
        \mathrm{lct}(X, \Delta; \mathfrak{a}_{\bullet}(v^{\mathrm{m}}))^n \cdot \mathrm{e}_X(\mathfrak{a}_{\bullet}(v^{\mathrm{m}})) \leq A_{X, \Delta}(v^{\mathrm{m}})^n \cdot \mathrm{vol}_X(v^{\mathrm{m}}) = \widehat{\mathrm{vol}}(x; X, \Delta). 
    \end{equation*}
    Thus the equality holds, and $v^{\mathrm{m}}$ computes the lc threshold $\mathrm{lct}(X, \Delta; \mathfrak{a}_{\bullet}(v^{\mathrm{m}}))$. 

    (2). Suppose $v$ computes the lc threshold of $\mathfrak{a}_{\bullet}(v^{\mathrm{m}})$, that is, 
    \begin{equation*}
        \mathrm{lct}(X, \Delta; \mathfrak{a}_{\bullet}(v^{\mathrm{m}})) = \frac{A_{X,\Delta}(v)}{v(\mathfrak{a}_{\bullet}(v^{\mathrm{m}}))}. 
    \end{equation*}
    By scaling, we may assume that $v(\mathfrak{a}_{\bullet}(v^{\mathrm{m}})) = \inf v(\mathfrak{a}_{\lambda}(v^{\mathrm{m}}))/\lambda = 1$. Thus $v(\mathfrak{a}_{\lambda}(v^{\mathrm{m}})) \geq \lambda$, i.e., 
    \begin{equation*}
        \mathfrak{a}_{\lambda}(v^{\mathrm{m}}) \subset \mathfrak{a}_{\lambda}(v)
    \end{equation*}
    for all $\lambda \geq 0$. Hence 
    \begin{equation*}
        \mathrm{vol}_X(v) = \mathrm{e}_X(\mathfrak{a}_{\lambda}(v)) \leq \mathrm{e}_X(\mathfrak{a}_{\lambda}(v^{\mathrm{m}})) = \mathrm{vol}_X(v^{\mathrm{m}}). 
    \end{equation*}
    Since $A_{X,\Delta}(v) = \mathrm{lct}(X, \Delta; \mathfrak{a}_{\bullet}(v^{\mathrm{m}})) = A_{X,\Delta}(v^{\mathrm{m}})$, we get 
    \begin{equation*}
        \widehat{\mathrm{vol}}_{X,\Delta}(v) \leq \widehat{\mathrm{vol}}_{X,\Delta}(v^{\mathrm{m}}) = \widehat{\mathrm{vol}}(x; X, \Delta). 
    \end{equation*}
    By uniqueness of the minimizer, $v = c \cdot v^{\mathrm{m}}$ for some $c > 0$. In fact, $v = v^{\mathrm{m}}$ since $v(\mathfrak{a}_{\bullet}(v^{\mathrm{m}})) = 1$. 

    (3). This is obtained in the proof of \cite[Cor.\ 1.3]{BLQ_filt}, which we briefly recall: We may assume that $\mathrm{lct}(\mathfrak{a}_{\bullet}) = \mathrm{lct}(\mathfrak{a}_{\bullet}(v^{\mathrm{m}})) = 1$ by scaling. Let $\mathfrak{a}_{\bullet, t}$, where $0 \leq t \leq 1$, be the geodesic between $\mathfrak{a}_{\bullet, 0} = \mathfrak{a}_{\bullet}$ and $\mathfrak{a}_{\bullet, 1} = \mathfrak{a}_{\bullet}(v^{\mathrm{m}})$ defined in \cite[Def.\ 4.1]{BLQ_filt}. Then 
    \begin{equation*}
        \mathrm{lct}(X, \Delta; \mathfrak{a}_{\bullet, t}) \leq (1-t) \cdot \mathrm{lct}(\mathfrak{a}_{\bullet}) + t \cdot \mathrm{lct}(\mathfrak{a}_{\bullet}(v^{\mathrm{m}})) = 1
    \end{equation*}
    by \cite[Prop.\ 5.1]{BLQ_filt}, and 
    \begin{equation*}
        \mathrm{e}_X(\mathfrak{a}_{\bullet, t})^{-1/n} \geq (1-t) \cdot \mathrm{e}_X(\mathfrak{a}_{\bullet})^{-1/n} + t \cdot \mathrm{e}_X(\mathfrak{a}_{\bullet}(v^{\mathrm{m}}))^{-1/n} = \widehat{\mathrm{vol}}(x; X, \Delta)^{-1/n}
    \end{equation*}
    by \cite[Thm.\ 1.1(2)]{BLQ_filt}. Hence 
    \begin{equation*}
        \mathrm{lct}(X, \Delta; \mathfrak{a}_{\bullet, t})^n \cdot \mathrm{e}_X(\mathfrak{a}_{\bullet, t}) \leq \widehat{\mathrm{vol}}(x; X, \Delta), 
    \end{equation*}
    so all inequalities above are equalities. Then $\mathfrak{a}_{\bullet}$ and $\mathfrak{a}_{\bullet}(v^{\mathrm{m}})$ have the same saturation, up to scaling, by \cite[Thm.\ 1.1(3)]{BLQ_filt}. But $\mathfrak{a}_{\bullet}(v^{\mathrm{m}})$ is saturated by \cite[Lem.\ 3.20]{BLQ_filt}, so we get (3). 
\end{proof}

\subsection{Families of log Fano cone singularities} A log Fano cone singularity is a klt singularity with a good torus action, which generalizes the usual affine cone over a Fano variety. The K-stability of log Fano cone singularities is parallel to that of Fano varieties; see \cite{LX_stability}, \cite{LWX}, \cite{Huang}, \cite{LW_LFC} for an algebraic approach. 

Let $S$ be a reduced scheme. Let $M \simeq \mathbb{Z}^r$ be a free abelian group of finite rank $r$ (i.e., a lattice), and $\mathbb{T}_S = \mathop{\mathrm{Spec}_S}(\mathscr{O}_S[M])$ be the corresponding torus over $S$. 

\begin{defn}
    A \emph{locally stable family of log Fano cone singularities} over $S$ (with \emph{weight lattice} $M$) consists of a locally stable family of affine klt singularities $\pi \colon (X, \Delta) \to S$ with $x \in X(S)$ and an action of $\mathbb{T}_S$ on $(X, \Delta)$ over $S$ such that $x(S)$ is the fixed locus, called the \emph{vertex}. 
    
    Thus $X = \mathop{\mathrm{Spec}_S}(\mathscr{A})$ where
    \begin{equation*}
        \mathscr{A} = \bigoplus_{m \in M} \mathscr{A}_m
    \end{equation*}
    is an $M$-graded $\mathscr{O}_S$-algebra, such that $\mathscr{A}_{+} = \bigoplus_{m \neq 0} \mathscr{A}_m \subset \mathscr{A}$ is an ideal, and $\mathscr{A}/\mathscr{A}_{+} = \mathscr{A}_0 = \mathscr{O}_S$. Note that each $\mathscr{A}_m$ is a flat coherent $\mathscr{O}_S$-module. 

    A \emph{polarization} is $\xi \in N_{\mathbb{R}} \coloneqq \mathrm{Hom}(M, \mathbb{R})$ such that $\langle \xi, m \rangle > 0$ for all $m \in M \smallsetminus \{0\}$ with $\mathscr{A}_m \neq 0$. Let $\mathrm{wt}_{\xi} \in \mathrm{Val}_{X_s}$ denote the $\mathbb{T}_s$-invariant valuation on each fiber $X_s$ given by 
    \begin{equation*}
        \mathrm{wt}_{\xi}(a) = \min \{ \langle \xi, m \rangle : a_m \neq 0 \}, 
    \end{equation*}
    for $a = \sum_{m \in M} a_m \in \bigoplus_{m \in M} \mathscr{A}_m \otimes_{\mathscr{O}_S} \kappa(s)$. 

    The family $\pi \colon (X, \Delta) \to S$ is called a locally stable family of K-semistable Fano cone singularities over $S$ if there is a polarization $\xi$ such that $(X_s, \Delta_s; \xi)$ is a K-semistable polarized log Fano cone singularity for all $s \in S$. Such a polarization $\xi$ is unique up to scaling if it exists (e.g., by \cite[Thm.\ 3.5]{LX_stability}), hence we often omit it. 
\end{defn}

The following is a family version of \cite[\S 3.1]{Zhuang_BoundednessII}. 

\begin{lem} \label{compactify_log_Fano_cone}
    Let $\pi \colon (X = \mathop{\mathrm{Spec}_S}(\mathscr{A}), \Delta) \to S$ be a locally stable family of log Fano cone singularities, where $\mathscr{A} = \bigoplus_{m \in M} \mathscr{A}_m$, and $\xi$ be a polarization. Assume $\xi \in N$, that is, $\langle \xi, m \rangle \in \mathbb{Z}$ for all $m \in M$. Let 
    \begin{equation*}
        \overline{X} \coloneqq \mathop{\mathrm{Proj}_S} \mathscr{A}[t]
    \end{equation*}
    where $\mathscr{A}[t]$ is given the $\mathbb{N}$-grading with $\deg(\mathscr{A}_m) = \langle \xi, m \rangle$ and $\deg(t) = 1$. Then the following hold: 
    \begin{enumerate}[label=\emph{(\arabic*)}, nosep]
        \item There is a canonical open immersion $X \hookrightarrow \overline{X}$, and the complement $V = \overline{X} \smallsetminus X$ is a $\mathbb{Q}$-Cartier relatively ample divisor. 
        \item Let $\overline{\Delta}$ be the closure of $\Delta$, then $\bar{\pi} \colon (\overline{X}, \overline{\Delta} + V) \to S$ is a locally stable family of plt pairs. 
        \item If $S$ is connected, then $K_{\overline{X}/S} + \overline{\Delta} + (1+a)V \sim_{S, \mathbb{Q}} 0$ for some $a > 0$. 
    \end{enumerate}
\end{lem}

\begin{proof}
    (1). The open subset $X \subset \overline{X}$ is the locus $t \neq 0$, so the complement $V$ is the divisor defined by $t$. Note that $\mathscr{O}_{\overline{X}}(V) = \mathscr{O}_{\overline{X}}(1)$, hence $V$ is relatively ample over $S$. 

    We have $V = \mathop{\mathrm{Proj}_S} \mathscr{A}$ where $\mathscr{A}$ is given the $\mathbb{N}$-grading with $\deg(\mathscr{A}_m) = \langle \xi, m \rangle$, and $X$ (resp., $\overline{X}$) is the affine cone (resp., projective cone), relatively over $S$, for $\mathscr{O}_V(1) = \mathscr{O}_{\overline{X}}(1)|_V$. Thus $V \subset \overline{X}$ is Cartier on the locus where $\mathscr{O}_{V}(1)$ is invertible; in particular, $V \subset \overline{X}$ is Cartier in codimension $2$. 
    
    (2). First assume $S$ is the spectrum of a DVR, with the closed point $s \in S$. Then $(X, \Delta + X_s)$ is plt by \cite[Thm.\ 4.54]{Kol_fam}, hence $(V, \Delta_V + V_s)$ is plt, where $\Delta_V = \mathrm{Diff}_{V}(\overline{\Delta}) = \overline{\Delta}|_V$ (see \cite[Lem.\ 3.1]{Kollar_singularity}). Then $(\overline{X}, \overline{\Delta} + V + \overline{X}_s)$ is dlt. Thus the conclusion holds in this case by \cite[Thm.\ 4.54]{Kol_fam} again. Note that the closed fiber $\overline{\Delta}_s$ of $\overline{\Delta}$ is the closure of $\Delta_s$. 

    In general, the formation of $\overline{X}$ commutes with base change. By the DVR case above, the formation of $\overline{\Delta}$ commutes with specialization, hence $\bar{\pi} \colon (\overline{X}, \overline{\Delta} + V) \to S$ is a well-defined family of pairs by \cite[Thm-Def.\ 4.3]{Kol_fam}. Then it is a locally stable family of plt pairs since its base change to every DVR is so by \cite[Def-Thm.\ 4.7]{Kol_fam}. 

    (3). Let 
    \begin{equation*}
        Y = \mathop{\mathrm{Spec}_V} \bigoplus_{n \geq 0} \mathscr{O}_V(n), 
    \end{equation*}
    with the canonical morphisms $p \colon Y \to V$ and $\mu \colon Y \to X$, and $E = \mu^{-1}(x(S)) \simeq V$. Then 
    \begin{equation*}
        \mu \colon (Y, E) \to (X, \Delta)
    \end{equation*}
    is a family of birational models, and we can write 
    \begin{equation*}
        \mu^{*}(K_{X/S} + \Delta) = K_{Y/S} + \mu^{-1}_{*}\Delta + E'
    \end{equation*}
    for some $E'$ on $Y$ supported on $E$. For each fiber we have $E'_s = (1 - A_{X_s,\Delta_s}(E_s))E_s$, and $E_s$ is a prime divisor. Thus the function $s \mapsto A_{X_s, \Delta_s}(E_s)$ is locally constant. If $S$ is connected, then it is constant, say with value $a$. Note that $a > 0$ since $(X_s, \Delta_s)$ is klt for all $s \in S$. Then 
    \begin{equation*}
        K_{\overline{X}_s} + \overline{\Delta}_s + (1 + a)V_s \sim_{\mathbb{Q}} 0
    \end{equation*}
    for all $s \in S$ by \cite[Lem.\ 3.3]{Zhuang_BoundednessII}. Hence $K_{\overline{X}/S} + \overline{\Delta} + (1+a)V \sim_{S, \mathbb{Q}} 0$. 
\end{proof}

\begin{lem} \label{family_of_log_Fano_cone_constant_vol}
    Let $\pi \colon (X = \mathop{\mathrm{Spec}_S}(\mathscr{A}), \Delta) \to S$ be a locally stable family of log Fano cone singularities, where $\mathscr{A} = \bigoplus_{m \in M} \mathscr{A}_m$, and $\xi$ be a polarization. Then the following hold: 
    \begin{enumerate}[label=\emph{(\arabic*)}, nosep]
        \item The function $s \mapsto \mathrm{vol}_{X_s}(\mathrm{wt}_{\xi})$ is locally constant. 
        \item The function $s \mapsto A_{X_s,\Delta_s}(\mathrm{wt}_{\xi})$ is locally constant. 
        \item The function $s \mapsto \widehat{\mathrm{vol}}_{X_s,\Delta_s}(\mathrm{wt}_{\xi})$ is locally constant. 
    \end{enumerate}
    In particular, if $\pi \colon (X, \Delta) \to S$ is a locally stable family of K-semistable log Fano cone singularities with the polarization $\xi$, then the function
    \begin{equation*}
        s \mapsto \widehat{\mathrm{vol}}(x_s; X_s,\Delta_s)
    \end{equation*}
    is locally constant, where $x \colon S \to X$ is the vertex section. 
\end{lem}

\begin{proof}
    Write $n = \dim(X_s)$ for all $s \in S$. 
    
    (1). By definition, 
    \begin{equation*}
        \mathrm{vol}_{X_s}(\mathrm{wt}_{\xi}) = \lim_{\lambda \to \infty} \frac{n!}{\lambda^n} \sum_{\langle \xi, m \rangle < \lambda} \dim_{\kappa(s)}(\mathscr{A}_{m} \otimes_{\mathscr{O}_S} \kappa(s)), 
    \end{equation*}
    where the summation ranges over all $m \in M$ with $\langle \xi, m \rangle < \lambda$. Since $\mathscr{A}_m$ is a flat coherent $\mathscr{O}_S$-module, 
    \begin{equation*}
        s \mapsto \dim_{\kappa(s)}(\mathscr{A}_{m} \otimes_{\mathscr{O}_S} \kappa(s))
    \end{equation*}
    is locally constant for all $m \in M$ (see \cite[\texttt{05P2}]{stacks-project}). Hence $s \mapsto \mathrm{vol}_{X_s}(\mathrm{wt}_{\xi})$ is locally constant. 

    (2). Fix $s \in S$, then $\xi \mapsto A_{X_s,\Delta_s}(\mathrm{wt}_{\xi})$ is a linear map on $N_{\mathbb{R}}$ by \cite[Lem.\ 2.18]{LX_stability}. Hence it suffices to show $s \mapsto A_{X_s,\Delta_s}(\mathrm{wt}_{\xi})$ is locally constant when $\xi \in N$. We may assume that $S$ is connected. Let 
    \begin{equation*}
        \bar{\pi} \colon (\overline{X}, \overline{\Delta} + V) \to S
    \end{equation*}
    be the family constructed in Lemma \ref{compactify_log_Fano_cone} using $\xi$. Then 
    \begin{equation*}
        K_{\overline{X}/S} + \overline{\Delta} + (1 + a)V \sim_{S, \mathbb{Q}} 0
    \end{equation*}
    for some $a > 0$. Restricting to each fiber, we get $a = A_{X_s, \Delta_s}(\mathrm{wt}_{\xi})$ for all $s \in S$ by \cite[Lem.\ 3.3]{Zhuang_BoundednessII}. 

    (3). By definition, 
    \begin{equation*}
        \widehat{\mathrm{vol}}_{X_s,\Delta_s}(\mathrm{wt}_{\xi}) = A_{X_s,\Delta_s}(\mathrm{wt}_{\xi})^n \cdot \mathrm{vol}_{X_s}(\mathrm{wt}_{\xi}). 
    \end{equation*}
    Hence it is a locally constant function on $S$ by (1) and (2). 

    If $\pi \colon (X, \Delta) \to S$ is a family of K-semistable log Fano cone singularities with the polarization $\xi$, then 
    \begin{equation*}
        \widehat{\mathrm{vol}}(x_s; X_s,\Delta_s) = \widehat{\mathrm{vol}}_{X_s,\Delta_s}(\mathrm{wt}_{\xi})
    \end{equation*}
    by \cite[Thm.\ 3.5]{LX_stability}. Hence $s \mapsto \widehat{\mathrm{vol}}(x_s; X_s,\Delta_s)$ is locally constant by (3). 
\end{proof}

\section{Special valuations} \label{sec: special_val}

\subsection{Models of qdlt Fano type}

\begin{defn}[{\cite[Def.\ 3.5, 3.7]{XZ_stable_deg}}]
    Let $(X, \Delta)$ be an affine klt pair. A model 
    \begin{equation*}
        f \colon (Y, E = E_1 + \cdots + E_r) \to (X, \Delta)
    \end{equation*}
    is said to be of \emph{qdlt Fano type} if there is an effective $\mathbb{Q}$-divisor $D$ on $Y$ such that $\lfloor D \rfloor = 0$, $D \geq f^{-1}_{*}\Delta$, $(Y, D + E)$ is qdlt, and $-(K_Y + D + E)$ is $f$-ample. 
    
    The model $f \colon (Y, E) \to (X, \Delta)$ is called a \emph{qdlt anti-canonical model} if one can take $D = f^{-1}_{*}\Delta$ above, that is, $(Y, f^{-1}_{*}\Delta + E)$ is qdlt and $-(K_Y + f^{-1}_{*}\Delta + E)$ is $f$-ample. 

    Let $x \in (X, \Delta)$ be a klt singularity. A qdlt anti-canonical model $f \colon (Y, E) \to (X, \Delta)$ is called a \emph{Koll\'ar model} at $x$ if all components of $E$ are centered at $x$. 
\end{defn}

\begin{lem}[{\cite[Prop.\ 3.9]{XZ_stable_deg}}] \label{ample_model_of_qdltFano_is_Kollar}
    Let $(X, \Delta)$ be an affine klt pair, and $f \colon (Y, E) \to (X, \Delta)$ be a model of qdlt Fano type. Then there exists a birational contraction $\phi \colon Y \dashrightarrow \overline{Y}$ over $X$ that is a local isomorphism at generic points of strata of $(Y, E)$, such that $\bar{f} \colon (\overline{Y}, \overline{E} = \phi_{*}E) \to (X, \Delta)$ is a qdlt anti-canonical model, and $\mathrm{QM}(Y, E) = \mathrm{QM}(\overline{Y}, \overline{E})$. 
\end{lem}

\begin{rem}
    Let $f \colon (Y, E = \sum_{i=1}^{r} E_i) \to (X, \Delta)$ be a model of qdlt Fano type. Then the anti-canonical model $\bar{f} \colon (\overline{Y}, \overline{E} = \sum_{i=1}^{r} \overline{E}_i) \to (X, \Delta)$ is uniquely determined by the divisorial valuations $\mathrm{ord}_{\overline{E}_i} = \mathrm{ord}_{E_i}$. In fact, since 
    \begin{equation*}
        -(K_{\overline{Y}} + \bar{f}^{-1}_{*}\Delta + \overline{E}) \sim_{X, \mathbb{Q}} -\sum_{i=1}^{r} a_i\overline{E}_i, 
    \end{equation*}
    is $\bar{f}$-ample, where $a_i = A_{X,\Delta}(\overline{E}_i) = A_{X, \Delta}(\mathrm{ord}_{E_i})$, if $\ell \in \mathbb{Z}_{>0}$ such that $\ell a_i \in \mathbb{Z}$ for all $i$, then 
    \begin{equation*}
        \overline{Y} = \mathop{\mathrm{Proj}_X} \bigoplus_{m \in \mathbb{N}} \bar{f}_{*}\mathscr{O}_{\overline{Y}}\left( -\sum_{i=1}^{r}m\ell a_i \overline{E}_i \right), 
    \end{equation*}
    and 
    \begin{equation*}
        \bar{f}_{*}\mathscr{O}_{\overline{Y}}\left( -\sum_{i=1}^{r}m\ell a_i \overline{E}_i \right) = f_{*}\mathscr{O}_Y\left( -\sum_{i=1}^{r}m\ell a_i E_i \right) = \bigcap_{i=1}^{r} \mathfrak{a}_{m\ell a_i}(\mathrm{ord}_{E_i})
    \end{equation*}
    only depend on the valuations $\mathrm{ord}_{E_i}$. 
\end{rem}

\begin{lem} \label{minimal_model_with_given_exceptional_is_qdlt_Fano}
    Let $(X, \Delta)$ be an affine klt pair, and $f \colon (Y, E = \sum_{i=1}^{r} E_i) \to (X, \Delta)$ be a model of qdlt Fano type. Suppose $f' \colon (Y', E' = \sum_{i=1}^{r} E_i') \to (X, \Delta)$ is a model such that $E'$ contains all $f'$-exceptional divisors, $\mathrm{ord}_{E_i'} = \mathrm{ord}_{E_i}$ for every $i$, and $-(K_{Y'} + f'^{-1}_{*}\Delta + E')$ is $f'$-nef. Then $f' \colon (Y', E') \to (X, \Delta)$ is a model of qdlt Fano type, and the birational map $\phi \colon Y \dashrightarrow Y'$ is a local isomorphism at all generic points of strata of $(Y, E)$ and $(Y', E')$. 
\end{lem}

\begin{proof}
     By Lemma \ref{ample_model_of_qdltFano_is_Kollar}, we may assume that $-(K_Y + f^{-1}_{*}\Delta + E)$ is ample. Since $Y$ and $Y'$ have the same exceptional divisors over $X$, the birational map $\phi^{-1} \colon Y' \dashrightarrow Y$ is an isomorphism in codimension $1$, so it is the relative ample model for $-(K_{Y'} + f'^{-1}_{*}\Delta + E')$. Since $-(K_{Y'} + f'^{-1}_{*}\Delta + E')$ is $f'$-nef, $\phi^{-1} \colon Y' \to Y$ is a morphism. Note that $Y$ is $\mathbb{Q}$-factorial at the generic points of strata of $(Y, E)$, so $\phi \colon Y \dashrightarrow Y'$ is a local isomorphism at these points, as otherwise $\phi^{-1}$ would have an exceptional divisor over there. Thus $f' \colon (Y', E') \to (X, \Delta)$ is a model of qdlt Fano type by \cite[Prop.\ 3.6]{XZ_stable_deg}. 
\end{proof}

\begin{defn}[{\cite[Def.\ 3.12]{XZ_stable_deg}}] \label{def:family_of_Kol_model}
    Let $S$ be a reduced scheme, and $\pi \colon (X, \Delta) \to S$ with $x \in X(S)$ be a locally stable family of klt singularities. A family of models $f \colon (Y, E) \to (X, \Delta)$ at $x$ is called a \emph{locally stable family of Koll\'ar models at $x$} if $\pi \circ f \colon (Y, f^{-1}_{*}\Delta + E) \to S$ is a locally stable family with qdlt fibers, and $-(K_Y + f^{-1}_{*}\Delta + E)$ is $f$-ample. 

    Note that the fiber $f_s \colon (Y_s, E_s) \to (X_s, \Delta_s)$ is a Koll\'ar model at $x_s$ for every $s \in S$. 
\end{defn}

\begin{rem} \label{local_stability_over_regular_local}
    Let $S$ be a regular local scheme of dimension $d$, with the generic point $\eta \in S$ and the closed point $s \in S$. Let $H = \sum_{j=1}^{d} H_j$ be an snc divisor on $S$ defined by a regular system of parameters. 
    
    Suppose $(X, \Delta)$ is a pair, and $\pi \colon X \to S$ is a dominant morphism. Then $\pi \colon (X, \Delta) \to S$ is a locally stable family if and only if $(X, \Delta + \pi^{*}H)$ is slc; see \cite[Thm.\ 4.54]{Kol_fam}. By the inversion of adjunction, $\pi \colon (X, \Delta) \to S$ has qdlt (resp., klt) fibers if and only if $(X, \Delta + \pi^{*}H)$ is qdlt (resp., qdlt with $\lfloor \Delta \rfloor = 0$). 

    Suppose $\pi \colon (X, \Delta) \to S$ with $x \in X(S)$ is a locally stable family of klt singularities. Then a locally stable family of Koll\'ar models 
    \begin{equation*}
        f \colon (Y, E) \to (X, \Delta)
    \end{equation*}
    is a model for $(X, \Delta)$ such that each component of $E$ dominates $S$, $(Y, f^{-1}_{*}\Delta + E + f^{*}\pi^{*}H)$ is qdlt, and $-(K_Y + f^{-1}_{*}\Delta + E)$ is $f$-ample. In other words, 
    \begin{equation*}
        f \colon (Y, E + f^{*}\pi^{*}H) \to (X, \Delta)
    \end{equation*}
    is a qdlt anti-canonical model for $(X, \Delta)$. 
\end{rem}

\subsection{Special valuations and complements}

\begin{defn}
    Let $(X, \Delta)$ be a pair. A subset $\sigma \subset \mathrm{Val}_X$ is called a \emph{quasi-monomial (simplicial) cone} for $(X, \Delta)$ if $\sigma = \mathrm{QM}_y(Y, E)$, where $f \colon (Y, E) \to (X, \Delta)$ is a model, $E$ contains all $f$-exceptional divisors, and $y \in (Y, E)$ is a generic point of a stratum such that $(Y, f^{-1}_{*}\Delta + E)$ is simple-toroidal at $y$. 

    Note that by taking a log resolution $(Y', E') \to (Y, f^{-1}_{*}\Delta + E)$ that is a local isomorphism over $y$, we may assume that $(Y, E)$ is a simple-toroidal model. 
    
    If $E_1, \ldots, E_r$ are the components of $E$ that contain $y$, then there is an isomorphism $\sigma \simeq \mathbb{R}_{\geq 0}^r$. If $\tau \subset \sigma$ corresponds to a rational simplicial cone in $\mathbb{R}_{\geq 0}^{r}$, then $\tau \subset \mathrm{Val}_X$ is also a quasi-monomial simplicial cone for $(X, \Delta)$. If $K_X + \Delta$ is $\mathbb{Q}$-Cartier, then the log discrepancy function $A_{X,\Delta}$ is linear on $\sigma$. 
\end{defn}

\begin{defn}
    Let $(X, \Delta)$ be an affine klt pair, and $f \colon (Y, E) \to (X, \Delta)$ be a simple-toroidal model. A $\mathbb{Q}$-complement $\Gamma$ of $(X, \Delta)$ is said to be \emph{special with respect to} $f \colon (Y, E) \to (X, \Delta)$ if $f^{-1}_{*}\Gamma \geq G$ for some effective $f$-ample $\mathbb{Q}$-divisor $G$ on $Y$ whose support does not contain any stratum of $(Y, E)$. 

    A quasi-monomial simplicial cone $\sigma \subset \mathrm{Val}_X$ for $(X, \Delta)$ is said to be \emph{special} if there is a simple-toroidal model $f \colon (Y, E) \to (X, \Delta)$, a special $\mathbb{Q}$-complement $\Gamma$ of $(X, \Delta)$ with respect to $f \colon (Y, E) \to (X, \Delta)$, and a generic point $y$ of a stratum of $(Y, E)$, such that $\sigma \subset \mathrm{QM}_y(Y, E) \cap \mathrm{LCP}(X, \Delta + \Gamma)$. 

    A real valuation $v \in \mathrm{Val}_X$ is said to be \emph{special} for $(X, \Delta)$ if $v \in \sigma$ for some special quasi-monomial simplicial cone $\sigma \subset \mathrm{Val}_X$. In particular, $v$ is quasi-monomial. 
\end{defn}

\begin{lem} \label{special_cone_give_qdlt_Fano_type_model}
    Let $(X, \Delta)$ be an affine klt pair, and $\sigma \subset \mathrm{Val}_X$ be a quasi-monomial cone for $(X, \Delta)$. Then the following are equivalent: 
    \begin{enumerate}[label=\emph{(\roman*)}, nosep]
        \item $\sigma$ is special. 
        \item there exists a model $f \colon (Y, E) \to (X, \Delta)$ of qdlt Fano type such that $\sigma = \mathrm{QM}(Y, E)$; 
    \end{enumerate}
\end{lem}

\begin{proof}
    (i) $\Rightarrow$ (ii). See \cite[Thm.\ 3.14]{XZ_stable_deg} or \cite[Thm.\ 5.26]{Xu_K-Stability_Book}. 

    (ii) $\Rightarrow$ (i). By Lemma \ref{ample_model_of_qdltFano_is_Kollar}, we may assume $(Y, f^{-1}_{*}\Delta + E)$ is qdlt and $-(K_Y + f^{-1}_{*}\Delta + E)$ is ample. Choose a general $\mathbb{Q}$-divisor $G \in |{-(K_Y + f^{-1}_{*}\Delta + E)}|_{\mathbb{Q}}$ such that $(Y, f^{-1}_{*}\Delta + G + E)$ is qdlt. Note that $f_{*}(G + E) \sim_{\mathbb{Q}} -(K_X + \Delta)$, where $f_{*}E$ consists of components of $E$ that are not $f$-exceptional. Let 
    \begin{equation*}
        \mu \colon (Z, F) \to (Y, f^{-1}_{*}\Delta + G + E)
    \end{equation*}
    be a log resolution such that $\mu$ is a local isomorphism over the simple-toroidal locus of $(Y, E)$, and there is a $\mu$-ample divisor $-A$ on $Z$ with $A \geq 0$ and $\mathrm{Supp}(A) = \mathrm{Ex}(\mu)$. Then we have 
    \begin{equation*}
            0 \sim_{\mathbb{Q}} \mu^{*}(K_Y + f^{-1}_{*}\Delta + G + E) = K_Z + \mu^{-1}_{*}f^{-1}_{*}\Delta + (\mu^{*}G - \epsilon A) + \mu^{-1}_{*}E + F' + \epsilon A. 
    \end{equation*}
    where $F' = \sum_{i} (1 - A_{Y, f^{-1}_{*}\Delta + E}(F_i)) \cdot F_i$ and $F_i$ ranges over all $\mu$-exceptional divisors. By our choice of $\mu$, we have $\lfloor F' \rfloor \leq 0$. Thus $\lfloor F' + \epsilon A \rfloor \leq 0$, and $\mu^{*}G - \epsilon A$ is ample for $0 < \epsilon \ll 1$. Now 
    \begin{equation*}
        \mathrm{LCP}(Z, \mu^{-1}_{*}f^{-1}_{*}\Delta + \mu^{-1}_{*}E + F' + \epsilon A) = \mathrm{QM}(Z, \mu^{-1}_{*}E) = \mathrm{QM}(Y, E). 
    \end{equation*}
    Choose a general $H \in |\mu^{*}G - \epsilon A|_{\mathbb{Q}}$, such that $(Z, \mu^{-1}_{*}f^{-1}_{*}\Delta + H + \mu^{-1}_{*}E + F' + \epsilon A)$ is sub-lc and 
    \begin{equation*}
        \mathrm{LCP}(Z, \mu^{-1}_{*}f^{-1}_{*}\Delta + H + \mu^{-1}_{*}E + F' + \epsilon A) = \mathrm{LCP}(Z, \mu^{-1}_{*}f^{-1}_{*}\Delta + \mu^{-1}_{*}E + F' + \epsilon A). 
    \end{equation*}
    Let $\Gamma = f_{*}\mu_{*}H + f_{*}E \sim_{\mathbb{Q}} f_{*}(G + E) \sim_{\mathbb{Q}} -(K_X + \Delta)$, then $\mu^{-1}_{*}f^{-1}_{*}\Gamma = H$, and
    \begin{equation*}
        \mu^{*}f^{*}(K_X + \Delta + \Gamma) = K_Z + \mu^{-1}_{*}f^{-1}_{*}\Delta + H + \mu^{-1}_{*}E + F' + \epsilon A. 
    \end{equation*}
    Hence $\Gamma$ is a special $\mathbb{Q}$-complement with respect to $(Z, F) \to (X, \Delta)$, and $\mathrm{LCP}(X, \Delta + \Gamma) = \sigma$. 
\end{proof}

\begin{lem} \label{special_val_equiv_conditions}
    Let $x \in (X, \Delta)$ be a klt singularity. Suppose $v \in \mathrm{Val}_{X,x}^{\mathrm{qm}}$, and let $\mathfrak{a}_{\bullet} = \mathfrak{a}_{\bullet}(v)$ denote the associated ideal sequence. Then the following are equivalent: 
    \begin{enumerate}[label=\emph{(\roman*)}, nosep]
        \item $v$ is special.  
        \item There exists a model $f \colon (Y, E) \to (X, \Delta)$ at $x$ of qdlt Fano type such that $v \in \mathrm{QM}(Y, E)$; 
        \item The associated graded ring 
        \begin{equation*}
            R_v = \mathrm{gr}_v(\mathscr{O}_{X}) = \bigoplus_{\lambda \geq 0} \mathfrak{a}_{\lambda}/\mathfrak{a}_{>\lambda}
        \end{equation*}
        is a finitely generated graded $\kappa(x)$-algebra, and the pair $(X_v = \mathop{\mathrm{Spec}(R_v)}, \Delta_v)$ is klt, where $\Delta_v$ is defined by the divisorial part of the initial ideal of $\Delta$. 
        \item There exists $\delta > 0$ such that if $D$ is an effective $\mathbb{Q}$-Cartier $\mathbb{Q}$-divisor on $X$ with $v(D) < \delta$, then $(X, \Delta + D)$ is klt and $v$ is special for $(X, \Delta + D)$. 
        \item There exists $\delta > 0$ such that if $D$ is an effective $\mathbb{Q}$-Cartier $\mathbb{Q}$-divisor on $X$ with $v(D) < \delta$, then $(X, \Delta + D)$ is klt and $v$ computes the lc threshold $\mathrm{lct}(X, \Delta + D; \mathfrak{a}_{\bullet})$. 
        \item For every effective $\mathbb{Q}$-Cartier $\mathbb{Q}$-divisor $D$ on $X$, there exists a $\mathbb{Q}$-complement $\Gamma$ of $(X, \Delta)$ such that $\Gamma \geq \epsilon D$ for some $\epsilon > 0$, and $v \in \mathrm{LCP}(X, \Delta + \Gamma)$. 
    \end{enumerate}
\end{lem}

\begin{proof}
    (i), (ii), and (iii) are equivalent by \cite[Thm.\ 4.1]{XZ_stable_deg}. 

    (i-iii) $\Rightarrow$ (iv). The valuation $v$ induces a canonical valuation $\mathrm{wt}_v$ on $X_v$, such that $\mathrm{wt}_v(D_v) = v(D)$ for every effective $\mathbb{Q}$-Cartier $\mathbb{Q}$-divisor $D$ on $X$. By the Izumi type inequality \cite[Lem.\ 3.5]{Zhuang_boundednessI}, 
    \begin{equation*}
        \mathrm{lct}(X_v, \Delta_v; D_v) \geq c_0\frac{\widehat{\mathrm{vol}}(x_v; X_v, \Delta_v)}{\widehat{\mathrm{vol}}_{X_v, \Delta_v}(\mathrm{wt}_v)} \cdot \frac{A_{X_v, \Delta_v}(\mathrm{wt}_v)}{\mathrm{wt}_v(D_v)}, 
    \end{equation*}
    where $x_v \in X_v$ is the vertex defined by the irrelevant ideal of $R_v$, and $c_0$ is a constant only depending on $\dim_x(X) = \dim(X_v)$. Thus, if 
    \begin{equation*}
        v(D) = \mathrm{wt}_v(D_v) < \delta \coloneqq c_0\frac{\widehat{\mathrm{vol}}(x_v; X_v, \Delta_v)}{\widehat{\mathrm{vol}}_{X_v, \Delta_v}(\mathrm{wt}_v)} A_{X_v, \Delta_v}(\mathrm{wt}_v), 
    \end{equation*}
    then $(X_v, \Delta_v + D_v)$ is klt. Since we know (iii) $\Rightarrow$ (i), this implies that $v$ is special for $(X, \Delta + D)$. 

    (iv) $\Rightarrow$ (v). It suffices to show that if $v$ is special for $(X, \Delta)$, then $v$ computes the lc threshold of $\mathfrak{a}_{\bullet}$. Suppose $v$ is special, then $v \in \mathrm{LCP}(X, \Delta + \Gamma)$ for some $\mathbb{Q}$-complement $\Gamma$. Hence we get the conclusion by \cite[Thm.\ 7.8]{JM}\footnote{By \cite[Thm.\ 7.8]{JM}, if a valuation $v$ computes the Arnold multiplicity $\mathrm{Arn} = 1/\mathrm{lct}$ of an ideal sequence, then $v$ also computes that of the ideal sequence $\mathfrak{a}_{\bullet}(v)$ associated with itself; note that computing $\mathrm{Arn}$ is equivalent to computing $\mathrm{lct}$. In our case, $v$ computes the $\mathrm{lct}$ of $\Gamma$, which is the same as the $\mathrm{lct}$ of $\{\mathscr{O}_X(-\lceil \lambda \Gamma \rceil)\}_{\lambda \geq 0}$. The result there is only stated on regular schemes with no boundary divisor, but the proof is the same as the one we give below.}. We also give a proof here: For $w \in \mathrm{Val}_{X,x}$, we have 
    \begin{equation*}
        A_{X,\Delta}(w) - w(\Gamma) = A_{X,\Delta+\Gamma}(w) \geq 0
    \end{equation*}
    since $(X, \Delta + \Gamma)$ is lc, and $w(\mathfrak{a}_{\bullet}) \leq w(\Gamma)/v(\Gamma)$. Meanwhile, $A_{X,\Delta}(v) = v(\Gamma)$ since $v \in \mathrm{LCP}(X, \Delta + \Gamma)$. Thus 
    \begin{equation*}
        \frac{A_{X,\Delta}(w)}{w(\mathfrak{a}_{\bullet})} \geq \frac{A_{X,\Delta}(w)}{w(\Gamma)} v(\Gamma) \geq v(\Gamma) = A_{X, \Delta}(v) = \frac{A_{X,\Delta}(v)}{v(\mathfrak{a}_{\bullet})}. 
    \end{equation*}
    Taking infimum over all $w \in \mathrm{Val}_{X,x}$, we get $\mathrm{lct}(X, \Delta; \mathfrak{a}_{\bullet}) = A_{X,\Delta}(v)/v(\mathfrak{a}_{\bullet})$. 

    (v) $\Rightarrow$ (vi). Let $\epsilon < \delta/v(D)$, so $(X, \Delta + \epsilon D)$ is klt and $v$ computes the lc threshold $\mathrm{lct}(X, \Delta + \epsilon D; \mathfrak{a}_{\bullet})$ by (v). Then the conclusion follows from the fact that a quasi-monomial valuation that computes the lc threshold of an ideal sequence is an lc place of a $\mathbb{Q}$-complement; see \cite[Rmk.\ 4.4]{LX_stability}. This fact is also the special case of Lemma \ref{specialize_val_computes_constant_lct} when the base $S$ is a point. 

    (vi) $\Rightarrow$ (i). The proof is the same as \cite[Lem.\ 3.4]{XZ_stable_deg}. Let $f \colon (Y, E) \to (X, \Delta)$ be a simple-toroidal model, and $y \in (Y, E)$ be a generic point of a stratum, such that $v \in \mathrm{QM}_y(Y, E)$, and there is an $f$-ample divisor $-A$ with $A \geq 0$ and $\mathrm{Supp}(A) = \mathrm{Ex}(f)$. Then $-A$ is ample since $X$ is affine. Let $H \in |{-A}|_{\mathbb{Q}}$ be a general $\mathbb{Q}$-divisor whose support does not contain any stratum of $(Y, \mathrm{Supp}(f^{-1}_{*}\Delta + E + \mathrm{Ex}(f)))$. Let $D = f_{*}H$, then there is a $\mathbb{Q}$-complement $\Gamma \geq \epsilon D$ of $(X, \Delta)$ such that $v \in \mathrm{LCP}(X, \Delta + \Gamma)$ by (vi). By construction, $f^{-1}_{*}\Gamma \geq \epsilon H$. That is, $\Gamma$ is a special $\mathbb{Q}$-complement with respect to $f \colon (Y, E) \to (X, \Delta)$. 
\end{proof}

\begin{rem} \label{estimate_special_threshold}
    We have $A_{X_v, \Delta_v}(\mathrm{wt}_v) = A_{X,\Delta}(v)$ and $\widehat{\mathrm{vol}}_{X_v, \Delta_v}(\mathrm{wt}_v) = \widehat{\mathrm{vol}}_{X, \Delta}(v)$ by \cite[Lem.\ 4.10]{LX_stability}. Thus we can take 
    \begin{equation*}
        \delta = c_0\frac{\widehat{\mathrm{vol}}(x_v; X_v, \Delta_v)}{\widehat{\mathrm{vol}}_{X, \Delta}(v)}A_{X,\Delta}(v)
    \end{equation*}
    in (iv) and (v). Moreover, if $v$ is the minimizer of the normalized volume function for $x \in (X, \Delta)$, then it is proved that one can take $\delta = A_{X,\Delta}(v)/(\dim X)$ in \cite[Lem.\ 3.3]{XZ_stable_deg} using K-stability of valuations, hence $v$ is special. 
\end{rem}

\begin{lem} \label{special_val_computes_lct_of_itself_uniquely}
    Let $x \in (X, \Delta)$ be a klt singularity. Suppose $v \in \mathrm{Val}_{X,x}$ is a special valuation for $(X, \Delta)$, and $\mathfrak{a}_{\bullet} = \mathfrak{a}_{\bullet}(v)$ is the ideal sequence associated to $v$. Then, up to scaling, $v$ is the unique valuation that computes the lc threshold $\mathrm{lct}(X, \Delta; \mathfrak{a}_{\bullet})$. 
\end{lem}

\begin{proof}
    We know $v$ computes the lc threshold of $\mathfrak{a}_{\bullet}$ by Lemma \ref{special_val_equiv_conditions}. Suppose $w \in \mathrm{Val}_{X,x}$ such that 
    \begin{equation*}
        \frac{A_{X,\Delta}(v)}{v(\mathfrak{a}_{\bullet})} = \mathrm{lct}(X, \Delta; \mathfrak{a}_{\bullet}) = \frac{A_{X,\Delta}(w)}{w(\mathfrak{a}_{\bullet})}. 
    \end{equation*}
    By rescaling, we may assume that $A_{X,\Delta}(v) = A_{X,\Delta}(w) = 1$. Hence 
    \begin{equation*}
        \inf_{\lambda > 0} \frac{w(\mathfrak{a}_{\lambda}(v))}{\lambda} = w(\mathfrak{a}_{\bullet}) = v(\mathfrak{a}_{\bullet}) = 1. 
    \end{equation*}
    Thus $w(f) \geq v(f)$ for every $f \in \mathscr{O}_{X}$. Assume that $w(f) > v(f)$ for some $f$, then $f \in \mathfrak{m}_x \smallsetminus \{0\}$, where $\mathfrak{m}_x \subset \mathscr{O}_{X}$ is the maximal ideal of $x$. Let $D = \mathrm{div}(f)$. Then, for $\epsilon > 0$, 
    \begin{equation*}
        \frac{A_{X, \Delta + \epsilon D}(w)}{w(\mathfrak{a}_{\bullet})} = 1 - \epsilon w(f) < 1 - \epsilon v(f) = \frac{A_{X, \Delta + \epsilon D}(v)}{v(\mathfrak{a}_{\bullet})}. 
    \end{equation*}
    This contradicts to that $v$ computes $\mathrm{lct}(X, \Delta + \epsilon D; \mathfrak{a}_{\bullet})$ for $0 < \epsilon \ll 1$ by Lemma \ref{special_val_equiv_conditions}. 
\end{proof}

\begin{rem}
    If $v = \mathrm{ord}_E$ is divisorial, and $f \colon (Y, E) \to (X, \Delta)$ is the prime blow-up where $-E$ is $f$-ample. Then the lemma follows from that $(Y, f^{-1}_{*}\Delta + E)$ is plt. 
\end{rem}

\subsection{Models of qdlt Fano type with ample exceptional divisors}

\begin{lem} \label{flat_Rees}
    Let $A$ be a ring. Suppose $M$ is an $A$-module, and $M_m \subset M$ is a sub-$A$-module for every $m = (m_1, \ldots, m_r) \in \mathbb{Z}^r$, such that $M_0 = M$ and $M_{m'} \subset M_m$ whenever $m' \geq m$, where 
    \begin{equation*}
        m' = (m_1', \ldots, m'_r) \geq (m_1, \ldots, m_r) = m
    \end{equation*}
    if and only if $m_i' \geq m_i$ for all $i = 1, \ldots, r$. Let $R = A[t_1, \ldots, t_r]$, and 
    \begin{equation*}
        N = \bigoplus_{m \in \mathbb{Z}^r} M_{m}t_1^{-m_1} \cdots t_r^{-m_r} \subset M[t_1^{\pm 1}, \ldots, t_r^{\pm 1}]. 
    \end{equation*}
    Assume $N$ is a flat $R$-module. Then for each $\xi = (\xi_1, \ldots, \xi_r) \in (\mathbb{R}_{>0})^r$ and $\lambda \in \mathbb{R}$, there is a canonical isomorphism 
    \begin{equation*}
        \frac{\sum_{\langle \xi, m \rangle \geq \lambda} M_m}{\sum_{\langle \xi, m \rangle > \lambda} M_m} \simeq \bigoplus_{\langle \xi, m \rangle = \lambda} M_m/M_{>m}, 
    \end{equation*}
    where $\langle \xi, m \rangle = \sum_{i=1}^{r} \xi_i m_i$, and $M_{>m} = \sum_{m' > m} M_{m'}$. 
\end{lem}

\begin{proof}
    Let $I_a = Rt_1^{a_1} \cdots t_r^{a_r} \subset R$ for each $a = (a_1, \ldots, a_r) \in \mathbb{N}^r$. Then
    \begin{equation*}
        I_a \otimes_R N \simeq I_aN = \bigoplus_{m \in \mathbb{Z}^r} M_{a+m}t_1^{-m_1} \cdots t_r^{-m_r} \subset N. 
    \end{equation*}
    Note that $I_a \cap I_b = I_{\max(a, b)}$, where $\max(a, b) = (\max(a_1, b_1), \ldots, \max(a_r,b_r))$. Since $N$ is flat over $R$, we have $I_aN \cap I_bN = I_{\max(a, b)}N$, that is, 
    \begin{equation*}
        M_{k} \cap M_{\ell} = M_{\max(k,\ell)}
    \end{equation*}
    for all $k, \ell \in \mathbb{Z}^r$. Since $M_0 = M$, we have $M_m = M_{\max(m,0)}$ for all $m \in \mathbb{Z}^r$. Hence it suffices to prove 
    \begin{equation*}
        \frac{\sum_{\langle \xi, a \rangle \geq \lambda} M_a}{\sum_{\langle \xi, a \rangle > \lambda} M_a} \simeq \bigoplus_{\langle \xi, a \rangle = \lambda} M_a/M_{>a}, 
    \end{equation*}
    where instead of $m \in \mathbb{Z}^r$, the summations only range over $a \in \mathbb{N}^r$. 

    Note that there is a canonical isomorphism 
    \begin{equation*}
        \frac{\sum_{\langle \xi, a \rangle \geq \lambda} I_a}{\sum_{\langle \xi, a \rangle > \lambda} I_a} \simeq \bigoplus_{\langle \xi, a \rangle = \lambda} At_1^{a_1} \cdots t_r^{a_r} \simeq \bigoplus_{\langle \xi, a \rangle = \lambda} I_a/I_{>a}. 
    \end{equation*}
    Since $N$ is flat over $R$, we get 
    \begin{equation*}
        \frac{\sum_{\langle \xi, a \rangle \geq \lambda} I_aN}{\sum_{\langle \xi, a \rangle > \lambda} I_aN} \simeq \bigoplus_{\langle \xi, a \rangle = \lambda} I_aN/I_{>a}N. 
    \end{equation*}
    Taking the degree $0$ component, we get the desired conclusion. 
\end{proof}

\begin{lem} \label{qdlt_Fano_type_with_given_ample_exc}
    Let $x \in (X, \Delta)$ be a klt singularity, and let $f \colon (Y, E = \sum_{i=1}^{r} E_i) \to (X, \Delta)$ be a model of qdlt Fano type at $x$. Then there exists an open convex cone $W \subset \mathbb{R}_{> 0}^r$ satisfying the following condition: For every $a = (a_1, \ldots, a_r) \in W$, there exists a birational contraction $\phi \colon Y \dashrightarrow Y'$ over $X$ that is a local isomorphism at the generic point of every stratum of $(Y, E)$ such that $f' \colon (Y', E' = \phi_{*}E) \to (X, \Delta)$ is a model of qdlt Fano type, and $-\sum_{i=1}^{r} a_iE_i'$ is $f'$-ample. 
\end{lem}

\begin{proof}
    By Lemma \ref{ample_model_of_qdltFano_is_Kollar}, we may assume $f \colon (Y, E) \to (X, \Delta)$ is the Koll\'ar model, that is, $(Y, f^{-1}_{*}\Delta + E)$ is qdlt and $-(K_Y + f^{-1}_{*}\Delta + E)$ is ample. For each $m = (m_1, \ldots, m_r) \in \mathbb{Z}^r$, let 
    \begin{equation*}
        I_m = f_{*}\mathscr{O}_Y\left(- \sum_{i=1}^{r} m_iE_i \right) \subset \mathscr{O}_X. 
    \end{equation*}
    Let $R_m = I_m/I_{>m}$ and $R = \bigoplus_{m \in \mathbb{N}^r} R_m$ be the associated graded ring. By \cite[\S4.2]{XZ_stable_deg}, $R$ is a finitely generated integral domain over $R_0 = \kappa(x)$, and $\mathrm{Spec}(R)$ is the equivariant degeneration of $x \in (X, \Delta)$ induced by the Koll\'ar model $(Y, E)$. Let $\overline{W} \subset \mathbb{R}_{\geq 0}^r$ be the closed convex cone generated by all $m \in \mathbb{N}^r$ such that $R_m \neq 0$, and let $W = \mathrm{int}(\overline{W})$ be its interior. Note that for all $a \in \overline{W} \cap \mathbb{Q}^r$, we have $R_{ka} \neq 0$ for all sufficiently divisible integers $k$ since $R$ is finitely generated and integral. We claim $\dim(\overline{W}) = r$, hence $W$ is a non-empty open subset in $\mathbb{R}_{>0}^r$. 

    For $\xi \in \mathbb{R}_{\geq 0}^r$, let $v_{\xi} \in \mathrm{QM}(Y, E) \subset \mathrm{Val}_{X,x}$ be the corresponding quasi-monomial valuation. Then its ideal sequence is 
    \begin{equation*}
        \mathfrak{a}_{\lambda}(v_{\xi}) = \sum_{\langle \xi, m \rangle \geq \lambda} I_m \subset \mathscr{O}_X
    \end{equation*}
    by \cite[Cor.\ 4.10]{XZ_stable_deg}. Since the extended Rees algebra $\bigoplus_{m \in \mathbb{Z}^r} I_mt_1^{-m_1} \cdots t_r^{-m_r}$ is flat over $\Bbbk[t_1, \ldots, t_r]$ by \cite[\S4.2]{XZ_stable_deg}, where $\Bbbk$ is the base field, we have 
    \begin{equation*}
        \mathfrak{a}_{\lambda}(v_{\xi})/\mathfrak{a}_{>\lambda}(v_{\xi}) \simeq \bigoplus_{\langle \xi, m \rangle = \lambda} R_m
    \end{equation*}
    by Lemma \ref{flat_Rees}. If we choose $\xi$ such that $\xi_1, \ldots, \xi_r$ are linearly independent over $\mathbb{Q}$, then $\mathrm{rat.rank}(v_{\xi}) = r$; on the other hand, the value group $\Gamma_{v_{\xi}}$ is generated by all $\langle \xi, m \rangle$ for $R_m \neq 0$. Hence we conclude that the linear subspace spanned by all such $m$ has dimension $r$. 

    Suppose $a = (a_1, \ldots, a_r) \in W$. Since $Y$ is of Fano type over $X$, we can take the relative ample model $\phi \colon Y \dashrightarrow Y'$ for $-\sum_{i=1}^{r} a_iE_i$ after a small $\mathbb{Q}$-factorial modification. Then we need to show $\phi$ satisfies all the desired conditions. For some $a' \in W \cap \mathbb{Q}^r$ near $a$, the divisor $-\sum_{i=1}^{r} a_i'E_i$ yields the same ample model. Hence we may assume that $a \in W \cap \mathbb{Q}^r$. 
    
    Note that 
    \begin{equation*}
        -(K_Y + f^{-1}_{*}\Delta + E) \sim_{X, \mathbb{Q}} -\sum_{i=1}^{r} A_{X,\Delta}(E_i) \cdot E_i
    \end{equation*}
    is ample. Then we can choose a general $\mathbb{Q}$-divisor $F \sim_{\mathbb{Q}} -\sum_{i=1}^{r} A_{X,\Delta}(E_i) \cdot E_i$ that does not contain any stratum of $(Y, E)$. Since $W$ is open, $a' = (a_i - \delta A_{X,\Delta}(E_i))_i \in W$ for some rational $0 < \delta \ll 1$. Then $R_{ka'} \neq 0$ for all sufficiently divisible $k$ since $R$ is an integral domain. Choose $g \in I_{ka'}$ such that its image in $R_{ka'}$ is non-zero. Let $\Gamma = \mathrm{div}_X(g)$, then 
    \begin{equation*}
        0 \sim f^{*}\Gamma = f^{-1}_{*}\Gamma + \sum_{i=1}^{r} ka_i'E_i. 
    \end{equation*}
    For $\xi \in \mathbb{R}_{\geq 0}^r$, we have $v_{\xi}(g) = \langle \xi, ka' \rangle$ since the image of $g$ in $R_{ka'}$ is non-zero; but we also have
    \begin{equation*}
        v_{\xi}(g) = v_{\xi}(f^{-1}_{*}\Gamma) + \sum_{i=1}^{r} ka_i' v_{\xi}(E_i) = v_{\xi}(f^{-1}_{*}\Gamma) + \sum_{i=1}^{r} ka_i'\xi_i. 
    \end{equation*}
    Thus $v_{\xi}(f^{-1}_{*}\Gamma) = 0$. It follows that $f^{-1}_{*}\Gamma$ does not contain any stratum of $(Y, E)$. Let $G = \frac{1}{k}f^{-1}\Gamma$. Then 
    \begin{equation*}
        G + \delta F \sim_{X, \mathbb{Q}} -\sum_{i=1}^{r} a_i E_i, 
    \end{equation*}
    and $\mathrm{Supp}(G + \delta F)$ does not contain any stratum of $(Y, E)$. By \cite[Lem.\ 2.3]{XZ_stable_deg}, for some $0 < \epsilon \ll 1$, there exists an effective $\mathbb{Q}$-divisor $H$ on $Y$ such that $(Y, f^{-1}_{*}\Delta + E + \epsilon(G + \delta F) + H)$ is qdlt, 
    \begin{equation*}
        \lfloor f^{-1}_{*}\Delta + E + \epsilon(G + \delta F) + H \rfloor = E
    \end{equation*}
    and $-(K_Y + f^{-1}_{*}\Delta + E + \epsilon(G + \delta F) + H)$ is ample. Add a general ample effective $\mathbb{Q}$-divisor, we may instead assume that 
    \begin{equation*}
        K_Y + f^{-1}_{*}\Delta + E + \epsilon(G + \delta F) + H \sim_{X, \mathbb{Q}} 0. 
    \end{equation*}
    In other words, let $\Delta' = \Delta + f_{*}H$, then 
    \begin{equation*}
        K_Y + f^{-1}_{*}\Delta' + E = K_Y + f^{-1}_{*}\Delta + E + H \sim_{X, \mathbb{Q}} -\epsilon(G + \delta F) \sim_{X, \mathbb{Q}} \sum_{i=1}^{r} \epsilon a_i E_i. 
    \end{equation*}
    So $(X, \Delta')$ is klt since it is crepant to $(Y, f^{-1}_{*}\Delta' + \sum_{i=1}^{r} (1-\epsilon a_i)E_i)$. Moreover, 
    \begin{equation*}
        -(K_Y + f^{-1}_{*}\Delta' + E + \epsilon G) \sim_{X, \mathbb{Q}} \epsilon \delta F
    \end{equation*}
    is ample, and $(Y, f^{-1}_{*}\Delta' + E + \epsilon G)$ is qdlt. Thus $f \colon (Y, E) \to (X, \Delta')$ is also a model of qdlt Fano type. Then the ample model $(Y', E' = \phi_{*}E)$ for $-\sum_{i=1}^{r} \epsilon a_i E_i$ is the corresponding Koll\'ar model for $(X, \Delta')$ by \cite[Prop.\ 3.9]{XZ_stable_deg}. Hence $(Y', E')$ is also a model of qdlt Fano type for $(X, \Delta)$. 
\end{proof}

\begin{rem}
    If we choose $a \in W$ such that $a_1, \ldots, a_r$ are linearly independent over $\mathbb{Q}$, then we get a model $f' \colon (Y', E') \to (X, \Delta)$ of qdlt Fano type such that each $E_i$ is $\mathbb{Q}$-Cartier, and there is an ample $\mathbb{Q}$-divisor $-A$ on $Y'$ with $A \geq 0$ and $\mathrm{Supp}(A) = \mathrm{Ex}(f')$. 
\end{rem}

\section{Families of Koll\'ar models} \label{sec: family_of_models}

\subsection{Degeneration of valuations via families of models}

In this subsection, let $S$ be a regular local scheme of dimension $d$, with the closed point $s \in S$ and the generic point $\eta \in S$. Let $H = \sum_{j=1}^{d} H_j$ be an snc divisor on $S$ defined by a regular system of parameters. The subscript $(-)_t$ will denote the fiber at a point $t \in S$ for schemes, divisors, and sheaves of modules over $S$. 

\begin{lem} \label{ACC_family}
    Let $\pi \colon (X, \Delta) \to S$ be a locally stable family of pairs. Then there exists $\epsilon > 0$ such that if $f \colon (Y, E) \to (X, \Delta)$ is a family of models where $E$ is $\mathbb{Q}$-Cartier and $(Y, f^{-1}_{*}\Delta + (1-\epsilon)E) \to S$ is locally stable, then $(Y, f^{-1}_{*}\Delta + E) \to S$ is also locally stable. 
\end{lem}

\begin{proof}
    We will apply the ACC of lc thresholds in \cite[Thm.\ 1.1]{HMX_ACC}. Let $\mathrm{Coeff}(\Delta) \subset \mathbb{R}$ denote the set of coefficients of $\Delta$. Let $\mathrm{LCT} \subset \mathbb{R}_{\geq 0}$ be the set of all numbers $\mathrm{lct}(W, D_W; M)$ where $(W, D_W)$ is an slc pair with $\dim(W) = \dim(X)$, with coefficients of $D_W$ in $\mathrm{Coeff}(\Delta) \cup \{1\}$, and $M$ is an effective non-zero $\mathbb{Q}$-Cartier $\mathbb{Z}$-divisor on $W$. Then $\mathrm{LCT}$ satisfies the ACC by \cite[Thm.\ 1.1]{HMX_ACC}; note that we may allow slc pairs $(W, D_W)$ by \cite[Thm.\ 5.38]{Kollar_singularity}. Since $\mathrm{LCT} \subset [0,\, 1]$ and $1 \in \mathrm{LCT}$, there exists $\epsilon > 0$ such that 
    \begin{equation*}
        \mathrm{LCT} \cap [1 - \epsilon,\, 1] = \{1\}
    \end{equation*}
    by the ACC. In our case, $(Y, f^{-1}_{*}\Delta + (1-\epsilon)E + f^{*}\pi^{*}H)$ is slc by \cite[Thm.\ 4.54]{Kol_fam}, that is, 
    \begin{equation*}
        1 - \epsilon \leq \mathrm{lct}(Y, f^{-1}_{*}\Delta + f^{*}\pi^{*}H; E) \in \mathrm{LCT}. 
    \end{equation*}
    Hence $\mathrm{lct}(Y, f^{-1}_{*}\Delta + f^{*}\pi^{*}H; E) = 1$, that is, $(Y, f^{-1}_{*}\Delta + E + f^{*}\pi^{*}H)$ is slc. So $(Y, f^{-1}_{*}\Delta + E) \to S$ is a locally stable family by \cite[Thm.\ 4.54]{Kol_fam}. 
\end{proof}

\begin{lem} \label{specialize_val_computes_constant_lct}
    Let $\pi \colon (X, \Delta) \to S$ with $x \in X(S)$ be a locally stable family of klt singularities, and $I_{\bullet}$ be an ideal sequence on $X$ cosupported on $x(S)$ such that
    \begin{equation*}
         \mathrm{lct}(X_{s}, \Delta_{s}; I_{s, \bullet}) = \mathrm{lct}(X_{\eta}, \Delta_{\eta}; I_{\eta, \bullet}) < \infty
    \end{equation*}
    Suppose $v_0 \in \mathrm{Val}_{X_{\eta}, x_{\eta}}^{\mathrm{qm}}$ computes the lc threshold of $I_{\eta, \bullet}$, and $\sigma \subset \mathrm{Val}_{X_{\eta}, x_{\eta}}$ is a quasi-monomial simplicial cone for $(X_{\eta}, \Delta_{\eta})$ such that $v_0 \in \sigma$. Then there exists a family of models at $x$ over $S$
    \begin{equation*}
        f \colon (Y, E = E_1 + \cdots + E_r) \to (X, \Delta)
    \end{equation*}
    such that the following hold: 
    \begin{enumerate}[label=\emph{(\arabic*)}, nosep]
        \item $E_1, \ldots, E_r$ are all the exceptional divisors of $f$, and the valuations $\mathrm{ord}_{E_{1}}, \ldots, \mathrm{ord}_{E_{r}}$ are centered at $x_{\eta}$ and span a simplicial cone $\tau \subset \sigma$ such that $v_0 \in \tau$; 
        \item $\pi \circ f \colon (Y, f^{-1}_{*}\Delta + E) \to S$ is a locally stable family; 
        \item there exists a relative $\mathbb{Q}$-Cartier $\mathbb{Q}$-divisor $\Gamma$ on $X/S$ such that $\pi \colon (X, \Delta + \Gamma) \to S$ is a locally stable family, $K_X + \Delta + \Gamma \sim_{S, \mathbb{Q}} 0$, and $\tau \subset \mathrm{LCP}(X_{\eta}, \Delta_{\eta} + \Gamma_{\eta})$. 
        \item $Y$ is $\mathbb{Q}$-factorial, and $-(K_Y + f^{-1}_{*}\Delta + E)$ is $f$-semi-ample.  
    \end{enumerate}
\end{lem}

\begin{proof}
    By rescaling the index of $I_{\bullet}$, we may assume that $\mathrm{lct}(X_{s}, \Delta_{s}; I_{s, \bullet}) = \mathrm{lct}(X_{\eta}, \Delta_{\eta}; I_{\eta, \bullet}) = 1$. The assumption that $v_0$ computes the lc threshold means 
    \begin{equation*}
        \mathrm{lct}(X_{\eta}, \Delta_{\eta}; I_{\eta, \bullet}) = \frac{A_{X_{\eta}, \Delta_{\eta}}(v_0)}{v_0(I_{\eta, \bullet})}. 
    \end{equation*}
    By rescaling $v_0$, we may assume $A_{X_{\eta}, \Delta_{\eta}}(v_0) = 1$. Hence $v_0(I_{\eta, \bullet}) = 1$. Since $X_{\eta} \subset X$ is the generic fiber, we identify $\mathrm{Val}_{X_{\eta}, x_{\eta}} = \mathrm{Val}_{X,x}$, so that $A_{X, \Delta}(v_0) = v_0(I_{\bullet}) = 1$. By inversion of adjunction, $(X, \Delta + \pi^{*}H)$ is dlt, and 
    \begin{equation*}
        \mathrm{lct}(X, \Delta + \pi^{*}H; I_{\bullet}) = \mathrm{lct}(X_s, \Delta_s; I_{s, \bullet}) = 1. 
    \end{equation*}
    It follows that $\mathrm{lct}(X, \Delta; I_{\bullet}) = 1$, and $v_0$ computes the lc threshold of $I_{\bullet}$ on $(X, \Delta)$. 

    Fix an isomorphism $\sigma \simeq \mathbb{R}_{\geq 0}^r$. Note that the function 
    \begin{equation*}
        A_{X, \Delta + I_{\bullet}}(v) \coloneqq A_{X, \Delta}(v) - v(I_{\bullet}) = A_{X, \Delta}(v) - \inf_{\lambda > 0} \frac{v(I_{\lambda})}{\lambda}
    \end{equation*}
    is non-negative, homogeneous, and convex on $\sigma \simeq (\mathbb{R}_{\geq 0})^r$. Choose a norm $\| \cdot \|$ on $\mathbb{R}^r$, and write $\mathrm{d}(\cdot, \cdot)$ for the induced metric on $\sigma$. Then $A_{X, \Delta+I_{\bullet}}$ is a Lipschitz function on $\sigma$.\footnote{Any finite convex function on a convex subset of $\mathbb{R}^r$ is locally Lipschitz. Here $A_{X, \Delta+I_{\bullet}}$ is homogeneous (of degree $1$) and $\sigma$ is a cone over a compact set, hence $A_{X, \Delta+I_{\bullet}}$ is globally Lipschitz on $\sigma$.} Hence there exists $C > 0$ such that 
    \begin{equation*}
        A_{X, \Delta + I_{\bullet}}(v) \leq A_{X, \Delta + I_{\bullet}}(v_0) + C \cdot \mathrm{d}(v_0, v) = C \cdot \mathrm{d}(v_0, v). 
    \end{equation*}
    Fix $0 < \epsilon < 1$ as in Lemma \ref{ACC_family} for $\pi \colon (X, \Delta) \to S$. By Diophantine approximation (see \cite[Lem.\ 2.7]{LX_stability} or \cite[Lem.\ 4.47]{Xu_K-Stability_Book}), there exist $v_1, \ldots, v_r \in \sigma$ and $q_1, \ldots, q_r \in \mathbb{Z}_{>0}$ such that: 
    \begin{enumerate}[label=(\arabic*), nosep]
        \item $v_0$ is in the convex cone $\tau$ spanned by $v_1, \ldots, v_r$; 
        \item $q_iv_i = c_i \mathop{\mathrm{ord}_{E_i}}$, where $E_i$ is a prime divisor over $X$ and $c_i \in \mathbb{Z}_{>0}$, for all $i = 1, \ldots, r$; 
        \item $\mathrm{d}(v_0, v_i) < \frac{\epsilon}{2Cq_i}$ for all $i = 1, \ldots, r$. 
    \end{enumerate}
    Thus 
    \begin{equation*}
        0 \leq A_{X, \Delta + I_{\bullet}}(E_i) = \frac{q_i}{c_i} A_{X, \Delta + I_{\bullet}}(v_i) < \frac{\epsilon}{2} 
    \end{equation*}
    for all $i = 1, \ldots, r$. 

    Choose $\epsilon' > 0$ such that $2\epsilon' \mathop{\mathrm{ord}_{E_i}}(I_{\bullet}) < \epsilon$ for all $i = 1, \ldots, r$. Since 
    \begin{equation*}
        \sup_{\lambda > 0} \lambda \cdot \mathrm{lct}(X_s, \Delta_s; I_{s,\lambda}) = \lim_{\lambda \to \infty} \lambda \cdot \mathrm{lct}(X_s, \Delta_s; I_{s,\lambda}) = \mathrm{lct}(X_s, \Delta_s; I_{s,\bullet}) = 1, 
    \end{equation*}
    we have 
    \begin{equation*}
        (1 - \epsilon')/\lambda < \mathrm{lct}(X_s, \Delta_s; I_{s, \lambda}) \leq 1/\lambda
    \end{equation*}
    for $\lambda \gg 0$. Write $c = (1 - \epsilon')/\lambda$, then $\mathrm{lct}(X, \Delta; I_{\lambda}^c) \geq \mathrm{lct}(X_s, \Delta_s; I_{s, \lambda}^c) > 1$, and
    \begin{equation*}
        \begin{aligned}
            A_{X, \Delta + I_{\lambda}^{c}}(E_i) &\leq A_{X, \Delta + I_{\bullet}}(E_i) + (\mathop{\mathrm{ord}_{E_i}}(I_{\bullet}) - c\mathop{\mathrm{ord}_{E_i}}(I_{\lambda})) \\
            &< \frac{\epsilon}{2} + (1 - \epsilon')\left( \mathop{\mathrm{ord}_{E_i}}(I_{\bullet}) - \frac{1}{\lambda}\mathop{\mathrm{ord}_{E_i}}(I_{\lambda}) \right) + \epsilon' \mathop{\mathrm{ord}_{E_i}}(I_{\bullet}) \\
            &< \frac{\epsilon}{2} + 0 + \frac{\epsilon}{2} = \epsilon 
        \end{aligned}
    \end{equation*}
    where we used the general fact that $v(I_{\bullet}) \leq v(I_{\lambda})/\lambda$ for every valuation $v$. 

    Write $a_i \coloneqq A_{X, \Delta + I_{\lambda}^{c}}(E_i) < \epsilon < 1$. By Lemma \ref{extract_div}, there exists a model 
    \begin{equation*}
        f \colon (Y, E) \to (X, \Delta)
    \end{equation*}
    such that $E = E_1 + \cdots + E_r$ is the sum of all exceptional divisors of $f$, and $Y$ is $\mathbb{Q}$-factorial. Since 
    \begin{equation*}
        \mathrm{lct}(X, \Delta + \pi^{*}H; I_{\lambda}^c) = \mathrm{lct}(X_s, \Delta_s; I_{s, \lambda}^c) > 1, 
    \end{equation*}
    $(Y, f^{-1}_{*}\Delta + \sum_{i=1}^{r} (1-a_i)E_{i} + f^{*}\pi^{*}H)$ is lc, that is, $\pi \circ f \colon (Y, f^{-1}_{*}\Delta + \sum_{i=1}^{r} (1-a_i)E_{i}) \to S$ is locally stable. Thus $\pi \circ f \colon (Y, f^{-1}_{*}\Delta + E) \to S$ is locally stable since $a_i < \epsilon$ for all $i$, by Lemma \ref{ACC_family}. 

    Since $Y$ is of Fano type over $X$, we can run a $-(K_Y + f^{-1}_{*}\Delta + E)$-MMP over $X$ and get a $\mathbb{Q}$-factorial good minimal model 
    \begin{equation*}
        \begin{tikzcd}
            Y \ar[rr, "\phi", dashed] \ar[rd, "f"'] & & Y' \ar[ld, "f'"] \\
            & X
        \end{tikzcd}
    \end{equation*}
    where $-(K_{Y'} + f'^{-1}_{*}\Delta + \phi_{*}E)$ is $f'$-semi-ample. As above, $(Y', f'^{-1}_{*}\Delta + \sum_{i=1}^{r} (1-a_i)\phi_{*}E_{i} + f'^{*}\pi^{*}H)$ is also lc, hence $\pi \circ f' \colon (Y', f'^{-1}_{*}\Delta + \phi_{*}E) \to S$ is a locally stable family by Lemma \ref{ACC_family}. Applying \cite[Lem.\ 3.38]{Kollar--Mori} in each step of the $-(K_Y + f^{-1}_{*}\Delta + E)$-MMP, we have 
    \begin{equation*}
        A_{Y, f^{-1}_{*}\Delta + E + f^{*}\pi^{*}H}(F) \geq A_{Y', f'^{-1}_{*}\Delta + \phi_{*}E + f'^{*}\pi^{*}H}(F)
    \end{equation*}
    for every prime divisor $F$, and the equality holds if and only if $\phi$ is a local isomorphism at $\mathrm{center}_Y(F)$. Thus $\phi$ is a local isomorphism at every lc center of $(Y, f^{-1}_{*}\Delta + E + f^{*}\pi^{*}H)$. In particular, $\phi$ does not contract any component $E_i$ of $E$. Thus $f' \colon (Y', E' = \phi_{*}E) \to (X, \Delta)$ satisfies (1), (2), and (4). 

    It remains to show (3). Since $-(K_{Y} + f^{-1}_{*}\Delta + E)$ is $f$-semi-ample and $X$ is affine, we can choose a general $\mathbb{Q}$-divisor
    \begin{equation*}
        G \in |{-(K_{Y} + f^{-1}_{*}\Delta + E)}|_{\mathbb{Q}}
    \end{equation*}
    such that $(Y, f^{-1}_{*}\Delta + G + E + f^{*}\pi^{*}H)$ is lc. Let $\Gamma = f_{*}G \sim_{\mathbb{Q}} -(K_X + \Delta)$, then 
    \begin{equation*}
        f^{*}(K_X + \Delta + \Gamma + \pi^{*}H) = K_Y + f^{-1}_{*}\Delta + G + E + f^{*}\pi^{*}H, 
    \end{equation*}
    so $\pi \colon (X, \Delta + \Gamma) \to S$ is locally stable, and each $E_i$ is an lc place of $(X, \Delta + \Gamma)$. Since the function 
    \begin{equation*}
        v \mapsto A_{X,\Delta + \Gamma}(v) = A_{X, \Delta}(v) - v(\Gamma)
    \end{equation*}
    is non-negative, homogeneous, and convex on $\tau$, while vanishing at $\mathrm{ord}_{E_i}$, we get $A_{X,\Delta + \Gamma}(v) = 0$ for all $v \in \tau$, that is, $\tau \subset \mathrm{LCP}(X, \Delta + \Gamma)$. 
\end{proof}

\begin{rem} \label{get_qdlt_model_over_generic_point}
    Assume $v_0$ is a special valuation, and $\sigma$ is a special quasi-monomial cone. Then $\tau$ is also special, so $f_{\eta} \colon (Y_{\eta}, E_{\eta}) \to (X_{\eta}, \Delta_{\eta})$ is a model of qdlt Fano type by Lemma \ref{minimal_model_with_given_exceptional_is_qdlt_Fano}, with $\mathrm{QM}(Y_{\eta}, E_{\eta}) = \tau$. In particular, $v_0 \in \mathrm{QM}(Y_{\eta}, E_{\eta})$. Since every lc center of $(Y, f^{-1}_{*}\Delta + E)$ lies over $\eta$ by \cite[Cor.\ 4.56]{Kol_fam}, $(Y, f^{-1}_{*}\Delta + E)$ is qdlt. 

    For $a = (a_1, \ldots, a_r) \in (\mathbb{R}_{\geq 0})^r$, let $f' \colon (Y', E') \to (X, \Delta)$ be the relative ample model for $-\sum_{i=1}^{r} a_i E_i$, where $E'$ is the strict transform of $E$. By Lemma \ref{qdlt_Fano_type_with_given_ample_exc}, for suitable choices of $a$, the generic fiber $(Y'_{\eta}, E'_{\eta})$ is also a model of qdlt Fano type for $(X_{\eta}, \Delta_{\eta})$. In particular, we may assume that $a_1, \ldots, a_r$ are linearly independent over $\mathbb{Q}$, then each $E_i'$ is $\mathbb{Q}$-Cartier, and there exists an $f'$-ample $\mathbb{Q}$-divisor $-A'$ on $Y'$ such that $A' \geq 0$ and $\mathrm{Supp}(A') = \mathrm{Ex}(f') = E'$. Note that $f' \colon (Y', E') \to (X, \Delta)$ satisfies (1), (2), and (3). 
\end{rem}

\begin{prg} \label{quasi-monomial_deg_val}
    Let $\pi \colon (X, \Delta) \to S$ be a family of pairs, and $f \colon (Y, E) \to (X, \Delta)$ be a family of models such that $E = \sum_{i=1}^{k} E_i$ is the sum of all $f$-exceptional divisors. Let $Z$ be an irreducible component of $\bigcap_{i\in I} E_i$ for some $I \subset \{1, \ldots, k\}$. Let $\zeta$ be the generic point of $Z$, and $z$ be a generic point of $Z_s = Z \cap Y_s$. Assume that $\pi \circ f \colon (Y, f^{-1}_{*}\Delta + E) \to S$ is locally stable in an open neighborhood of $z$, and $E_i$ is $\mathbb{Q}$-Cartier at $z$ for all $i \in I$. 
    
    Then $(Y, f^{-1}_{*}\Delta + E + f^{*}\pi^{*}H)$ is lc, and $Z$ is an lc center of $(Y, f^{-1}_{*}\Delta + E)$, in an open neighborhood of $z$. Thus $Z$ dominates $S$ (see \cite[Cor.\ 4.56]{Kol_fam}), that is, the generic point $\zeta$ of $Z$ lies over the generic point $\eta$ of $S$. Since $E_{i,\eta}$ is $\mathbb{Q}$-Cartier at $\zeta$ for all $i \in I$, the pair $(Y_{\eta}, E_{\eta})$ is simple-toroidal at $\zeta$, and the other components $E_{i',\eta}$ does not pass through $\eta$ for $i' \notin I$, by Lemma \ref{simple_toroidal}. Now
    \begin{equation*}
        z \in \bigcap_{i \in I} E_i \cap \bigcap_{j=1}^{d} f^{*}\pi^{*}H_j
    \end{equation*}
    where each $E_i$ is $\mathbb{Q}$-Cartier, and each $f^{*}\pi^{*}H_j$ is Cartier. Since $z$ has codimension at most $|I| + d$ in $Y$, $(Y, E + f^{*}\pi^{*}H)$ is simple-toroidal at $z$ by Lemma \ref{simple_toroidal}, hence $(Y_s, E_s)$ is simple-toroidal at $z$ as well. By Lemma \ref{degenerate_monomial_val}, we have a commutative diagram  
    \begin{equation*}
        \begin{tikzcd}
            \mathrm{QM}_{\zeta}(Y_{\eta}, E_{\eta}) \ar[r, "i", hook] \ar[d, "\simeq"] & \mathrm{QM}_{z}(Y, E + f^{*}\pi^{*}H) \ar[r, "p"] \ar[d, "\simeq"] & \mathrm{QM}_{z}(Y_s, E_s) \ar[d, "\simeq"] \\ 
            \mathbb{R}_{\geq 0}^{I} \ar[r, "{\alpha \mapsto (\alpha, 0)}"] & \mathbb{R}_{\geq 0}^{I} \times \mathbb{R}_{\geq 0}^{d} \ar[r, "{(\alpha, \beta) \mapsto \alpha}"] & \mathbb{R}_{\geq 0}^{I}
        \end{tikzcd}
    \end{equation*}
    giving a canonical isomorphism 
    \begin{equation*}
        \mathrm{QM}_{\zeta}(Y_{\eta}, E_{\eta}) \simeq \mathrm{QM}_{z}(Y_s, E_s)
    \end{equation*}
    which maps $E_{i, \eta}$ to the unique irreducible component $E_{i,s}'$ of $E_{i,s}$ containing $z$. 
\end{prg}

\begin{defn}
    Keep the notations in \ref{quasi-monomial_deg_val}. For a quasi-monomial valuation $v^{\alpha}_{\eta} \in \mathrm{QM}_{\zeta}(Y_{\eta}, E_{\eta}) \subset \mathrm{Val}_{X_{\eta}}^{\mathrm{qm}}$, we say its image $v^{\alpha}_s \in \mathrm{QM}_{z}(Y_s, E_s) \subset \mathrm{Val}_{X_s}^{\mathrm{qm}}$ is a \emph{quasi-monomial degeneration} of $v^{\alpha}_{\eta}$, or $v^{\alpha}_{\eta}$ \emph{degenerates} to $v^{\alpha}_s$, via the family of models $f \colon (Y, E) \to (X, \Delta)$. 
\end{defn}

\begin{lem} \label{compare_qm_deg_val}
    Let $\pi \colon (X, \Delta) \to S$ be a family of pairs. Suppose $v_{\eta} \in \mathrm{Val}_{X_{\eta}}^{\mathrm{qm}}$ degenerates to $v_s \in \mathrm{Val}_{X_{s}}^{\mathrm{qm}}$ via a family of models $f \colon (Y, E = \sum_{i=1}^{k} E_i) \to (X, \Delta)$ over $S$. Then the following hold: 
    \begin{enumerate}[label=\emph{(\arabic*)}, nosep]
        \item $v_{\eta}(\mathfrak{a}_{\eta}) \leq v_{s}(\mathfrak{a}_s)$ for every coherent ideal $\mathfrak{a} \subset \mathscr{O}_X$; 
        \item if $K_X + \Delta$ is $\mathbb{Q}$-Cartier, then $A_{X_{\eta}, \Delta_{\eta}}(v_{\eta}) = A_{X_s, \Delta_s}(v_s)$; 
        \item if $\Gamma$ is a relative $\mathbb{Q}$-Cartier $\mathbb{Q}$-divisor on $X/S$ such that $\pi \colon (X, \Delta + \Gamma) \to S$ is locally stable and $v_{\eta} \in \mathrm{LCP}(X_{\eta}, \Delta_{\eta} + \Gamma_{\eta})$, then $v_{s} \in \mathrm{LCP}(X_{s}, \Delta_{s} + \Gamma_{s})$ and $v_{\eta}(\Gamma_{\eta}) = v_s(\Gamma_s)$. 
    \end{enumerate}
\end{lem}

\begin{proof}
     Suppose $Z$ is an irreducible component $\bigcap_{i\in I} E_i$ as in Definition \ref{quasi-monomial_deg_val}, such that $v_{\eta}$ maps to $v_s$ in 
     \begin{equation*}
         \mathrm{QM}_{\zeta}(Y_{\eta}, E_{\eta}) \simeq \mathbb{R}_{\geq 0}^r \simeq \mathrm{QM}_z(Y_s, E_s), 
     \end{equation*}
     where $\zeta$ is the generic point of $Z_{\eta}$, and $z$ is a generic point of $Z_s$, and they are given by $\alpha = (\alpha_i) \in \mathbb{R}_{\geq 0}^{I}$ in coordinates. Then (1) follows from Lemma \ref{degenerate_monomial_val}. 

     (2). Write 
     \begin{equation*}
         f^{*}(K_X + \Delta + \pi^{*}H) = K_Y + f^{-1}_{*}\Delta + \sum_{i=1}^{k} (1 - A_{X,\Delta}(E_i)) \cdot E_i + f^{*}\pi^{*}H. 
     \end{equation*}
     Restricting to $f_{\eta} \colon (Y_{\eta}, \Delta_{\eta}) \to (X_{\eta}, \Delta_{\eta})$ and $f_{s} \colon (Y_{s}, \Delta_{s}) \to (X_{s}, \Delta_{s})$, respectively, by adjunction we get 
     \begin{equation*}
         A_{X_{\eta},\Delta_{\eta}}(E_{i,\eta}) = A_{X,\Delta}(E_i) = A_{X_s, \Delta_s}(E_{i,s}')
     \end{equation*}
     for every component $E_{i,s}'$ of $E_{i,s}$. Hence $A_{X_{\eta}, \Delta_{\eta}}(v_{\eta}) = \sum_{i \in I} \alpha_i A_{X, \Delta}(E_i) = A_{X_{s}, \Delta_{s}}(v_{s})$. 

     (3). By (1) and (2) we have 
     \begin{equation*}
         A_{X_s, \Delta_s + \Gamma_s}(v_s) = A_{X_s, \Delta_s}(v_s) - v_s(\Gamma_s) \leq A_{X_{\eta}, \Delta_{\eta}}(v_{\eta}) - v_{\eta}(\Gamma_{\eta}) = A_{X_{\eta}, \Delta_{\eta} + \Gamma_{\eta}}(v_{\eta}) = 0,  
     \end{equation*}
     while $A_{X_s, \Delta_s + \Gamma_s}(v_s) \geq 0$ since $(X_s, \Delta_s + \Gamma_s)$ is slc. Thus all the inequalities above must be equalities, that is, $A_{X_s, \Delta_s + \Gamma_s}(v_s) = 0$ and $v_s(\Gamma_s) = v_{\eta}(\Gamma_{\eta})$. 
\end{proof}

\begin{lem} \label{construct_family_of_Kol_model}
    Let $\pi (X, \Delta) \to S$ with $x \in X(S)$ be a locally stable family of klt singularities, and $I_{\bullet}$ be an ideal sequence on $X$ cosupported on $x(S)$ such that
    \begin{equation*}
         \mathrm{lct}(X_{s}, \Delta_{s}; I_{s, \bullet}) = \mathrm{lct}(X_{\eta}, \Delta_{\eta}; I_{\eta, \bullet}) < \infty. 
    \end{equation*}
    Suppose $v_s \in \mathrm{Val}_{X_s, x_s}^{\mathrm{qm}}$ and $v_{\eta} \in \mathrm{Val}_{X_{\eta}, x_{\eta}}^{\mathrm{qm}}$ are special valuations with 
    \begin{equation*}
        A_{X_s, \Delta_s}(v_s) = A_{X_{\eta}, \Delta_{\eta}}(v_{\eta}) = 1, 
    \end{equation*}
    satisfying the following: for every relative effective $\mathbb{Q}$-Cartier divisor $D$ on $X/S$ with $v_{s}(D_s) = v_{\eta}(D_{\eta})$, there exists $\epsilon > 0$ such that $\pi \colon (X, \Delta + \epsilon D) \to S$ has klt fibers, and $v_s$ and $v_{\eta}$ are, up to scaling, unique quasi-monomial valuations computing lc thresholds $\mathrm{lct}(X_s, \Delta_s + \epsilon D_s; I_{s, \bullet})$ and $\mathrm{lct}(X_{\eta}, \Delta_{\eta} + \epsilon D_{\eta}; I_{\eta, \bullet})$, respectively. Then there exists a locally stable family of Koll\'ar models at $x$
    \begin{equation*}
        \bar{f} \colon (\overline{Y}, \overline{E}) \to (X, \Delta)
    \end{equation*}
    such that $v_s \in \mathrm{QM}(\overline{Y}_s, \overline{E}_s)$, $v_{\eta} \in \mathrm{QM}(\overline{Y}_{\eta}, \overline{E}_{\eta})$, and $v_{\eta}$ degenerates to $v_s$. 
\end{lem}

\begin{proof}
    First note that for $D = 0$ we have 
    \begin{equation*}
        v_s(I_{s, \bullet}) = A_{X_s, \Delta_s}(v_s) \cdot \mathrm{lct}(X_{s}, \Delta_{s}; I_{s, \bullet}) = A_{X_{\eta}, \Delta_{\eta}}(v_{\eta}) \cdot \mathrm{lct}(X_{\eta}, \Delta_{\eta}; I_{\eta, \bullet}) = v_{\eta}(I_{\eta, \bullet}). 
    \end{equation*}
    Assume $D$ is a relative effective $\mathbb{Q}$-Cartier divisor on $X/S$ with $v_{s}(D_s) = v_{\eta}(D_{\eta})$, then 
    \begin{equation*}
        \begin{aligned}
            \mathrm{lct}(X_s, \Delta_s + \epsilon D_s; I_{s, \bullet}) &= \frac{A_{X_s, \Delta_s}(v_s) - \epsilon v_s(D_s)}{v_s(I_{s, \bullet})} \\ 
            &= \frac{A_{X_{\eta}, \Delta_{\eta}}(v_{\eta}) - \epsilon v_{\eta}(D_{\eta})}{v_{\eta}(I_{\eta, \bullet})} = \mathrm{lct}(X_{\eta}, \Delta_{\eta} + \epsilon D_{\eta}; I_{\eta, \bullet}). 
        \end{aligned}
    \end{equation*}
    Thus we can apply Lemma \ref{specialize_val_computes_constant_lct} to the family $\pi \colon (X, \Delta + \epsilon D) \to S$, the ideal sequence $I_{\bullet}$, and $v_{\eta}$, so we get a family of models $f^D \colon (Y^D, E^D) \to (X, \Delta + \epsilon D)$, via which $v_{\eta}$ degenerates to some $v^D_s \in \mathrm{Val}_{X_s,x_s}^{\mathrm{qm}}$ as in \ref{quasi-monomial_deg_val} (note that $v_{\eta} \in \mathrm{QM}(Y^D_{\eta}, E^D_{\eta})$ by Remark \ref{get_qdlt_model_over_generic_point}). Then 
    \begin{equation*}
        \mathrm{lct}(X_{\eta}, \Delta_{\eta} + \epsilon D_{\eta}; I_{\eta, \bullet}) = \frac{A_{X_{\eta}, \Delta_{\eta} + \epsilon D_{\eta}}(v_{\eta})}{v_{\eta}(I_{\eta, \bullet})} \geq \frac{A_{X_{s}, \Delta_{s} + \epsilon D_{s}}(v_{s}^D)}{v_{s}^D(I_{s, \bullet})}
    \end{equation*}
    by Lemma \ref{compare_qm_deg_val}. This implies that $v_s^D$ computes $\mathrm{lct}(X_s, \Delta_s + \epsilon D_s; I_{s, \bullet})$. By the uniqueness assumption, we conclude that $v_s^D = v_s$ since $A_{X_{s}, \Delta_{s} + \epsilon D_{s}}(v_{s}^D) = A_{X_{\eta}, \Delta_{\eta} + \epsilon D_{\eta}}(v_{\eta}) = A_{X_{s}, \Delta_{s} + \epsilon D_{s}}(v_{s})$. Moreover, Lemma \ref{specialize_val_computes_constant_lct} gives a relative effective $\mathbb{Q}$-Cartier $\mathbb{Q}$-divisor $\Gamma \geq \epsilon D$ on $X/S$ such that $\pi \colon (X, \Delta + \Gamma) \to S$ is locally stable, $K_X + \Delta + \Gamma \sim_{S, \mathbb{Q}} 0$, and $v_{\eta} \in \mathrm{LCP}(X_{\eta}, \Delta_{\eta} + \Gamma_{\eta})$. By Lemma \ref{compare_qm_deg_val}, $v_s \in \mathrm{LCP}(X_s, \Delta_s + \Gamma_s)$. 

    By Remark \ref{get_qdlt_model_over_generic_point}, since $v_{\eta}$ is a special, we can choose a family of models $f \colon (Y, E = \sum_{i=1}^{r} E_i) \to (X, \Delta)$ such that generic fiber is of qdlt Fano type, each $E_i$ is $\mathbb{Q}$-Cartier, and there exists an $f$-ample divisor supported on $E$. Thus $Z = \bigcap_{i=1}^{r} E_i$ is irreducible, with the generic point $\zeta$ such that $v_{\eta} \in \mathrm{QM}_{\zeta}(Y_{\eta}, E_{\eta})$. Let $z$ be a generic point of $Z_s$, then $(Y, E + f^{*}\pi^{*}H)$ is simple-toroidal at $z$, and as in the last paragraph, we have $v_s \in \mathrm{QM}_z(Y_s, E_s)$ corresponding to $v_{\eta}$ and $\alpha = (\alpha_i) \in \mathbb{R}_{\geq 0}^r$ under 
    \begin{equation*}
        \mathrm{QM}_{\zeta}(Y_{\eta}, E_{\eta}) \simeq \mathbb{R}_{\geq 0}^r \simeq \mathrm{QM}_z(Y_s, E_s). 
    \end{equation*}
    In fact, by the uniqueness assumption, $z$ is the unique generic point of $Z_s$. Let 
    \begin{equation*}
        \mu \colon (W, F) \to (Y, f^{-1}_{*}\Delta + E + f^{*}\pi^{*}H)
    \end{equation*}
    be a log resolution that is a local isomorphism over $z$ and the generic point of $Y_s$, such that there is a $\mu$-ample divisor supported on $\mathrm{Ex}(\mu)$. Hence there is a $g$-ample divisor $-A$ with $\mathrm{Supp}(A) = \mathrm{Ex}(g)$ and $A \geq 0$, where $g = f \circ \mu$. Since $X$ is affine, $-A$ is ample. Let $G \in |{-A}|_{\mathbb{Q}}$ be a general $\mathbb{Q}$-divisor whose support does not contain any stratum of $(W, F)$, and let $D = g_{*}G = g_{*}(G + A)$. Then $D$ is $\mathbb{Q}$-Cartier, and $\mathrm{Supp}(D)$ does not contain $X_s$. Thus $D$ is flat over $S$. Write
    \begin{equation*}
        f^{*}D = f^{-1}_{*}D + \sum_{i=1}^{r} \mathrm{ord}_{E_i}(D) E_i
    \end{equation*}
    where the support of $f^{-1}_{*}D = \mu_{*}G$ does not contain $z$. Hence $v_s(D_s) = \sum_{i=1}^{r} \alpha_i \mathop{\mathrm{ord}_{E_i}}(D) = v_{\eta}(D_{\eta})$. By last paragraph, there exists a relative effective $\mathbb{Q}$-Cartier $\mathbb{Q}$-divisor $\Gamma \geq \epsilon D$ on $X/S$, for some $\epsilon > 0$, such that $\pi \colon (X, \Delta + \Gamma) \to S$ is locally stable, $K_X + \Delta + \Gamma \sim_{S, \mathbb{Q}} 0$, and $v_{\eta} \in \mathrm{LCP}(X_{\eta}, \Delta_{\eta} + \Gamma_{\eta})$. 

    Consider the diagram 
    \begin{equation*}
        \begin{tikzcd}
            \mathrm{QM}_{\zeta}(Y_{\eta}, E_{\eta}) \ar[r, "i", hook] \ar[d, "\simeq"] & \mathrm{QM}_{z}(Y, E + f^{*}\pi^{*}H) \ar[r, "p"] \ar[d, "\simeq"] & \mathrm{QM}_{z}(Y_s, E_s) \ar[d, "\simeq"] \\ 
            \mathbb{R}_{\geq 0}^{r} \ar[r, "{\alpha \mapsto (\alpha, 0)}"] & \mathbb{R}_{\geq 0}^{r} \times \mathbb{R}_{\geq 0}^{d} \ar[r, "{(\alpha, \beta) \mapsto \alpha}"] & \mathbb{R}_{\geq 0}^{r}
        \end{tikzcd}
    \end{equation*}
    Suppose $\alpha = (\alpha_i) \in \mathbb{R}_{\geq 0}^r$ corresponds to $v_{\eta}^{\alpha} \in \mathrm{QM}_{\zeta}(Y_{\eta}, E_{\eta}) \cap \mathrm{LCP}(X_{\eta}, \Delta_{\eta} + \Gamma_{\eta})$, then 
    \begin{equation*}
        v_s^{\alpha} \in \mathrm{QM}_{z}(Y_s, E_s) \cap \mathrm{LCP}(X_s, \Delta_s + \Gamma_s)
    \end{equation*}
    and $v_s^{\alpha}(\Gamma_s) = v_{\eta}^{\alpha}(\Gamma_{\eta})$ by Lemma \ref{compare_qm_deg_val}. For every $\beta = (\beta_j) \in \mathbb{R}_{\geq 0}^d$ and $v^{\alpha,\beta} \in \mathrm{QM}_{z}(Y, E + f^{*}\pi^{*}H)$, 
    \begin{equation*}
        A_{X, \Delta + \pi^{*}H}(v^{\alpha, \beta}) = \sum_{i=1}^{r} \alpha_i A_{X,\Delta + \pi^{*}H}(E_i) + \sum_{j=1}^{d} \beta_j A_{X, \Delta + \pi^{*}H}(f^{*}\pi^{*}H_j) = \sum_{i=1}^{r} \alpha_i A_{X,\Delta + \pi^{*}H}(E_i)
    \end{equation*}
    since $A_{X, \Delta + \pi^{*}H}(f^{*}\pi^{*}H_j) = 0$. Thus $A_{X_{\eta}, \Delta_{\eta}}(v^{\alpha}_{\eta}) = A_{X, \Delta + \pi^{*}H}(v^{\alpha,\beta}) = A_{X_s, \Delta_s}(v_s^{\alpha})$. Also, 
    \begin{equation*}
        v_{\eta}^{\alpha}(\Gamma_{\eta}) \leq v^{\alpha,\beta}(\Gamma) \leq v_{s}^{\alpha}(\Gamma_s)
    \end{equation*}
    by Lemma \ref{degenerate_monomial_val}, so all the equality holds. Then we get 
    \begin{equation*}
        A_{X_{\eta}, \Delta_{\eta} + \Gamma_{\eta}}(v^{\alpha}_{\eta}) = A_{X, \Delta + \Gamma + \pi^{*}H}(v^{\alpha,\beta}) = A_{X_s, \Delta_s + \Gamma_s}(v_s^{\alpha}) = 0. 
    \end{equation*}
    Let $\sigma_{\eta} \subset \mathrm{QM}_{\zeta}(Y_{\eta}, E_{\eta}) \cap \mathrm{LCP}(X_{\eta}, \Delta_{\eta} + \Gamma_{\eta})$ be a simplicial cone containing $v_{\eta}$, and let $\sigma \subset \mathbb{R}_{\geq 0}^r$ be the corresponding cone under the canonical isomorphism $\mathrm{QM}_{\zeta}(Y_{\eta}, E_{\eta}) \simeq \mathbb{R}_{\geq 0}^r$. Then $\sigma \times \mathbb{R}_{\geq 0}^{d}$ corresponds to a cone $\Sigma \subset \mathrm{QM}_{z}(Y, E + f^{*}\pi^{*}H)$ under $\mathrm{QM}_{z}(Y, E + f^{*}\pi^{*}H) \simeq \mathbb{R}_{\geq 0}^r \times \mathbb{R}_{\geq 0}^d$ such that  
    \begin{equation*}
        \Sigma \subset \mathrm{QM}_{z}(Y, E + f^{*}\pi^{*}H) \cap \mathrm{LCP}(X, \Delta + \Gamma + \pi^{*}H). 
    \end{equation*}
    Now $\Gamma + \pi^{*}H$ is a special $\mathbb{Q}$-complement with respect to $g \colon (W, F) \to (X, \Delta)$, since $g^{-1}_{*}(\Gamma + \pi^{*}H) \geq \epsilon G$. Thus, by Lemma \ref{special_cone_give_qdlt_Fano_type_model} and Lemma \ref{ample_model_of_qdltFano_is_Kollar}, we get a qdlt anti-canonical model 
    \begin{equation*}
        \bar{f} \colon (\overline{Y}, \overline{E} + \overline{H}) \to (X, \Delta)
    \end{equation*}
    with $\mathrm{QM}(\overline{Y}, \overline{E} + \overline{H}) = \Sigma$. Thus all components of $\overline{E}$ are centered at $x$, and $\overline{H} = \bar{f}^{-1}_{*}\pi^{*}H = \bar{f}^{*}\pi^{*}H$. So $\bar{f} \colon (\overline{Y}, \overline{E}) \to (X, \Delta)$ is a locally stable family of Koll\'ar models (see the remark after Definition \ref{def:family_of_Kol_model}). It is clear that $v_{\eta} \in \mathrm{QM}(\overline{Y}_{\eta}, \overline{E}_{\eta})$ degenerates to $v_{s} \in \mathrm{QM}(\overline{Y}_{s}, \overline{E}_{s})$. 
\end{proof}

\begin{rem} \label{construct_family_Kollar_model_in_a_fixed_cone}
    In Lemma \ref{construct_family_of_Kol_model}, if we fix a special quasi-monomial simplicial cone $\tau \subset \mathrm{Val}_{X_{\eta}, x_{\eta}}^{\mathrm{qm}}$ such that $v_{\eta} \in \tau$, then we can get a model $f \colon (\overline{Y}, \overline{E}) \to (X, \Delta)$ with $\mathrm{QM}(\overline{Y}_{\eta}, \overline{E}_{\eta}) \subset \tau$. Moreover, when we choose $\sigma$ in the last step, using Diophantine approximation as in the proof of Lemma \ref{specialize_val_computes_constant_lct}, we may assume that 
    \begin{equation*}
        A_{X, \Delta + I_{\lambda}^{c}}(\overline{E}_i) < 1
    \end{equation*}
    for every component $\overline{E}_i$ of $\overline{E}$, for some $c > 0$ and $\lambda > 0$ with $\mathrm{lct}(X, \Delta; I_{\lambda}^c) > 1$. 
\end{rem}

\subsection{Multiple equivariant degenerations}

\begin{defn}
    Let $x \in (X, \Delta)$ be a klt singularity of finite type over a field $k$, with $x \in X(k)$. Suppose
    \begin{equation*}
        f \colon (Y, E = E_1 + \cdots + E_r) \to (X, \Delta)
    \end{equation*}
    is a model of qdlt Fano type at $x$. Suppose $X = \mathop{\mathrm{Spec}}(R)$, and let $\mathcal{R}$ be the extended Rees algebra 
    \begin{equation*}
        \mathcal{R} \coloneqq \bigoplus_{(a_1, \ldots, a_r) \in \mathbb{Z}^r} H^0(Y, \mathscr{O}_Y(a_1E_1 + \cdots + a_rE_r))t_1^{a_1} \cdots t_r^{a_r} \subset R[t_1^{\pm 1}, \ldots, t_r^{\pm 1}]. 
    \end{equation*}
    Then $\mathcal{R}$ is a finitely generated $k$-algebra, and we have $\mathcal{X} = \mathop{\mathrm{Spec}}(\mathcal{R}) \to \mathbb{A}^r_k$ with an isomorphism 
    \begin{equation*}
        \mathcal{X} \times_{\mathbb{A}^r_k} (\mathbb{G}_{\mathrm{m},k})^r \simeq X \times_k (\mathbb{G}_{\mathrm{m},k})^r. 
    \end{equation*}
    Let $\Delta_{\mathcal{X}}$ be the closure of $\Delta \times_{k} (\mathbb{G}_{\mathrm{m},k})^r$. The morphism $(\mathcal{X}, \Delta_{\mathcal{X}}) \to \mathbb{A}^r_k$ is called the \emph{multiple degeneration} of $(X, \Delta)$ induced by $f \colon (Y, E) \to (X, \Delta)$. Moreover, there is an action of $(\mathbb{G}_{\mathrm{m},k})^r$ on $(\mathcal{X}, \Delta_{\mathcal{X}})$ such that the multiple degeneration is equivariant, and the point $x \in X(k)$ induces a section $x \in \mathcal{X}(\mathbb{A}^r_k)$. 

    By \cite[\S 4]{XZ_stable_deg}, $x \in (\mathcal{X}, \Delta_{\mathcal{X}}) \to \mathbb{A}_k^r$ is a locally stable family of klt singularities, and 
    \begin{equation*}
        \mathcal{R}/(t_1, \ldots, t_r) \simeq \mathrm{gr}_v(R)
    \end{equation*}
    for every $v \in \mathrm{QM}^{\circ}(Y, E)$, which induces $(X_v, \Delta_v) \simeq (\mathcal{X}, \Delta_{\mathcal{X}}) \times_{\mathbb{A}^r_k} \{0\}$. 
\end{defn}

\begin{lem} \label{muti_degen_family}
    Let $S$ be a regular connected scheme. Let $\pi \colon (X, \Delta) \to S$ with $x \in X(S)$ be a locally stable family of klt singularities, and $f \colon (Y, E = \sum_{i=1}^{r} E_i) \to (X, \Delta)$ be a locally stable family of Koll\'ar models at $x$. For each $1 \leq i \leq r$, let $(\mathcal{X}^{(i)}, \Delta^{(i)}) \to \mathbb{A}^i_S$ be the multiple degeneration induced by $E_1 + \cdots + E_i$, that is, 
    \begin{equation*}
        \mathcal{R}^{(i)} \coloneqq \bigoplus_{(a_1, \ldots, a_i) \in \mathbb{Z}^i} \pi_{*}f_{*}\mathscr{O}_Y(a_1E_1 + \cdots + a_iE_i)t_1^{a_1} \cdots t_i^{a_i} \subset \pi_{*}\mathscr{O}_X[t_1^{\pm 1}, \ldots, t_i^{\pm 1}], 
    \end{equation*}
    and $\mathcal{X}^{(i)} = \mathop{\mathrm{Spec}_{S}} \mathcal{R}^{(i)} \to \mathbb{A}^i_S$. Then the following hold: 
    \begin{enumerate}[label=\emph{(\arabic*)}, nosep]
        \item $(\mathcal{X}^{(i)}, \Delta^{(i)}) \to \mathbb{A}^i_S$ is a locally stable family of klt pairs. 
        \item There exists a locally stable family of Koll\'ar models $f^{(i)} \colon (\mathcal{Y}^{(i)}, \mathcal{E}^{(i)}) \to (\mathcal{X}^{(i)}, \Delta^{(i)})$ over $\mathbb{A}^i_S$ with 
        \begin{equation*}
            (\mathcal{Y}^{(i)}, \mathcal{E}^{(i)}) \times_{\mathbb{A}^i_S} (\mathbb{G}_{\mathrm{m},S})^i \simeq (Y, E) \times_S (\mathbb{G}_{\mathrm{m},S})^i
        \end{equation*}
        as families of models for $(X, \Delta) \times_S (\mathbb{G}_{\mathrm{m},S})^i$. 
        \item $(\mathcal{X}^{(i+1)}, \Delta^{(i+1)}) \to \mathbb{A}^{i+1}_S$ is the degeneration of $(\mathcal{X}^{(i)}, \Delta^{(i)}) \to \mathbb{A}^i_S$ induced by the component $\mathcal{E}^{(i)}_{i+1}$ corresponding to $E_{i+1}$. 
        \item For every $s \in S$, the base change $(\mathcal{X}^{(i)}_s, \Delta^{(i)}_s) \to \mathbb{A}^i_s$ is the multiple degeneration of $x_s \in (X_s, \Delta_s)$ induced by $f_s \colon (Y_s, E_{1,s} + \cdots + E_{i,s}) \to (X_s, \Delta_s)$. 
    \end{enumerate}
\end{lem}

\begin{proof}
    The case when $S$ is a point is proved in \cite[\S 4.2]{XZ_stable_deg}. In general, we may assume that $S$ is affine. If we have (1) and (2), then (3) holds by \cite[Prop.\ 3.6]{Xu_towards}. We prove (1) and (2) by induction on $i$. By (3), it suffices to prove the case $i = 1$. By shrinking $S$, we may assume that there is a relative effective $\mathbb{Q}$-Cartier $\mathbb{Q}$-divisor $\Gamma$ for $X/S$ such that $\pi \colon (X, \Delta + \Gamma) \to S$ is locally stable, $K_X + \Delta + \Gamma \sim_{S, \mathbb{Q}} 0$, and 
    \begin{equation*}
        \mathrm{LCP}(X_s, \Delta_s + \Gamma_s) = \mathrm{QM}(Y_s, E_s)
    \end{equation*}
    for every $s \in S$. Now 
    \begin{equation*}
        \mathcal{X}^{(1)} = \mathop{\mathrm{Spec}_S}\left( \bigoplus_{m \in \mathbb{Z}} \pi_{*} f_{*}\mathscr{O}_Y(mE_1)t^m \right) \to \mathbb{A}^1_S = \mathop{\mathrm{Spec}_S}(\mathscr{O}_S[t]). 
    \end{equation*}
    By \cite[Lem.\ 3.13(4)]{XZ_stable_deg}, $E_1$ gives a locally stable family of Koll\'ar components for $(X, \Delta + (1-\epsilon)\Gamma)$ at $x$ for all $0 < \epsilon \leq 1$, and the fiber $\mathcal{X}^{(1)}_s \to \mathbb{A}^1_s$ is the degeneration of $X_s$ induced by $E_{1,s}$ for every $s \in S$. Hence by \cite[Lem.\ 3.3]{Xu_towards} and the case over a point, $(\mathcal{X}^{(1)}, \Delta^{(1)} + (1-\epsilon)\Gamma^{(1)}) \to \mathbb{A}^1_S$ is a locally stable family with klt fibers. It follows that $(\mathcal{X}^{(1)}, \Delta^{(1)} + \Gamma^{(1)}) \to \mathbb{A}^1_S$ is a locally stable family, and (4) holds. 
    
    It remains to prove (2) when $i = 1$. Since $(\mathcal{X}^{(1)}, \Delta^{(1)} + \Gamma^{(1)}) \times_{\mathbb{A}^1_S} \mathbb{G}_{\mathrm{m},S} \simeq (X, \Delta + \Gamma) \times_S \mathbb{G}_{\mathrm{m},S}$, we can extract the divisors $E_i \times_S \mathbb{G}_{\mathrm{m},S}$ to get a model
    \begin{equation*}
        f^{(1)} \colon (\mathcal{Y}^{(1)}, \mathcal{E}^{(1)}) \to (\mathcal{X}^{(1)}, \Delta^{(1)})
    \end{equation*}
    by Lemma \ref{extract_div}. Passing to the ample model over $\mathcal{X}^{(1)}$, we may assume $-(K_{\mathcal{Y}^{(1)}} + (f^{(1)})^{-1}_{*}\Delta^{(1)} + \mathcal{E}^{(1)})$ is ample, and there is an isomorphism 
    \begin{equation*}
        (\mathcal{Y}^{(1)}, \mathcal{E}^{(1)}) \times_{\mathbb{A}^1_S} \mathbb{G}_{\mathrm{m},S} \simeq (Y, E) \times_S \mathbb{G}_{\mathrm{m},S}. 
    \end{equation*}
    Hence $\mathcal{E}^{(1)}$ is the closure of $E \times_S \mathbb{G}_{\mathrm{m},S}$, and
    \begin{equation*}
        K_{\mathcal{Y}^{(1)}} + (f^{(1)})^{-1}_{*}(\Delta^{(1)} + \Gamma^{(1)}) + \mathcal{E}^{(1)} = (f^{(1)})^{*}(K_{\mathcal{X}^{(1)}} + \Delta^{(1)} + \Gamma^{(1)}). 
    \end{equation*}
    Suppose $t_0, \ldots, t_d$ is a regular system of parameters at a closed point of $\mathbb{A}^1_S$, and 
    \begin{equation*}
        H = \mathop{\mathrm{div}}(t_0) + \cdots + \mathop{\mathrm{div}}(t_d) \subset \mathcal{X}^{(1)}. 
    \end{equation*}
    Then $(\mathcal{X}^{(1)}, \Delta^{(1)} + \Gamma^{(1)} + H)$ is lc. Hence $(\mathcal{Y}^{(1)}, \mathcal{E}^{(1)} + (f^{(1)})^{-1}_{*}(\Delta^{(1)} + \Gamma^{(1)}) + (f^{(1)})^{*}H)$ is lc, that is, 
    \begin{equation*}
        (\mathcal{Y}^{(1)}, \mathcal{E}^{(1)} + (f^{(1)})^{-1}_{*}(\Delta^{(1)} + \Gamma^{(1)})) \to \mathbb{A}^1_S
    \end{equation*}
    is locally stable. Hence, for every $s \in S$, the fiber $f_s^{(1)} \colon (\mathcal{Y}^{(1)}_s, \mathcal{E}^{(1)}_s) \to (\mathcal{X}^{(1)}_s, \Delta^{(1)}_s)$ is a family of models which extends $(Y_s, E_s) \times_{\kappa(s)} \mathbb{G}_{\mathrm{m},\kappa(s)} \to (X_s, \Delta_s) \times_{\kappa(s)} \mathbb{G}_{\mathrm{m},\kappa(s)}$ such that $-(K_{\mathcal{Y}^{(1)}_s} + (f_s^{(1)})^{-1}_{*}\Delta_s^{(1)} + \mathcal{E}_s^{(1)})$ is ample. By the case over a point, $(Y_s, E_s) \times_{\kappa(s)} \mathbb{G}_{\mathrm{m},\kappa(s)} \to (X_s, \Delta_s) \times_{\kappa(s)} \mathbb{G}_{\mathrm{m},\kappa(s)}$ can also be extended to a locally stable family of Koll\'ar models, which must coincide with $f_s^{(1)} \colon (\mathcal{Y}^{(1)}_s, \mathcal{E}^{(1)}_s) \to (\mathcal{X}^{(1)}_s, \Delta^{(1)}_s)$. So every fiber of $(\mathcal{Y}^{(1)}, \mathcal{E}^{(1)} + (f^{(1)})^{-1}_{*}\Delta^{(1)}) \to \mathbb{A}^1_S$ is qdlt, that is, 
    \begin{equation*}
        f^{(1)} \colon (\mathcal{Y}^{(1)}, \mathcal{E}^{(1)}) \to (\mathcal{X}^{(1)}, \Delta^{(1)})
    \end{equation*}
    is a locally stable family of Koll\'ar models. 
\end{proof}

\begin{lem} \label{graded_ring_of_val_on_family_of_Kol_model}
    Let $S$ be a regular connected scheme. Let $\pi \colon (X, \Delta) \to S$ with $x \in X(S)$ be a locally stable family of klt singularities, and $f \colon (Y, E = \sum_{i=1}^{r} E_i) \to (X, \Delta)$ be a locally stable family of Koll\'ar models at $x$. Suppose $\alpha = (\alpha_i) \in \mathbb{R}_{> 0}^{r}$, and $v_s^{\alpha} \in \mathrm{QM}^{\circ}(Y_s, E_s)$ is the corresponding quasi-monomial valuation on $x_s \in (X_s, \Delta_s)$ for every $s \in S$. For each $\lambda \geq 0$, let 
    \begin{equation*}
        \mathfrak{a}_{\lambda} \coloneqq \sum_{\langle \alpha, m \rangle \geq \lambda} f_{*}\mathscr{O}_Y\left( - \sum_{i=1}^{r} m_iE_i \right) \subset \mathscr{O}_X, 
    \end{equation*}
    where the sum ranges over all $m = (m_i) \in \mathbb{N}^r$ such that $\sum_{i=1}^{r} \alpha_i m_i \geq \lambda$. Then the following hold: 
    \begin{enumerate}[label=\emph{(\arabic*)}, nosep]
        \item $\mathfrak{a}_{\bullet} = \{\mathfrak{a}_{\lambda}\}_{\lambda}$ is an ideal sequence on $X$ centered at $x$, and $\mathfrak{a}_{\lambda}/\mathfrak{a}_{>\lambda}$ is flat over $S$ for all $\lambda \geq 0$. 
        \item Let $X_{\alpha} = \mathop{\mathrm{Spec}_S}(\pi_{*}\mathrm{gr}_{\mathfrak{a}}(\mathscr{O}_X))$, where 
        \begin{equation*}
            \mathrm{gr}_{\mathfrak{a}}(\mathscr{O}_X) \coloneqq \bigoplus_{\lambda \geq 0} \mathfrak{a}_{\lambda}/\mathfrak{a}_{>\lambda}, 
        \end{equation*}
        then the canonical morphism $\pi_{\alpha} \colon X_{\alpha} \to S$ is flat, of finite type, with normal and geometrically integral fibers. There is a canonical action of $\mathbb{T} = (\mathbb{G}_{\mathrm{m},S})^r$ on $X_0$ preserving the grading, whose fixed locus is the image of a section $x_{\alpha} \colon S \to X_{\alpha}$. 
        \item For every $s \in S$, the restriction $\mathfrak{a}_{s, \bullet} = \{\mathfrak{a}_{\lambda}\mathscr{O}_{X_s}\}_{\lambda}$ is the ideal sequence associated with $v_s^{\alpha}$. 
        \item There exists a $\mathbb{T}$-invariant effective $\mathbb{Q}$-divisor $\Delta_{\alpha}$ on $X_{\alpha}$ such that $\pi_{\alpha} \colon (X_{\alpha}, \Delta_{\alpha}) \to S$ is a locally stable family of klt pairs, and $(X_{\alpha,s}, \Delta_{\alpha,s})$ is the degeneration of $x_s \in (X_s, \Delta_s)$ induced by $v_s^{\alpha}$ for every $s \in S$. 
    \end{enumerate}
\end{lem}

\begin{proof}
    (1). It is clear that $\mathfrak{a}_{\bullet} = \{\mathfrak{a}_{\lambda}\}_{\lambda}$ is an ideal sequence on $X$. Write 
    \begin{equation*}
        I_{m} = f_{*}\mathscr{O}_Y\left( - \sum_{i=1}^{r} m_iE_i \right) \subset \mathscr{O}_X
    \end{equation*}
    for $m \in \mathbb{Z}^r$, and $I_{>m} \coloneqq \sum_{m' > m} I_{m'}$, where $m' = (m'_i) > (m_i) = m$ if and only if $m'_i \geq m_i$ for all $i$ and $m' \neq m$. By Lemma \ref{muti_degen_family}, 
    \begin{equation*}
        \mathcal{R} \coloneqq \bigoplus_{m \in \mathbb{Z}^r} I_mt_1^{-m_1} \cdots t_r^{-m_r}
    \end{equation*}
    is flat over $\mathbb{A}_S^r = \mathop{\mathrm{Spec}}_S(\mathscr{O}_S[t_1, \ldots, t_r])$. Hence 
    \begin{equation*}
        \mathcal{R}/(t_1, \ldots, t_r) \simeq \bigoplus_{m \in \mathbb{Z}^r} I_m/I_{>m} 
    \end{equation*}
    is flat over $S$. By Lemma \ref{flat_Rees}, we have
    \begin{equation*}
        \mathfrak{a}_{\lambda}/\mathfrak{a}_{>\lambda} \simeq \bigoplus_{\langle \alpha, m \rangle = \lambda} I_m/I_{>m}, 
    \end{equation*}
    so $\mathfrak{a}_{\lambda}/\mathfrak{a}_{>\lambda}$ is also flat over $S$. 

    (2). By the computation above, $X_{\alpha}$ is isomorphic to $\mathcal{X} \times_{\mathbb{A}^r_S} 0_S$ over $S$, where $0_S \subset \mathbb{A}_S^r$ is the zero section, and $(\mathcal{X}, \Delta_{\mathcal{X}}) \to \mathbb{A}^r_S$ is the multiple degeneration induced by $f \colon (Y, E) \to (X, \Delta)$ in Lemma \ref{muti_degen_family}. Hence $\pi_{\alpha} \colon X_{\alpha} \to S$ is flat, of finite type, with normal and geometrically integral fibers. Since $\mathcal{X} \to \mathbb{A}^r_S$ is $\mathbb{T}$-equivariant, we get an action of $\mathbb{T}$ on $X_{\alpha}$, and 
    \begin{equation*}
        {\pi_{\alpha}}_{*}\mathscr{O}_{X_{\alpha}} \simeq \bigoplus_{m \in \mathbb{Z}^r} \pi_{*}(I_m/I_{>m})
    \end{equation*}
    is the eigenspace decomposition. In particular, the $\mathbb{T}$-action preserves each $\mathfrak{a}_{\lambda}/\mathfrak{a}_{>\lambda}$. Since 
    \begin{equation*}
        \pi_{*}(I_0/I_{>0}) = \pi_{*}(\mathscr{O}_{X}/I_{>0}) \simeq \mathscr{O}_{S}, 
    \end{equation*}
    the fixed locus is given by a section $x_{\alpha} \colon S \to X_{\alpha}$. 

    (3) For $s \in S$, we construct $I_{s,m} \subset \mathscr{O}_{X_s}$ and $\mathcal{R}_s \coloneqq \bigoplus_{m \in \mathbb{Z}^r} I_{s,m}t_1^{-m_1} \cdots t_r^{-m_r}$ as above from the Koll\'ar model $f_s \colon (Y_s, E_s) \to (X_s, \Delta_s)$. Then $\mathcal{R}_s = \mathcal{R} \otimes_{\mathscr{O}_S} \kappa(s)$ as subrings of $\mathscr{O}_{X_s}[t_1^{\pm 1}, \ldots, t_r^{\pm 1}]$, by Lemma \ref{muti_degen_family}. So $I_{s,m} = I_{m}\mathscr{O}_{X_s} = I_{m} \otimes_{\mathscr{O}_S} \kappa(s)$ for all $m \in \mathbb{Z}^r$. By \cite[Cor.\ 4.10]{XZ_stable_deg}, we have  
    \begin{equation*}
        \mathfrak{a}_{\lambda}(v_s^{\alpha}) = \sum_{\langle \alpha, m \rangle \geq \lambda} I_{s,m}. 
    \end{equation*}
    Thus $\mathfrak{a}_{\lambda}(v_s^{\alpha}) = \mathfrak{a}_{\lambda}\mathscr{O}_{X_s}$, and $\mathfrak{a}_{\lambda}(v_s^{\alpha})/\mathfrak{a}_{>\lambda}(v_s^{\alpha}) \simeq (\mathfrak{a}_{\lambda}/\mathfrak{a}_{>\lambda}) \otimes_{\mathscr{O}_S} \kappa(s)$. So we get (3). 

    (4). Recall that $X_{\alpha} \simeq \mathcal{X} \times_{\mathbb{A}^r_S} 0_S$ over $S$. Let $\Delta_{\alpha}$ be the base change of $\Delta_{\mathcal{X}}$, then $\pi_{\alpha} \colon (X_{\alpha}, \Delta_{\alpha}) \to S$ is a locally stable family of klt pairs. As in (3), the fiber $(X_{\alpha, s}, \Delta_{\alpha, s})$ is the central fiber of the multiple degeneration $(\mathcal{X}_s, \Delta_{s}) \to \mathbb{A}_{s}^r$, so it is the degeneration of $x_s \in (X_s, \Delta_s)$ induced by $v_s^{\alpha}$. 
\end{proof}

\begin{rem} \label{initial_ideal}
    Keep the notations in Lemma \ref{graded_ring_of_val_on_family_of_Kol_model} and the proof. We claim that $\Delta_{\alpha}$ is defined by the divisorial part of the initial ideal of $\Delta$. 

    More precisely, write $\Delta = \sum c_DD$, and let $\mathfrak{b} \subset \mathscr{O}_X$ be the ideal of the closed subscheme $D \subset X$ for each component $D$ of $\Delta$. The \emph{initial ideal} of $\mathfrak{b}$ (with respect to $\mathfrak{a}_{\bullet}$) is the ideal 
    \begin{equation*}
        \mathop{\mathrm{in}_{\mathfrak{a}}}(\mathfrak{b}) \coloneqq \bigoplus_{\lambda \geq 0} (\mathfrak{b} \cap \mathfrak{a}_{\lambda})/(\mathfrak{b} \cap \mathfrak{a}_{>\lambda}) \subset \mathrm{gr}_{\mathfrak{a}}(\mathscr{O}_X). 
    \end{equation*}
    Then $\mathop{\mathrm{in}_{\mathfrak{a}}}(\mathfrak{b})$ defines a closed subscheme of $X_{\alpha} = \mathop{\mathrm{Spec}_S}(\pi_{*}\mathrm{gr}_{\mathfrak{a}}(\mathscr{O}_X))$. Let $D_{\alpha} \subset X_{\alpha}$ be its divisorial part (see \cite[{}4.16]{Kol_fam}). We will show that $\Delta_{\alpha} \coloneqq \sum_j c_{D} D_{\alpha}$ is the divisor given by Lemma \ref{graded_ring_of_val_on_family_of_Kol_model}. 

    Recall that $X_{\alpha} = \mathcal{X} \times_{\mathbb{A}_S^r} 0_S$. Let $\Delta_0$ be the pullback of $\Delta_{\mathcal{X}}$ to $X_{\alpha}$, then we need to show $\Delta_0 = \Delta_{\alpha}$. There is an isomorphism 
    \begin{equation*}
        (\mathcal{X}, \Delta_{\mathcal{X}}) \times_{\mathbb{A}_S^r} (\mathbb{G}_{\mathrm{m}, S})^r \simeq (X, \Delta) \times_S (\mathbb{G}_{\mathrm{m}, S})^r, 
    \end{equation*}
    such that $\Delta_{\mathcal{X}}$ is the closure of $\Delta \times_S (\mathbb{G}_{\mathrm{m}, S})^r$. Thus, we have $\Delta_{\mathcal{X}} = \sum c_D \mathcal{D}$, where $\mathcal{D} \subset \mathcal{X}$ is defined by the ideal 
    \begin{equation*}
        \mathcal{R} \cap \mathfrak{b}[t_1^{\pm 1}, \ldots, t_r^{\pm 1}] = \bigoplus_{m \in \mathbb{Z}^r} (I_m \cap \mathfrak{b}) t_1^{-m_1} \cdots t_r^{-m_r} \subset \mathcal{R}. 
    \end{equation*}
    Hence the pullback $D_{0}$ of $\mathcal{D}$ is the divisorial part of the closed subscheme defined by the ideal 
    \begin{equation*}
        \mathop{\mathrm{in}_I}(\mathfrak{b}) = \bigoplus_{m \in \mathbb{Z}^r} (I_m \cap \mathfrak{b})/(I_{>m} \cap \mathfrak{b}) \subset \bigoplus_{m \in \mathbb{Z}^r} I_m/I_{>m}. 
    \end{equation*}
    If $\alpha_1, \ldots \alpha_r$ are linearly independent over $\mathbb{Q}$, then we have $\mathop{\mathrm{in}_I}(\mathfrak{b}) = \mathop{\mathrm{in}_{\mathfrak{a}}}(\mathfrak{b})$ since $\mathfrak{a}_{\lambda}/\mathfrak{a}_{>\lambda} = I_m/I_{>m}$ for the unique $m \in \mathbb{N}^r$ with $\langle \alpha, m \rangle = \lambda$. Hence $D_{0} = D_{\alpha}$, so $\Delta_0 = \Delta_{\alpha}$. 

    In general, suppose the subspace generated by $\alpha_1, \ldots, \alpha_r$ has dimension $q < r$ over $\mathbb{Q}$, and let $\sigma \subset \mathbb{R}^r_{\geq 0}$ be a rational simplicial cone of dimension $q$ with $\alpha \in \sigma$, spanned by $\mu_1, \ldots, \mu_q \in \mathbb{N}^r$. Then $\sigma$ induces a morphism $\mathbb{A}_S^q \to \mathbb{A}_S^r$ given by 
    \begin{equation*}
        \mathscr{O}_S[t_1, \ldots, t_r] \to \mathscr{O}_S[u_1, \ldots, u_q], \quad t_i \mapsto u_1^{\mu_{1,i}} \cdots u_q^{\mu_{q,i}}. 
    \end{equation*}
    Let $(\mathcal{X}', \Delta_{\mathcal{X}'}) = (\mathcal{X}, \Delta_{\mathcal{X}}) \times_{\mathbb{A}_S^r} \mathbb{A}_S^q$, so that $X_{\alpha} = \mathcal{X}' \times_{\mathbb{A}^q_S} 0_S$. More explicitly, 
    \begin{equation*}
        \mathcal{X}' = \mathop{\mathrm{Spec}_S} \bigoplus_{n \in \mathbb{Z}^q} J_n u_1^{-n_1} \cdots u_q^{-n_q}, \quad \text{where} \quad J_n = \sum_{\langle \mu_j, m \rangle \geq n_j} I_m. 
    \end{equation*}
    Note that if we write $\alpha = \sum_{j=1}^q \beta_j\mu_j$ and $\beta = (\beta_j) \in \mathbb{R}_{\geq 0}^q$, then 
    \begin{equation*}
        \mathfrak{a}_{\lambda} = \sum_{\langle \alpha, m \rangle \geq \lambda} I_m = \sum_{\langle \beta, n \rangle \geq \lambda} J_n. 
    \end{equation*}
    Then $(\mathcal{X}', \Delta_{\mathcal{X}'}) \to \mathbb{A}^q_S$ is a locally stable family, in particular every component of $\Delta_{\mathcal{X}'}$ dominates $\mathbb{A}^q_S$. So $\Delta_{\mathcal{X}'}$ is closure of $\Delta \times_S (\mathbb{G}_{\mathrm{m}, S})^q$ under 
    \begin{equation*}
        (\mathcal{X}', \Delta_{\mathcal{X}'}) \times_{\mathbb{A}_S^q} (\mathbb{G}_{\mathrm{m}, S})^q \simeq (X, \Delta) \times_S (\mathbb{G}_{\mathrm{m}, S})^q. 
    \end{equation*}
    By our choices, $\beta_1, \ldots, \beta_q$ are linearly independent over $\mathbb{Q}$, hence the pullback $\Delta_0'$ of $\Delta_{\mathcal{X}'}$ to $X_{\alpha}$ coincides with the divisor $\Delta_{\alpha}$ defined by the initial ideals, as in the paragraphs above. Finally, $\Delta_0' = \Delta_0$ since $\Delta_{\mathcal{X}}$ has well-defined pullbacks (see \cite[Thm-Def.\ 4.3]{Kol_fam}). 
\end{rem}

\section{Families with constant local volume} \label{sec: constant_vol}

\subsection{Degeneration of the minimizer over a DVR} In this subsection, let $S = \mathop{\mathrm{Spec}}(A)$, where $A$ is a DVR, with the generic point $\eta \in S$ and the closed point $s \in S$. 

\begin{thm} \label{construct_Kol_model_of_minimizer_over_DVR}
    Let $\pi \colon (X, \Delta) \to S$ with $x \in X(S)$ be a locally stable family of klt singularities with 
    \begin{equation*}
        \widehat{\mathrm{vol}}(x_{\eta}; X_{\eta}, \Delta_{\eta}) = \widehat{\mathrm{vol}}(x_{s}; X_{s}, \Delta_{s}). 
    \end{equation*}
    Suppose $v^{\mathrm{m}}_{\eta} \in \mathrm{Val}_{X_{\eta}, x_{\eta}}$ and $v^{\mathrm{m}}_s \in \mathrm{Val}_{X_s, x_s}$ are minimizers of the normalized volume for $x_{\eta} \in (X_{\eta}, \Delta_{\eta})$ and $x_s \in (X_s, \Delta_s)$, respectively, scaled such that $A_{X_{\eta}, \Delta_{\eta}}(v^{\mathrm{m}}_{\eta}) = A_{X_s, \Delta_s}(v^{\mathrm{m}}_s)$. Then there exists a locally stable family of Koll\'ar models $f \colon (Y, E) \to (X, \Delta)$ at $x$ over $S$ such that 
    \begin{equation*}
        v^{\mathrm{m}}_{\eta} \in \mathrm{QM}(Y_{\eta}, E_{\eta}) \quad \text{and} \quad v^{\mathrm{m}}_{s} \in \mathrm{QM}(Y_{s}, E_{s}), 
    \end{equation*}
    and they are identified under the canonical isomorphism $\mathrm{QM}(Y_{\eta}, E_{\eta}) \simeq \mathrm{QM}(Y_{s}, E_{s})$ in Definition \ref{quasi-monomial_deg_val}. 
\end{thm}

\begin{proof}
    Let $\mathfrak{a}_{\bullet}(v^{\mathrm{m}}_{\eta}) \subset \mathscr{O}_{X_{\eta}}$ and $\mathfrak{a}_{\bullet}(v^{\mathrm{m}}_s) \subset \mathscr{O}_{X_s}$ be the ideal sequences associated with $v^{\mathrm{m}}_{\eta}$ and $v^{\mathrm{m}}_s$, respectively. We may assume that $A_{X_{\eta}, \Delta_{\eta}}(v^{\mathrm{m}}_{\eta}) = A_{X_s, \Delta_s}(v^{\mathrm{m}}_s) = 1$, so that 
    \begin{equation*}
        \mathrm{lct}(X_{\eta}, \Delta_{\eta}; \mathfrak{a}_{\bullet}(v^{\mathrm{m}}_{\eta})) = \mathrm{lct}(X_s, \Delta_s; \mathfrak{a}_{\bullet}(v^{\mathrm{m}}_s)) = 1, 
    \end{equation*}
    since they are computed by $v^{\mathrm{m}}_{\eta}$ and $v^{\mathrm{m}}_{s}$, respectively, by Lemma \ref{minimizing_ideal_seq}. Thus 
    \begin{equation*}
        \mathrm{e}_{X_{\eta}}(\mathfrak{a}_{\bullet}(v^{\mathrm{m}}_{\eta})) = \widehat{\mathrm{vol}}(x_{\eta}; X_{\eta}, \Delta_{\eta}) = \widehat{\mathrm{vol}}(x_{s}; X_{s}, \Delta_{s}) = \mathrm{e}_{X_s}(\mathfrak{a}_{\bullet}(v^{\mathrm{m}}_s)). 
    \end{equation*}
    Let $I_{\bullet} = j_{*}\mathfrak{a}_{\bullet}(v^{\mathrm{m}}_{\eta}) \cap \mathscr{O}_X$, where $j \colon X_{\eta} \to X$ is the inclusion, and $\mathscr{O}_X \to j_{*}\mathscr{O}_{X_{\eta}}$ is injective since $X_s$ is a Cartier divisor. It is clear that $\mathscr{O}_X/I_{\lambda}$ is supported on the section $\{x\}$ and torsion-free over $S$ for all $\lambda > 0$. Since $S$ is the spectrum of a DVR, we conclude that $\mathscr{O}_X/I_{\lambda}$ is flat over $S$. Hence 
    \begin{equation*}
        \mathrm{e}_{X_{\eta}}(I_{\eta, \bullet}) = \mathrm{e}_{X_{s}}(I_{s, \bullet}), 
    \end{equation*}
    where $I_{\eta, \bullet} = I_{\bullet}\mathscr{O}_{X_{\eta}} = \mathfrak{a}_{\bullet}(v^{\mathrm{m}}_{\eta})$, and $I_{s, \bullet} = I_{\bullet}\mathscr{O}_{X_s}$. Since $\mathrm{lct}$ is lower semi-continuous in a family, we have 
    \begin{equation*}
        \begin{aligned}
            \widehat{\mathrm{vol}}(x_{s}; X_{s}, \Delta_{s}) &\leq \mathrm{lct}(X_s, \Delta_s; I_{s,\bullet})^n \cdot \mathrm{e}_{X_{s}}(I_{s, \bullet}) \\
            &\leq \mathrm{lct}(X_{\eta}, \Delta_{\eta}; I_{\eta,\bullet})^n \cdot \mathrm{e}_{X_{\eta}}(I_{\eta, \bullet}) = \widehat{\mathrm{vol}}(x_{\eta}; X_{\eta}, \Delta_{\eta}), 
        \end{aligned}
    \end{equation*}
    where $n = \dim_{x_{\eta}}(X_{\eta}) = \dim_{x_s}(X_s)$. But $\widehat{\mathrm{vol}}(x_{\eta}; X_{\eta}, \Delta_{\eta}) = \widehat{\mathrm{vol}}(x_{s}; X_{s}, \Delta_{s})$, so all the inequalities above must be equalities. Thus by the uniqueness up to saturation in Lemma \ref{minimizing_ideal_seq}, we have 
    \begin{equation*}
        I_{s, \bullet} \subset \widetilde{I}_{s, \bullet} = \mathfrak{a}_{c\bullet}(v^{\mathrm{m}}_s)
    \end{equation*}
    for some $c > 0$. In fact, $c = 1$ since $\mathrm{e}_{X_s}(\widetilde{I}_{s, \bullet}) = \mathrm{e}_{X_s}(I_{s, \bullet}) = \mathrm{e}_{X_s}(\mathfrak{a}_{\bullet}(v^{\mathrm{m}}_s))$. 

    Next, we will verify conditions of Lemma \ref{construct_family_of_Kol_model}. Assume $D$ is a relative effective $\mathbb{Q}$-Cartier $\mathbb{Q}$-divisor on $X/S$, such that 
    \begin{equation*}
        v^{\mathrm{m}}_{\eta}(D_{\eta}) = v^{\mathrm{m}}_s(D_s). 
    \end{equation*}
    By Lemma \ref{special_val_equiv_conditions}, there exists $\delta > 0$ such that if $\epsilon < \delta/v^{\mathrm{m}}_{\eta}(D_{\eta}) = \delta/v^{\mathrm{m}}_s(D_s)$, then $v^{\mathrm{m}}_{\eta}$ and $v^{\mathrm{m}}_s$ are special valuations for $x_{\eta} \in (X_{\eta}, \Delta_{\eta} + \epsilon D_{\eta})$ and $x_s \in (X_s, \Delta_s + \epsilon D_s)$, respectively. In fact, we can take $\delta = 1/n$; see Remark \ref{estimate_special_threshold}. 
    
    By Lemma \ref{special_val_computes_lct_of_itself_uniquely}, $v^{\mathrm{m}}_{\eta}$ and $v^{\mathrm{m}}_s$ are the unique, up to scaling, real valuations that compute lc thresholds $\mathrm{lct}(X_{\eta}, \Delta_{\eta} + \epsilon D_{\eta}; \mathfrak{a}_{\bullet}(v^{\mathrm{m}}_{\eta}))$ and $\mathrm{lct}(X_{s}, \Delta_{s} + \epsilon D_{s}; \mathfrak{a}_{\bullet}(v^{\mathrm{m}}_s))$, respectively. Then by Lemma \ref{lct_of_saturation}, $v^{\mathrm{m}}_s$ is also the unique, up to scaling, real valuation that compute $\mathrm{lct}(X_{s}, \Delta_{s} + \epsilon D_{s}; I_{s, \bullet})$. Now the theorem follows from Lemma \ref{construct_family_of_Kol_model}. 
\end{proof}

\subsection{General base schemes}

\begin{cor} \label{find_ideal_over_semi-normal}
    Let $S$ be a semi-normal scheme, and $\pi \colon (X, \Delta) \to S$ with $x \in X(S)$ be a locally stable family of klt singularities, such that
    \begin{equation*}
        s \mapsto \widehat{\mathrm{vol}}(x_{s}; X_{s}, \Delta_{s})
    \end{equation*}
    is a locally constant function on $S$. Suppose $v^{\mathrm{m}}_s \in \mathrm{Val}_{X_s, x_s}$ is a minimizer of the normalized volume for $x_s \in (X_s, \Delta_s)$, scaled such that $A_{X_s, \Delta_s}(v^{\mathrm{m}}_s) = 1$ for all $s \in S$. Then there is an ideal sequence $\mathfrak{a}_{\bullet} \subset \mathscr{O}_X$ cosupported at $x(S) \subset X$ such that the following hold: 
    \begin{enumerate}[label=\emph{(\arabic*)}, nosep]
        \item $\mathfrak{a}_{\lambda}/\mathfrak{a}_{>\lambda}$ is flat over $S$ for all $\lambda \geq 0$. 
        \item For every $s \in S$, $\mathfrak{a}_{s, \bullet} \coloneqq \{ \mathfrak{a}_{\lambda} \mathscr{O}_{X_s} \}_{\lambda}$ is the ideal sequence on $X_s$ associated with $v_s^{\mathrm{m}}$. 
    \end{enumerate}
\end{cor}

\begin{proof}
    First assume that $S = \mathop{\mathrm{Spec}}(A)$, where $A$ is a DVR. By Theorem \ref{construct_Kol_model_of_minimizer_over_DVR} and Lemma \ref{graded_ring_of_val_on_family_of_Kol_model}, we have an ideal sequence $\mathfrak{a}_{\bullet} \subset \mathscr{O}_X$ such that $\mathfrak{a}_{\lambda}/\mathfrak{a}_{>\lambda}$ is flat over $S$ for every $\lambda \geq 0$, $(X_0, \Delta_0) \to S$ is a locally stable family of klt singularities, and $\mathfrak{a}_{s, \bullet}$ is the ideal sequence associated with $v_s^{\mathrm{m}}$ for each $s \in S$. Since $v_s^{\mathrm{m}}$ is a minimizer of the normalized volume function, $x_{0,s} \in (X_{0,s}, \Delta_{0,s})$ is a K-semistable log Fano cone singularity by \cite[Thm.\ 1.1]{LX_stability}. 

    In general, we may assume that $S$ is connected. Let $\mathfrak{a}_{\bullet}(v^{\mathrm{m}}_s) \subset \mathscr{O}_{X_s}$ be the ideal sequence associated with $v^{\mathrm{m}}_s$. Note that if $\kappa(s) \subset K$ is field extension, then $\mathfrak{a}_{\bullet}(v^{\mathrm{m}}_s) \otimes_{\kappa(s)} K$ is the ideal sequence associated with the minimizer of the normalized volume function for $x_K \in (X_K, \Delta_K)$. Suppose $A$ is a DVR, and $g \colon T = \mathop{\mathrm{Spec}}(A) \to S$ is a morphism, then the base change $(X_T, \Delta_T) \to T$ with $x_T \in X_T(T)$ is a locally stable family of klt singularities, so that we have an ideal sequence $\mathfrak{a}_{T, \bullet} \subset \mathscr{O}_{X_T}$, whose fibers are the base changes of the corresponding $\mathfrak{a}_{\bullet}(v^{\mathrm{m}}_s)$. Therefore, the function
    \begin{equation*}
        \ell_{\lambda} \colon s \mapsto \mathrm{length}_{\mathscr{O}_{X_s}}\left( \mathscr{O}_{X_s}/\mathfrak{a}_{\lambda}(v^{\mathrm{m}}_s) \right)
    \end{equation*}
    is constant under specialization, hence constant, for all $\lambda \geq 0$. Consider the Hilbert scheme 
    \begin{equation*}
        h_{\lambda} \colon \mathrm{Hilb}_{X/S}^{\ell_{\lambda}} \to S
    \end{equation*}
    parameterizing closed subschemes of $X$ that are finite flat over $S$ with length $\ell_{\lambda}$. Then $h_{\lambda}$ is separated and of finite type. For each $s \in S$, we have a $\kappa(s)$-point $\sigma_{\lambda}(s) = [\mathscr{O}_{X_s}/\mathfrak{a}_{\lambda}(v^{\mathrm{m}}_s)] \in h_{\lambda}^{-1}(s)$. If $A$ is a DVR, and $g \colon T = \mathop{\mathrm{Spec}}(A) \to S$ is a morphism, then by the argument above, there is a lifting $f \colon T \to \mathrm{Hilb}_{X/S}^{\ell_{\lambda}}$ of $g$ such that $f(t) = \sigma_{\lambda}(g(t))$ for every $t \in T$. Thus, by Lemma \ref{valuative_criterion_section}, the set map $s \mapsto \sigma_{\lambda}(s)$ underlies a morphism of schemes $\sigma_{\lambda} \colon S \to \mathrm{Hilb}_{X/S}^{\ell_{\lambda}}$, which is a section of $h_{\lambda}$. Hence we get an ideal sequence $\mathfrak{a}_{\bullet}$ on $X$ such that $\mathfrak{a}_{s, \lambda} = \mathfrak{a}_{\lambda}(v^{\mathrm{m}}_s)$ for all $s \in S$. Each $\mathfrak{a}_{\lambda}/\mathfrak{a}_{>\lambda}$ is a finitely generated $\mathscr{O}_S$-module with constant rank, hence flat over $S$. 
\end{proof}

\begin{cor} \label{construct_Kollar_model_non-local}
    Let $S$ be a regular connected scheme, and $\pi \colon (X, \Delta) \to S$ with $x \in X(S)$ be a locally stable family of klt singularities such that the function 
    \begin{equation*}
        s \mapsto \widehat{\mathrm{vol}}(x_{s}; X_{s}, \Delta_{s})
    \end{equation*}
    is constant on $S$. Suppose $v^{\mathrm{m}}_s \in \mathrm{Val}_{X_s, x_s}$ is a minimizer of the normalized volume for $x_s \in (X_s, \Delta_s)$, scaled such that $A_{X_s, \Delta_s}(v^{\mathrm{m}}_s) = 1$ for all $s \in S$. Then there exists a locally stable family of Koll\'ar models $f \colon (Y, E = \sum_{i=1}^{r} E_i) \to (X, \Delta)$ at $x$ over $S$ such that $v^{\mathrm{m}}_{s} \in \mathrm{QM}(Y_{s}, E_{s})$, and the coordinate 
    \begin{equation*}
        s \mapsto (v^{\mathrm{m}}_{s}(E_{i,s}))_i \in \mathbb{R}_{\geq 0}^r
    \end{equation*}
    is a constant function on $S$. 
\end{cor}

\begin{proof}
    By Corollary \ref{find_ideal_over_semi-normal}, there is an ideal sequence $\mathfrak{a}_{\bullet}$ on $X$ cosupported on $x(S)$ such that $\mathfrak{a}_{\lambda}/\mathfrak{a}_{>\lambda}$ is flat over $S$ for all $\lambda \geq 0$ and $\mathfrak{a}_{s, \bullet} = \mathfrak{a}_{\bullet}(v^{\mathrm{m}}_s)$ for all $s \in S$. If $S$ is local, then we can apply Lemma \ref{construct_family_of_Kol_model} as in the proof of Theorem \ref{construct_Kol_model_of_minimizer_over_DVR}. 
    
    In general, consider the collection $\mathcal{C}$ of open subsets $U \subset S$ such that there exists such a locally stable family of Koll\'ar models over $U$. We will prove by Noetherian induction that $S \in \mathcal{C}$. 

    Assume that $U \in \mathcal{C}$ is a maximal element, and $f_U \colon (Y_U, E_U) \to (X_U, \Delta_U)$ is a locally stable family of Koll\'ar models satisfying the desired conditions. Suppose $U \neq S$, and $t \in S \smallsetminus U$. Let $T = \mathop{\mathrm{Spec}}(\mathscr{O}_{S,t})$, so that $T$ is regular local, with the closed point $t$ and the generic point $\eta$. Then there is a locally stable family of Koll\'ar models 
    \begin{equation*}
        f_T' \colon (Y'_T, E'_T) \to (X_T, \Delta_T), 
    \end{equation*}
    satisfying the desired conditions, $\mathrm{QM}(Y'_{\eta}, E'_{\eta}) \subset \mathrm{QM}(Y_{\eta}, E_{\eta})$ if $U$ is non-empty, and 
    \begin{equation*}
        A_{X, \Delta + \mathfrak{a}_{\lambda}^c}(v'_i) < 1
    \end{equation*}
    for all the divisorial valuations $v'_i$ on $X$ given by the components $E'_{\eta,i}$ of $E'_{\eta}$, and $\mathrm{lct}(X, \Delta; \mathfrak{a}_{\lambda}^c) > 1$; see Remark \ref{construct_family_Kollar_model_in_a_fixed_cone}. Now by Lemma \ref{extract_div}, we have a model 
    \begin{equation*}
        f' \colon (Y', E') \to (X, \Delta)
    \end{equation*}
    such that $E' = \sum_{i} E'_i$ with $v'_i = \mathrm{ord}_{E'_i}$, and $-(K_{Y'} + f'^{-1}_{*}\Delta + E')$ is $f'$-ample. Then its base change to $T$ coincide with the model $f'_T \colon (Y'_T, E'_T) \to (X_T, \Delta_T)$. 

    We claim that the base change $f'_U \colon (Y'_U, E'_U) \to (X_U, \Delta_U)$ is a locally stable family of Koll\'ar models. It suffices to check this at the localization $W = \mathop{\mathrm{Spec}}(\mathscr{O}_{S,u})$ for all $u \in U$. Let $H_W \subset X_W$ be the pullback of an snc divisor on $W$ defined by a regular system of parameters. Then 
    \begin{equation*}
        \mathrm{QM}(Y_W, E_W + f_W^{*}H_W) \simeq \mathrm{QM}(Y_{\eta}, E_{\eta}) \times \mathbb{R}_{\geq 0}^d, 
    \end{equation*}
    where $d = \dim(W)$, is a special cone on $(X_W, \Delta_W)$, since $f_W \colon (Y_W, E_W) \to (X_W, \Delta_W)$ is a locally stable family of Koll\'ar models; see Remark \ref{local_stability_over_regular_local}. Then the subcone $\mathrm{QM}(Y'_{\eta}, E'_{\eta}) \times \mathbb{R}_{\geq 0}^d$ is also special, giving a model 
    \begin{equation*}
        f'_W \colon (Y'_W, E'_W + f_W'^{*}H_W) \to (X_W, \Delta_W)
    \end{equation*}
    such that $(Y'_W, f_W'^{-1}\Delta_W + E'_W + f_W'^{*}H_W)$ is qdlt and $-(K_{Y'_W} + f_W'^{-1}\Delta_W + E'_W + f_W'^{*}H_W)$ is $f'_W$-ample. So $(Y'_W, E'_W)$ is a locally stable family of Koll\'ar models, and coincide with the base change of $(Y', E')$. 

    Thus, we have a model $f' \colon (Y', E') \to (X, \Delta)$ whose base changes to $T$ and $U$ are locally stable family of Koll\'ar models. Then there is an open subset $V \subset S$ with $T \subset V$ and $U \subset V$ such that the base change 
    \begin{equation*}
        (Y'_{V}, E'_V) \to (X_V, \Delta_V) \to V
    \end{equation*}
    is a locally stable family with qdlt fibers by \cite[Thm.\ 4.42]{Kol_fam}. Since $-(K_{Y'} + f'^{-1}_{*}\Delta + E')$ is $f'$-ample, it is a locally stable family of Koll\'ar models. By construction, $v^{\mathrm{m}}_{\eta} \in \mathrm{QM}(Y'_{\eta}, E'_{\eta})$, then $v^{\mathrm{m}}_{s} \in \mathrm{QM}(Y'_{s}, E'_{s})$ for all $s \in V$, with the same coordinates as $v_{\eta}^{\mathrm{m}}$, by Lemma \ref{compare_qm_deg_val}. Since $t \in V \smallsetminus U$, this is a contradiction to the maximality of $U$. Hence there is a locally stable family of Koll\'ar models over $S$ as desired. 
\end{proof}

\begin{thm} \label{K-semistable_degen_over_semi-normal}
    Let $S$ be a semi-normal scheme, and $\pi \colon (X, \Delta) \to S$ with $x \in X(S)$ be a locally stable family of klt singularities, such that
    \begin{equation*}
        s \mapsto \widehat{\mathrm{vol}}(x_{s}; X_{s}, \Delta_{s})
    \end{equation*}
    is a locally constant function on $S$. Suppose $v^{\mathrm{m}}_s \in \mathrm{Val}_{X_s, x_s}$ is a minimizer of the normalized volume for $x_s \in (X_s, \Delta_s)$, scaled such that $A_{X_s, \Delta_s}(v^{\mathrm{m}}_s) = 1$ for all $s \in S$. Then there is an ideal sequence $\mathfrak{a}_{\bullet} \subset \mathscr{O}_X$ cosupported at $x(S) \subset X$ such that the following hold: 
    \begin{enumerate}[label=\emph{(\arabic*)}, nosep]
        \item $\mathfrak{a}_{\lambda}/\mathfrak{a}_{>\lambda}$ is flat over $S$ for all $\lambda \geq 0$. 
        \item For every $s \in S$, $\mathfrak{a}_{s, \bullet} \coloneqq \{ \mathfrak{a}_{\lambda} \mathscr{O}_{X_s} \}_{\lambda}$ is the ideal sequence on $X_s$ associated with $v_s^{\mathrm{m}}$. 
        \item Let $X_0 = \mathop{\mathrm{Spec}_S}(\pi_{*}\mathrm{gr}_{\mathfrak{a}}(\mathscr{O}_X))$, where 
        \begin{equation*}
            \mathrm{gr}_{\mathfrak{a}}(\mathscr{O}_X) \coloneqq \bigoplus_{\lambda \geq 0} \mathfrak{a}_{\lambda}/\mathfrak{a}_{>\lambda}, 
        \end{equation*}
        then the canonical morphism $\pi_0 \colon X_0 \to S$ is flat, of finite type, with normal and geometrically integral fibers. The grading induces an action of a torus $\mathbb{T} \simeq (\mathbb{G}_{\mathrm{m},S})^r$ on $X_0$, whose fixed locus is the image of a section $x_0 \colon S \to X$. 
        \item There exists a $\mathbb{T}$-invariant effective $\mathbb{Q}$-divisor $\Delta_0$ on $X_0$ such that $\pi_0 \colon (X_0, \Delta_0) \to S$ is a locally stable family of K-semistable log Fano cone singularities, and $(X_{0,s}, \Delta_{0,s})$ is the degeneration of $x_s \in (X_s, \Delta_s)$ induced by $v_s^{\mathrm{m}}$ for every $s \in S$. 
    \end{enumerate}
\end{thm}

\begin{proof}
    (1) and (2) are proved in Corollary \ref{find_ideal_over_semi-normal}. Thus the morphism $\pi_0 \colon X_0 \to S$ is flat, with normal and geometrically integral fibers. We still need to prove $\pi_0$ is of finite type, and construct the divisor $\Delta_0$ on $X_0$ as in (4). We first consider the base changes to regular schemes. 
    
    Suppose $g \colon T \to S$ be a morphism of schemes such that $T$ is regular, and 
    \begin{equation*}
        \pi_T \colon (X_T, \Delta_T) = (X, \Delta) \times_S T \to T
    \end{equation*}
    with $x_T \in X_T(T)$ be the base change. Then there is a locally stable family of Koll\'ar models 
    \begin{equation*}
        f \colon (Y, E) \to (X_T, \Delta_T)
    \end{equation*}
    over $T$ by Corollary \ref{construct_Kollar_model_non-local} such that $v_t^{\mathrm{m}} \in \mathrm{QM}(Y_t, E_t)$ for all $t \in T$, with the same coordinates. The ideal sequences on $X_T$ given by Lemma \ref{graded_ring_of_val_on_family_of_Kol_model}, by Corollary \ref{find_ideal_over_semi-normal}, and by the base change of $\mathfrak{a}_{\bullet}$ are all the same, since they coincide on every fiber; denote it by $\mathfrak{a}_{T,\bullet}$, and let 
    \begin{equation*}
        \pi_{0,T} \colon X_{0,T} = \mathop{\mathrm{Spec}_T} \left( (\pi_T)_{*} \mathrm{gr}_{\mathfrak{a}_T}(\mathscr{O}_{X_T}) \right) \to T, 
    \end{equation*}
    so that $X_{0,T} \simeq X_0 \times_S T$. Then $X_{0,T} \to T$ is of finite type by Lemma \ref{graded_ring_of_val_on_family_of_Kol_model}, and there is a divisor $\Delta_{0,T}$ on $X_{0,T}$ induced by $\Delta_T$ such that $\pi_{T} \colon (X_{0,T}, \Delta_{0,T})$ is a locally stable family of log Fano cone singularities, and the fiber $(X_{0,t}, \Delta_{0,t})$ is K-semistable for all $t \in T$ since it is the degeneration induced by $v^{\mathrm{m}}_t$. 

    Now we prove (3): Take $g \colon T \to S$ to be a resolution of singularities, so that $g$ is surjective. Assume $\mathrm{gr}_{\mathfrak{a}_T}(\mathscr{O}_{X_T})$ is generated in degrees $\lambda_1, \ldots, \lambda_N$; in other words, the canonical map 
    \begin{equation*}
        \mathop{\mathrm{Sym}} \bigoplus_{j=1}^{N}\mathfrak{a}_{T,\lambda_j}/\mathfrak{a}_{T,>\lambda_j} \to \mathrm{gr}_{\mathfrak{a}_T}(\mathscr{O}_{X_T})
    \end{equation*}
    is surjective. If $t \in T$ and $s = g(t) \in S$, then 
    \begin{equation*}
        \mathrm{gr}_{\mathfrak{a}_t}(\mathscr{O}_{X_t}) = \mathrm{gr}_{\mathfrak{a}_T}(\mathscr{O}_{X_T}) \otimes_{\mathscr{O}_T} \kappa(t) = \mathrm{gr}_{\mathfrak{a}}(\mathscr{O}_X) \otimes_{\mathscr{O}_S} \kappa(t) = \mathrm{gr}_{\mathfrak{a}_s}(\mathscr{O}_{X_s}) \otimes_{\kappa(s)} \kappa(t)
    \end{equation*}
    is generated in degrees $\lambda_j$, hence $\mathrm{gr}_{\mathfrak{a}_s}(\mathscr{O}_{X_s})$ is also generated in degrees $\lambda_j$. This holds for all $s \in S$ since $g \colon T \to S$ is surjective, hence $\mathrm{gr}_{\mathfrak{a}}(\mathscr{O}_X)$
    is generated in degrees $\lambda_j$ by Nakayama's lemma, as each graded piece $\mathfrak{a}_{\lambda}/\mathfrak{a}_{>\lambda}$ is a coherent $\mathscr{O}_S$-module. That is, $\pi_0 \colon X_0 = \mathop{\mathrm{Spec}_S}(\pi_{*}\mathrm{gr}_{\mathfrak{a}}(\mathscr{O}_X)) \to S$ is of finite type. 

    Let $M \subset \mathbb{R}$ be the subgroup generated by all $\lambda$ with $\mathfrak{a}_{\lambda}/\mathfrak{a}_{>\lambda} \neq 0$; it is the value group of $v^{\mathrm{m}}_s$ for any $s \in S$. Then $M$ is a free abelian group of finite rank $r$, and the torus $\mathbb{T} = \mathop{\mathrm{Spec}_S} \mathscr{O}_S[M] \simeq (\mathbb{G}_{\mathrm{m},S})^r$ acts on $X_0$ over $S$. The fixed locus is given by 
    \begin{equation*}
        \pi_{*}\mathrm{gr}_{\mathfrak{a}}(\mathscr{O}_X) \to \pi_{*}(\mathfrak{a}_0/\mathfrak{a}_{>0}) \simeq \mathscr{O}_S, 
    \end{equation*}
    so it is a section $x_0 \colon S \to X_0$. Since $\mathrm{gr}_{\mathfrak{a}}(\mathscr{O}_X)$ is finitely generated, the elements $\lambda \in M$ with $\mathfrak{a}_{\lambda}/\mathfrak{a}_{>\lambda} \neq 0$ span a rational polyhedral cone $\omega \subset M_{\mathbb{R}}$. Let $\xi_0 \in N_{\mathbb{R}} = \mathrm{Hom}_{\mathbb{Z}}(M, \mathbb{R})$ be the inclusion $M \hookrightarrow \mathbb{R}$, then $\xi_0$ is positive on $\omega \smallsetminus \{0\}$, so there exists $\xi \in N = \mathrm{Hom}_{\mathbb{Z}}(M, \mathbb{Z})$ that is positive on $\omega \smallsetminus \{0\}$. 
    
    To prove (4), we consider $T = \mathop{\mathrm{Spec}}(A)$ where $A$ is a DVR, and will descend the divisors $\Delta_{0,T}$ to $X_0$ as in Corollary \ref{find_ideal_over_semi-normal}, using Lemma \ref{valuative_criterion_section} and the representability of relative Mumford divisors (see \cite[Thm.\ 4.76]{Kol_fam}). We need a projective family to have the representability. Choose $\xi \in N$ as above, and let 
    \begin{equation*}
        \overline{X}_0 \coloneqq \mathop{\mathrm{Proj}_S} \pi_{*}\mathrm{gr}_{\mathfrak{a}}(\mathscr{O}_X)[u]
    \end{equation*}
    where $\pi_{*}\mathrm{gr}_{\mathfrak{a}}(\mathscr{O}_X)[u]$ is $\mathbb{N}$-graded with $u$ in degree $1$ and $\pi_{*}(\mathfrak{a}_{\lambda}/\mathfrak{a}_{>\lambda})$ in degree $\langle \xi, \lambda \rangle$. Then the formation of $\overline{X}_0$ commutes with base change. Thus, by Lemma \ref{compactify_log_Fano_cone}, 
    \begin{equation*}
        \bar{\pi}_{0,T} \colon (\overline{X}_{0,T} = \overline{X}_0 \times_S T, \overline{\Delta}_{0,T} + V_T) \to T
    \end{equation*}
    is a locally stable family of projective plt pairs, where $X_{0, T} \subset \overline{X}_{0,T}$ is an open subscheme, $V_T$ is the complement, and $\overline{\Delta}_{0,T}$ is the closure of $\Delta_{0,T}$. If we write $\Delta = \sum c_D D$, then $\overline{\Delta}_{0,T} = \sum c_D \overline{D}_{0,T}$, where $\overline{D}_{0,T} \subset \overline{X}_{0,T}$ is the closure of the divisorial part $D_{0,T} \subset X_{0,T}$ of the closed subscheme defined by initial ideal of $D_T \subset X_T$; see Remark \ref{initial_ideal}. Then each $\overline{D}_{0,T}$ is a relative Mumford divisor on $\overline{X}_{0,T}/T$, giving a morphism 
    \begin{equation*}
        \tilde{g}_D \colon T \to \mathrm{MDiv}
    \end{equation*}
    that lifts $g \colon T \to S$, where $\mathrm{MDiv}$ is the $S$-scheme representing relative Mumford divisors on $\overline{X}_0/S$ of a suitable degree in \cite[Thm.\ 4.76]{Kol_fam}. On the underlying sets, the map $\tilde{g}_D$ lifts $s \in S$ to the point that represents $\overline{D}_{0,s} \subset \overline{X}_{0,s}$. Thus by Lemma \ref{valuative_criterion_section}, there is a relative Mumford divisors $\overline{D}_0$ on $\overline{X}_0/S$ whose base change is $\overline{D}_{T,0}$ for any $g \colon T \to S$ as above. Let $\overline{\Delta}_0 = \sum c_D \overline{D}_0$, and $\Delta_0 = \overline{\Delta}_0|_{X_0}$. Then 
    \begin{equation*}
        \pi_0 \colon (X_0, \Delta_0) \to S
    \end{equation*}
    is a locally stable family of log Fano cone singularities since its base change to any DVR is so, by \cite[Def-Thm.\ 4.7]{Kol_fam}. The fiber $(X_{0,s}, \Delta_{0,s})$ is K-semistable (with the polarization $\xi_0 \colon M \hookrightarrow \mathbb{R}$) since it is the degeneration of $(X, \Delta)$ induced by $v^{\mathrm{m}}_s$. 
\end{proof}

\begin{cor} \label{representability}
    Let $S$ be a reduced scheme, and $\pi \colon (X, \Delta) \to S$ with $x \in X(S)$ be a locally stable family of klt singularities. Then there exists a locally closed stratification $\bigsqcup_{i} S_{i} \to S$ satisfying the following: if $g \colon T \to S$ is a morphism where $T$ is a semi-normal scheme, then the base change $(X_T, \Delta_T) \to T$ with $x_T \in X_T(T)$ admits a degeneration to a locally stable family of K-semistable log Fano cone singularities if and only if $g$ factors through some $S^{\mathit{sn}}_i \to S$, where $S^{\mathit{sn}}_i$ is the semi-normalization of $S_i$. 
\end{cor}

Here we say a locally stable family of klt singularities admits a degeneration to a locally stable family of K-semistable log Fano cone singularities if the conclusions of Corollary \ref{K-semistable_degen_over_semi-normal} hold. 

\begin{proof}
    Suppose $\pi \colon (X, \Delta) \to S$ with $x \in X(S)$ degenerates to a locally stable family $(X_{0}, \Delta_{0}) \to S$ of K-semistable log Fano cone singularities, then 
    \begin{equation*}
        \widehat{\mathrm{vol}}(x_s; X_s, \Delta_s) = \widehat{\mathrm{vol}}(x_{0,s}; X_{0,s}, \Delta_{0,s})
    \end{equation*}
    for all $s \in S$ by \cite[Lem.\ 4.10]{LX_stability}, where $x_0 \in X_0(S)$ is the vertex. Hence $s \mapsto \widehat{\mathrm{vol}}(x_{s}; X_{s}, \Delta_{s})$ is locally constant on $S$ by Lemma \ref{family_of_log_Fano_cone_constant_vol}. The converse holds when $S$ is semi-normal by Corollary \ref{K-semistable_degen_over_semi-normal}. 
    
    In general, the function 
    \begin{equation*}
        s \mapsto \widehat{\mathrm{vol}}(x_{s}; X_{s}, \Delta_{s})
    \end{equation*}
    is constructible on $S$ by \cite[Thm.\ 1.3]{Xu_qm}, and lower semi-continuous on $S$ by \cite[Thm.\ 1.1]{BL_lsc}, and its formation commutes with arbitrary base change since the local volume is preserved by field extensions. Then it suffices to take each $S_i \subset S$ to be a level set of this function, which is a locally closed subscheme with the reduced scheme structure. 
\end{proof}

\subsection{Example: Unibranch plane curves} Let $\Bbbk$ be an algebraically closed field of characteristic zero. Let $C \subset \mathbb{A}^2 = \mathop{\mathrm{Spec}} \Bbbk[x, y]$ be a reduced plane curve, i.e., a reduced closed subscheme of pure dimension one. Assume $C$ is unibranch at the origin $0 \in \mathbb{A}^2$ (in particular, $0 \in C$). We are interested in the local volume 
\begin{equation*}
    \widehat{\mathrm{vol}}(0; \mathbb{A}^2, \lambda C)
\end{equation*}
for coefficients $0 \leq \lambda < \mathrm{lct}_0(\mathbb{A}^2; C)$. 

By \cite[Thm.\ 1.1]{LX_stability}, we need to find a quasi-monomial valuation $v \in \mathrm{Val}^{\mathrm{qm}}_{\mathbb{A}^2, 0}$ inducing a degeneration of $0 \in (\mathbb{A}^2, \lambda C)$ to a K-semistable log Fano cone singularity. We first consider the problem for analytic germs of curves. 

\begin{exmp} \label{example_analytic_curve_germ}
     Let $\widehat{C} \subset \widehat{\mathbb{A}}^2 = \mathop{\mathrm{Spec}} \Bbbk[\![x, y]\!]$ be an integral closed subscheme of dimension $1$. As in \cite{Zar_curve}, we may assume that $\widehat{C}$ is given by a parametrization 
     \begin{equation*}
         x = t^a, \quad y = \varphi(t) = t^b + \sum_{i > b} c_i t^i \in \Bbbk[\![t]\!]
     \end{equation*}
     such that $b > a > 1$ and $a \nmid b$. More precisely, $\widehat{C}$ is defined by the kernel of the map 
     \begin{equation*}
         \Bbbk[\![x,y]\!] \to \Bbbk[\![t]\!], \quad (x, y) \mapsto (t^a, \varphi(t)). 
     \end{equation*}
     We claim that 
     \begin{equation*}
         \mathrm{lct}(\widehat{C}) \coloneqq \mathrm{lct}(\widehat{\mathbb{A}}^2; \widehat{C}) = \frac{1}{a} + \frac{1}{b}. 
     \end{equation*}
     The function $f \in \Bbbk[\![x,y]\!]$ defining $\widehat{C}$ can be written as $f = \prod_{j=0}^{a-1} (y - \varphi(\zeta^jx^{1/a}))$, where $\zeta$ is a primitive $a$-th root of unity. Then from the Newton polygon of $f$ we get $\mathrm{lct}(\widehat{C}) \leq \frac{1}{a} + \frac{1}{b}$, by \cite[Thm.\ 6.40]{KSC_rational}. In the following, we will show that $f$ degenerates to $f_0 = (y^{a/d} - x^{b/d})^d$, where $d = \mathrm{gcd}(a, b)$. Hence, if $\widehat{C}_0$ is defined by $y^{a/d} - x^{b/d}$, then 
     \begin{equation*}
         \mathrm{lct}(\widehat{C}) \geq \mathrm{lct}(d \widehat{C}_0) = \frac{1}{d}\left( \frac{1}{a/d} + \frac{1}{b/d} \right) = \frac{1}{a} + \frac{1}{b}
     \end{equation*}
     by the lower semi-continuity of lct and \cite[Prop.\ 6.39]{KSC_rational}. 

     Let $v_{\xi} \in \mathrm{Val}_{\widehat{\mathbb{A}}^2}$ denote the monomial valuation given by weight $\xi = (\mu, \nu) \in \mathbb{R}^2_{>0}$. Then 
     \begin{equation*}
         \mathrm{gr}_{v_{\alpha}}(\Bbbk[\![x,y]\!]) = \Bbbk[X, Y], 
     \end{equation*}
     where we denote the initial terms of $x, y \in \Bbbk[\![x,y]\!]$ by $X, Y \in \Bbbk[X, Y]$, respectively, and the ring $\Bbbk[X, Y]$ is graded with $\deg(X) = \mu$ and $\deg(Y) = \nu$. We will show that $v_{\xi}$ induce K-semistable degenerations of $(\widehat{\mathbb{A}}^2, \lambda \widehat{C})$, for an appropriate weight $\xi$ depending on the coefficient $\lambda$. There are two cases: 

     (1). Let $d = \mathrm{gcd}(a, b)$, $a_0 = a/d$, $b_0 = b/d$, and $\xi = (a_0, b_0)$. Then the initial term of $f$ is 
     \begin{equation*}
         f_0 = \prod_{j=0}^{a-1} (Y - \zeta^{bj}X^{b/a}) = (Y^{a_0} - X^{b_0})^d
     \end{equation*}
     where $\zeta \in \Bbbk$ is a primitive $a$-th root of unity. Let $C_0 \subset \mathbb{A}^2$ be the curve defined by $Y^{a_0} - X^{b_0}$. Then $v_{\xi}$ induces a $\mathbb{G}_{\mathrm{m}}$-equivariant degeneration of $(\widehat{\mathbb{A}}^2, \lambda \widehat{C})$ to $(\mathbb{A}^2, \lambda d C_0)$, where $\mathbb{G}_{\mathrm{m}}$ acts on $\mathbb{A}^2$ with weight $\xi$, such that $C_0 \subset \mathbb{A}^2$ is invariant. In this case, the polarized log Fano cone singularity $(\mathbb{A}^2, \lambda d C_0; \xi)$ is an (orbifold) affine cone over the log Fano pair 
     \begin{equation*}
         (E, \Delta_E) = \left( \mathbb{P}^1, \left( 1 - \frac{1}{a_0} \right) \{0\} + \left( 1 - \frac{1}{b_0} \right) \{\infty\} + \lambda d \{1\} \right)
     \end{equation*}
     see \cite[Ex.\ 1.7]{Li_tangent}. By Fujita--Li's valuative criterion for K-stability (see \cite[Thm.\ 4.13]{Xu_K-Stability_Book}), it is easy to see that $(E, \Delta_E)$ is K-semistable if and only if $\lambda d \geq \frac{1}{a_0} - \frac{1}{b_0}$, that is, 
     \begin{equation*}
         \lambda \geq \frac{1}{a} - \frac{1}{b}. 
     \end{equation*}
     Thus $(\mathbb{A}^2, \lambda d C_0; \xi)$ is K-semistable if and only if $\lambda \geq \frac{1}{a} - \frac{1}{b}$. We can also compute 
     \begin{equation*}
         \widehat{\mathrm{vol}}_{\widehat{\mathbb{A}}^2, \lambda \widehat{C}}(v_{\xi}) = ab\left( \frac{1}{a} + \frac{1}{b} - \lambda \right)^2. 
     \end{equation*}

     (2). Suppose $0 \leq \lambda < \frac{1}{a} - \frac{1}{b}$. Let $\xi = (\mu, 1)$ where 
     \begin{equation*}
         \mu = 1 - \lambda a > a/b
     \end{equation*}
     Then $v_{\xi}$ induces a degeneration of $(\widehat{\mathbb{A}}^2, \lambda \widehat{C})$ to $(\mathbb{A}^2, \lambda aL; \xi)$, where $L \subset \mathbb{A}^2$ is the line $Y = 0$. The pair $(\mathbb{A}^2, \lambda aL)$ is toric, and it is easy to see the toric valuation of weight $\xi = (\mu, 1)$ minimizes the normalized volume among all toric valuations. Hence $(\mathbb{A}^2, \lambda aL; \xi)$ is K-semistable by \cite[Thm.\ 3.5, Prop.\ 4.21]{LX_stability}. In this case, we have 
     \begin{equation*}
         \widehat{\mathrm{vol}}_{\widehat{\mathbb{A}}^2, \lambda \widehat{C}}(v_{\xi}) = 4(1 - \lambda a). 
     \end{equation*}
\end{exmp}

Note that the above results only depend on the pair of integers $(a, b)$. They are the first two terms of the \emph{Puiseux characteristic} of the germ $\widehat{C} \subset \widehat{\mathbb{A}}^2$ (see \cite[Def.\ II.3.2]{Zar_curve}), and are independent of the choice of analytic local coordinates. 

\begin{lem} \label{local_vol_of_unibranch_curve}
    Let $C \subset \mathbb{A}^2$ be a reduced plane curve through the origin $0 \in \mathbb{A}^2$. Suppose $C$ is unibranch and singular at $0$, with the Puiseux characteristic $(a, b, \ldots)$. Then $\mathrm{lct}_0(\mathbb{A}^2; C) = \frac{1}{a} + \frac{1}{b}$, and the following hold: 
    \begin{enumerate}[label=\emph{(\arabic*)}, nosep]
        \item If $\frac{1}{a} - \frac{1}{b} \leq \lambda < \frac{1}{a} + \frac{1}{b}$, then 
        \begin{equation*}
            \widehat{\mathrm{vol}}(0; \mathbb{A}^2, \lambda C) = ab\left( \frac{1}{a} + \frac{1}{b} - \lambda \right)^2
        \end{equation*}
        and the K-semistable degeneration of $0 \in (\mathbb{A}^2, \lambda C)$ is $(\mathbb{A}^2, \lambda d C_0; \xi)$, where $d = \mathrm{gcd}(a, b)$ and $C_0$ is the curve $Y^{a/d} - X^{b/d} = 0$, with the polarization $\xi = (a/d, b/d)$; 
        \item if $0 \leq \lambda < \frac{1}{a} - \frac{1}{b}$, then 
        \begin{equation*}
            \widehat{\mathrm{vol}}(0; \mathbb{A}^2, \lambda C) = 4(1 - \lambda a)
        \end{equation*}
        and the K-semistable degeneration of $0 \in (\mathbb{A}^2, \lambda C)$ is $(\mathbb{A}^2, \lambda a L; \xi)$, where $L$ is the line $Y = 0$, with the polarization $\xi = (1 - \lambda a, 1)$
    \end{enumerate}
    If $C$ is smooth at $0$, then \emph{(2)} holds for all $0 \leq \lambda < 1$ with $a = 1$. 
\end{lem}

\begin{proof}
    Let $\widehat{C} \subset \widehat{\mathbb{A}}^2 = \mathop{\mathrm{Spec}} \Bbbk[\![x, y]\!]$ denote the completion at $0$. By Example \ref{example_analytic_curve_germ}, we have 
    \begin{equation*}
        \mathrm{lct}_0(\mathbb{A}^2, C) = \mathrm{lct}(\widehat{\mathbb{A}}^2, \widehat{C}) = \frac{1}{a} + \frac{1}{b}, 
    \end{equation*}
    since a pair is klt at a point if and only if its completion at that point is klt (see \cite[\S 2.16]{Kollar_singularity}). There is a quasi-monomial valuation $\hat{v} \in \mathrm{Val}_{\widehat{\mathbb{A}}^2, 0}$ inducing a degeneration of $0 \in (\widehat{\mathbb{A}}^2, \lambda \widehat{C})$ to a K-semistable log Fano cone singularity $(\mathbb{A}^2, \Delta_0; \xi)$ as in Example \ref{example_analytic_curve_germ}. Let $v$ denote the restriction of $\hat{v}$ on $\Bbbk[x, y] \subset \Bbbk[\![x, y]\!]$. Then $v$ is a quasi-monomial valuation centered at $0 \in \mathbb{A}^2$ by \cite[Lem.\ 3.10]{JM}, and the degeneration of $0 \in (\mathbb{A}^2, \lambda C)$ induced by $v$ is also $(\mathbb{A}^2, \Delta_0; \xi)$. 
    
    Hence $v$ is the minimizer of the normalized volume of $0 \in (\mathbb{A}^2, \lambda C)$ by \cite[Thm.\ 1.3]{LX_stability}. That is, 
    \begin{equation*}
        \widehat{\mathrm{vol}}(0; \mathbb{A}^2, \lambda C) = \widehat{\mathrm{vol}}_{\mathbb{A}^2, \lambda C}(v) = \widehat{\mathrm{vol}}_{\widehat{\mathbb{A}}^2, \lambda \widehat{C}}(\hat{v}) = \left\{ \begin{array}{ll}
            ab\left( \frac{1}{a} + \frac{1}{b} - \lambda \right)^2 & \text{if}\ \frac{1}{a} - \frac{1}{b} \leq \lambda < \frac{1}{a} + \frac{1}{b} \\
            4(1 - \lambda a) & \text{if}\ 0 \leq \lambda < \frac{1}{a} - \frac{1}{b}
        \end{array} \right.
    \end{equation*}
    where the two cases $\frac{1}{a} - \frac{1}{b} \leq \lambda < \frac{1}{a} + \frac{1}{b}$ and $0 \leq \lambda < \frac{1}{a} - \frac{1}{b}$ correspond to the two cases in Example \ref{example_analytic_curve_germ}, and the explicit form of the K-semistable degeneration $(\mathbb{A}^2, \Delta_0; \xi)$ is given there. 
    
    If $C$ is smooth at $0$, then we can choose analytic local coordinates such that $\widehat{C}$ is the line $y = 0$. Thus the second case gives the K-semistable degeneration for all $0 \leq \lambda < 1$ with $a = 1$. 
\end{proof}

We say two plane curves $C, C' \subset \mathbb{A}^2$ that are unibranch at $0 \in \mathbb{A}^2$ are \emph{equisingular at $0$} if they have the same Puiseux characteristics. See \cite{Zar_equisingular_I} and \cite{Zar_equisingular_III} for other equivalent definitions. In particular, if $\Bbbk = \mathbb{C}$, then $C$ and $C'$ are equisingular at $0$ if and only if they have the same topological type (see \cite[Def.\ I.1.2]{Zar_curve}). 

\begin{cor} \label{family_of_equisingular_curve_branch}
    Let $S$ be a reduced connected scheme of finite type over an algebraically closed field $\Bbbk$ of characteristic $0$. Let $C \subset \mathbb{A}^2_S$ be a closed subscheme that is flat of pure relative dimension $1$ over $S$. Let $0_S \subset \mathbb{A}^2_S$ denote the origin. Suppose $C_s \subset \mathbb{A}^2_{\Bbbk}$ is a reduced plane curve that is unibranch at $0$ for all $s \in S(\Bbbk)$, and $C_s, C_{s'}$ are equisingular at $0$ for all $s, s' \in S(\Bbbk)$. Then the following hold: 
    \begin{enumerate}[label=\emph{(\arabic*)}, nosep]
        \item The function $s \mapsto \mathrm{lct}_{0_s}(\mathbb{A}^2_s; C_s)$ is constant on $S$. 
        \item Suppose $0 \leq \lambda < \mathrm{lct}_{0_s}(\mathbb{A}^2_s; C_s)$ for all $s \in S$. Then $(\mathbb{A}^2_S, \lambda C) \to S$ is a locally stable family of klt singularities at $0_S$, and the function 
        \begin{equation*}
            s \mapsto \widehat{\mathrm{vol}}(0_s; \mathbb{A}^2_s, \lambda C_s)
        \end{equation*}
        is constant on $S$. 
        \item Assume $S$ is semi-normal. Then $(\mathbb{A}^2_S, \lambda C) \to S$ admits a degeneration to a locally stable family of K-semistable log Fano cone singularities $(X_0, \Delta_0) \to S$. Moreover, all fibers of $(X_0, \Delta_0) \to S$ over $\Bbbk$-points of $S$ are isomorphic. 
    \end{enumerate}
\end{cor}

\begin{proof}
    (1). Since all $C_s$ are equisingular at $0$ for $s \in S(\Bbbk)$, they have the same Puiseux characteristics $(a, b, \ldots)$. By Lemma \ref{local_vol_of_unibranch_curve}, the function $s \mapsto \mathrm{lct}_{0_s}(\mathbb{A}^2_s; C_s)$ is constant for $s \in S(\Bbbk)$, of value $\frac{1}{a} + \frac{1}{b}$. Since $S$ is of finite type over $\Bbbk$, the set of $\Bbbk$-points $S(\Bbbk)$ is dense in $S$, and every non-empty closed subset of $S$ contains a $\Bbbk$-point. Thus, by the lower semi-continuity of lct, $s \mapsto \mathrm{lct}_{0_s}(\mathbb{A}^2_s; C_s)$ is constant on $S$. 

    (2). Let $0 \leq \lambda < \frac{1}{a} + \frac{1}{b}$. Note that $\mathbb{A}^2_S \to S$ is a smooth family, and $C \subset \mathbb{A}^2_S$ is a relative Cartier divisor over $S$. Thus $(\mathbb{A}^2_S, \lambda C) \to S$ is a locally stable family of klt singularities at $0_S$ since every fiber $(\mathbb{A}^2_s, \lambda C_s)$ is klt at $0_s$ by (1), and $K_{\mathbb{A}^2_S/S} + \lambda C$ is $\mathbb{R}$-Cartier (see \cite[Def.\ 4.7]{Kol_fam}). The function 
    \begin{equation*}
        s \mapsto \widehat{\mathrm{vol}}(0_s; \mathbb{A}^2_s, \lambda C_s)
    \end{equation*}
    is lower semi-continuous by \cite[Thm.\ 1.1]{BL_lsc}, and is constant on $S(\Bbbk)$ with the value given by Lemma \ref{local_vol_of_unibranch_curve}. Hence it is a constant function on $S$. 

    (3). The existence of the K-semistable degeneration $(X_0, \Delta_0) \to S$ follows from Corollary \ref{K-semistable_degen_over_semi-normal}, since the family $(\mathbb{A}^2_S, \lambda C) \to S$ has constant local volume at $0_S$. Moreover, the fiber $(X_{0, s}, \Delta_{0, s})$ at $s \in S(\Bbbk)$ is the K-semistable degeneration of $0 \in (\mathbb{A}^2_{\Bbbk}, \lambda C_s)$. They are all isomorphic by Lemma \ref{local_vol_of_unibranch_curve}. 
\end{proof}

\appendix

\section{A valuative criterion for sections}

\begin{lem} \label{valuative_criterion_section}
    Let $S$ be a locally Noetherian semi-normal scheme, and $\pi \colon Z \to S$ be a morphism that is separated and of finite type. Suppose $|\sigma| \colon |S| \to |Z|$ is a map between underlying sets that is a section of the map $|\pi| \colon |Z| \to |S|$, such that for every morphism $g \colon T = \mathop{\mathrm{Spec}}(A) \to S$, where $A$ is a DVR, there exists a morphism $f \colon T \to Z$ such that $\pi \circ f = g$ and $|f| = |\sigma| \circ |g|$. Then there is a section $\sigma \colon S \to Z$ of $\pi$ whose underlying map between sets is $|\sigma|$. 

    Moreover, if $S$ is locally essentially of finite type over a field $k$, then it suffices to consider $A$ that are also essentially of finite type over $k$. 
\end{lem}

\begin{proof}
    We may assume $S$ is affine. If $T$ is the spectrum of a DVR, we write $\zeta \in T$ for the generic point, and $t \in T$ for the closed point. For $s \in S$ with residue field $\kappa = \kappa(s)$, let $A$ be a DVR with residue field $\kappa$, say, $A = \kappa[\varpi]_{(\varpi)}$, and consider the map 
    \begin{equation*}
        T = \mathop{\mathrm{Spec}}(A) \to \mathop{\mathrm{Spec}}(\kappa) \xrightarrow{s} S. 
    \end{equation*}
    Then we conclude that $|\sigma|(s) \in |Z_s|$ is a $\kappa$-point of $Z_s$. In particular, it is a closed point of $|Z_s|$. 

    Let $|S'| = |\sigma|(|S|) \subset |Z|$ be the image of $|\sigma|$. Note that $|\sigma| \colon |S| \to |S'|$ is a bijection, and the restriction of $|\pi|$ is the inverse. Suppose $x' \in |S'|$ specializes to $y' \in |Z|$, and $y' \neq x'$. Let $x = \pi(x')$ and $y = \pi(y')$. Since $x' = |\sigma|(x)$ is a closed point in $Z_x$, we have $y \neq x$. Then there exists a DVR $A$ with fraction field $K = \kappa(x)$ and a morphism $f' \colon T = \mathop{\mathrm{Spec}}(A) \to Z$ such that $f'(\zeta) = x'$, and $f'(t) = y'$ (see \cite[Tag 00PH]{stacks-project}), and we may assume that $A$ is essentially of finite type over $k$ if $S$ is so. By assumption, there exists a morphism $f \colon T \to Z$ such that $\pi \circ f = \pi \circ f'$, $f(\zeta) = |\sigma|(x) = x'$, and $f(t) = |\sigma|(y)$. Since $T$ has function field $K = \kappa(x)$, we must have $f|_{\mathop{\mathrm{Spec}}(K)} = f'|_{\mathop{\mathrm{Spec}}(K)}$. Since $\pi$ is separated, we conclude that $f = f'$ by the valuative criterion. Thus $y' = f(t) = |\sigma|(y) \in |S'|$. So $|S'|$ is stable under specialization. 
    
    Let $\eta_1, \ldots, \eta_n \in S$ be the generic points, and $\eta_i' = |\sigma|(\eta_i) \in |S'|$. Then $\bigcup_{i=1}^{n} \overline{\{\eta_i'\}} \subset |S'|$. Conversely, if $x' \in |S'|$, then $x = \pi(x') \in S$ generalizes to some $\eta_i$. Let $A$ be a DVR and $g \colon T = \mathop{\mathrm{Spec}}(A) \to S$ be a morphism such that $g(\zeta) = \eta_i$ and $g(t) = x$. By assumption, there is a lifting $f \colon T \to Z$ of $g$ such that $f(\zeta) = |\sigma|(\eta_i) = \eta_i'$ and $f(t) = |\sigma|(x) = x'$. Hence $x' \in \overline{\{\eta_i'\}}$. Thus we conclude that $\bigcup_{i=1}^{n} \overline{\{\eta_i'\}} = |S'|$. Hence $|S'| \subset |Z|$ is closed. Let $S' \subset Z$ be the corresponding reduced closed subscheme. 

    By the valuative criterion, $\pi|_{S'} \colon S' \to S$ is proper. Moreover, it is bijective and induces isomorphisms on residue fields. Since $S$ is semi-normal, we conclude that $\pi|_{S'} \colon S' \to S$ is an isomorphism. Therefore we get a section $\sigma \colon S \simeq S' \hookrightarrow Z$. 
\end{proof}

\printbibliography

\end{document}